\documentclass{amsart}
\usepackage[latin1]{inputenc}
\usepackage{epsf,graphicx}
\usepackage{amssymb,latexsym}
\usepackage[T1]{fontenc}
\usepackage{psfrag}
\usepackage{subfigure}

\numberwithin{equation}{section}

\def\eps{{\epsilon}}

\def\C{{\mathbb C}}
\def\A{{\mathbb A}}
\def\D{{\mathbb D}}

\def\R{{\mathbb R}}

\def\H{{\mathbb H}}
\def\CP{{\mathbb C\mathbb P}}

\def\ov{\overline}
\def\pl{\parallel}

\theoremstyle{plain}
\newtheorem{lemma}{Lemma}[section]
\newtheorem{proposition}[lemma]{Proposition}
\newtheorem{example}[lemma]{Example}
\newtheorem{remark}[lemma]{Remark}

\newtheorem{theorem}[lemma]{Theorem}
\newtheorem{corollary}[lemma]{Corollary}
\newtheorem{compatibility}[lemma]{Compatibility Condition}

\theoremstyle{definition}
\newtheorem{definition}[lemma]{Definition}

\theoremstyle{remark}

\begin{document}
\title{Moduli space for generic unfolded differential linear systems }
\author[J. Hurtubise]{Jacques Hurtubise} 
\address{Department of Mathematics, McGill
University, Burnside Hall, 805 Sherbrooke Street West, Montreal
(Qc), H3A 0B9, Canada} 
\email{jacques.hurtubise@mcgill.ca}

\author[C. Rousseau]{Christiane Rousseau }
\address{D\'epartement de
math\'ematiques et de statistique\\Universit\'e de Montr\'eal\\C.P. 6128,
Succursale Centre-ville, Montr\'eal (Qc), H3C 3J7, Canada.}
\email{rousseac@dms.umontreal.ca}

\thanks{Research supported by NSERC and partially by FRQNT in Canada}

\keywords{Stokes phenomenon, irregular singularity, unfolding, confluence, divergent series, monodromy, analytic classification, summability, flags, moduli space}
\date{\today}

\begin{abstract}
In this paper, we identify the moduli space for germs of generic unfoldings of nonresonant linear differential systems with an irregular singularity of Poincar\'e rank $k$ at the origin, under analytic equivalence. The modulus of a given family was determined in \cite{HLR}: it comprises a formal part depending analytically on the parameters, and an analytic part given by unfoldings of the Stokes matrices. These unfoldings are given on  ``Douady-Sentenac'' (DS) domains  in the parameter space covering the generic values of the parameters corresponding to Fuchsian singular points. 
Here we identify exactly which moduli can be realized. A necessary condition on the analytic part, called compatibility condition, is saying that the unfoldings define the same monodromy group (up to conjugacy) for the different presentations of the modulus on the intersections of DS domains. With the additional requirement that the corresponding cocycle is trivial and good limit behavior at some boundary points of the DS domains, this condition becomes sufficient. 
In particular we show that any modulus can be realized by a $k$-parameter family of systems of rational linear differential equations over $\CP^1$ with $k+1$, $k+2$ or $k+3$ singular points (with multiplicities). Under the generic condition of irreducibility, there are precisely $k+2$ singular points which are Fuchsian as soon as simple. This in turn implies that any unfolding of an irregular singularity of Poincar\'e rank $k$ is analytically equivalent to a rational system of the form $y'=\frac{A(x)}{p_\eps(x)}\cdot y$, with $A(x)$ polynomial of degree at most $k$ and $p_\eps(x)$ is the generic unfolding of the polynomial $x^{k+1}$.  \end{abstract}

\maketitle

\pagestyle{myheadings}\markboth{J. Hurtubise and C. Rousseau}{Moduli space for generic unfolded  linear differential systems}

\section{Introduction}

The local classification in the complex domain of  germs of systems of linear differential equations, with a pole at the origin, 
$$y' = \frac {A(x)}{x^{k+1}} \cdot y,$$
exhibits  a qualitative shift as one goes from  $k=0$ to $k>0$. In both cases, one can first go to a normal form by a formal gauge transformation $g(x)$, that is a power series in $x$. Let us suppose for simplicity that the system is nonresonant i.e. the leading term $A(0)$ is diagonal, with eigenvalues which are distinct (for $k=0$, one would ask that they be distinct modulo the integers).  One can perform a formal normalization to have $A(x)$ diagonal, and a polynomial of order $k$. 
If we then proceed to the analytic classification, one finds that for $k=0$, the formal classification is the same as the analytic classification, in the absence of resonance. For $k>0$, the situation is very different. The formal gauge transformation does not in general converge, and one only has analytic solutions on $2k$ sectors around the origin, with constant matrices (Stokes matrices) relating the solutions as one goes from sector to sector. If we further assume that $A(0)= \mathrm{diag} (\lambda_1,\dots, \lambda_n)$, and that we have permuted the coordinates of $y$ and rotated $x$ to $e^{i\theta}x$ so that 
\begin{equation}
\mathrm{Re}(\lambda_1)>\dots > \mathrm{Re}(\lambda_n),\label{order_eigenvalues}\end{equation} 
then the Stokes matrices alternate between upper triangular and lower triangular as one goes from sector to sector. Once one has fixed the formal normal form, the Stokes matrices provide complete invariants. While these can be thought of as generalised monodromies (e.g.,\cite{MR}), the passage from the irregular case ($k>0$) to the regular case $k=0$ is not immediate, since the monodromies for $k=0$ have no limit at the confluence.
This passage however is a natural one to consider, in particular when unfolding a system with an irregular singularity. Doing so sheds new light on the meaning of the Stokes matrices, and this has been studied in particular in \cite{Ra}, \cite{G}, \cite{cLR2}, \cite{HLR}.

Indeed, one has a deformation from one to the other. Let $p_\eps(x), \eps\in \C^k$ be the generic deformation of $x^{k+1}$ as a polynomial of degree $k+1$:
\begin{equation} p_\eps(x)=x^{k+1}+ \eps_{k-1}x^{k-1} +\dots+\eps_1x+\eps_0,\label{def:p}\end{equation} 
and then consider a deformed system 
\begin{equation}y' = \frac {A(\eps, x)}{p_\eps(x)} \cdot y. \label{eq_deployee}\end{equation} 
 For a generic value of $\eps$, the singularities are simple poles, and the classification, for a fixed formal form, is essentially the monodromy representation; at $\eps=0$, and more generally, along the discriminant divisor $\Delta(\eps)=0$, one has higher order singularities and so the Stokes factors.
 
 In  \cite{cLR2}, \cite{HLR}, the problem of studying the family, i.e., the {\it unfolding} of the original system was addressed. It involved the seemingly simple step of rewriting the equation as a coupled system
 \begin{align} \dot y &= A(\eps, x)\cdot y\label{vector}\\ \dot x &= p_\eps(x)\label{scalar}\end{align}
 of a vector equation (\ref{vector}) and a scalar equation (\ref{scalar}) in an extra variable $t$.   
 
 One begins by analyzing the scalar equation, appealing to some quite elegant work of Douady and Sentenac \cite {DS05}. One considers real trajectories (${\rm Im}(t)$ = constant) of the scalar equation.  Away from a real codimension one  bifurcation locus, the results of  \cite {DS05} partition the $\eps$-space into $C_k$ \emph{Douady-Sentenac domains}, or \emph{DS domains} $\widetilde{S_s}$, all adherent to $0$, where 
 \begin{equation}C_k= \frac{\binom{2k}{k}}{k+1}\label{C_k}\end{equation} is the $k$-th Catalan number. The $\widetilde{S_s}$ can be extended to wider DS   domains $S_s$, which retract to $\widetilde{S_s}$, the union of which covers all values of $\eps$ for which the singular points are all of multiplicity one, that is, the complement of the discriminantal locus. Each of these domains is contractible. In \cite{HLR}, these domains are referred to as sectoral domains.

On each $S_s$, and for each $\eps\in S_s$, the $x$-plane is decomposed into $2k$ generalized sectors $\Omega_{j,\eps}^\pm$ (see Figure~\ref{Omega_Stokes}(b)), each adherent to two singular points of the scalar equation as in Figure~\ref{Omega_Stokes}(b). This generalizes the natural sectors of normalization for $\eps=0$;  indeed, the boundaries of the sectors, instead of all terminating at a point, terminate at different singular points, which are the various zeroes of $p_\eps(x)$. These zeroes are vertices of a natural tree, so that the singular point at $\eps=0$ has in some sense expanded into a tree. 

\begin{figure}\begin{center}
\subfigure[$\eps=0$]{\includegraphics[width=5cm]{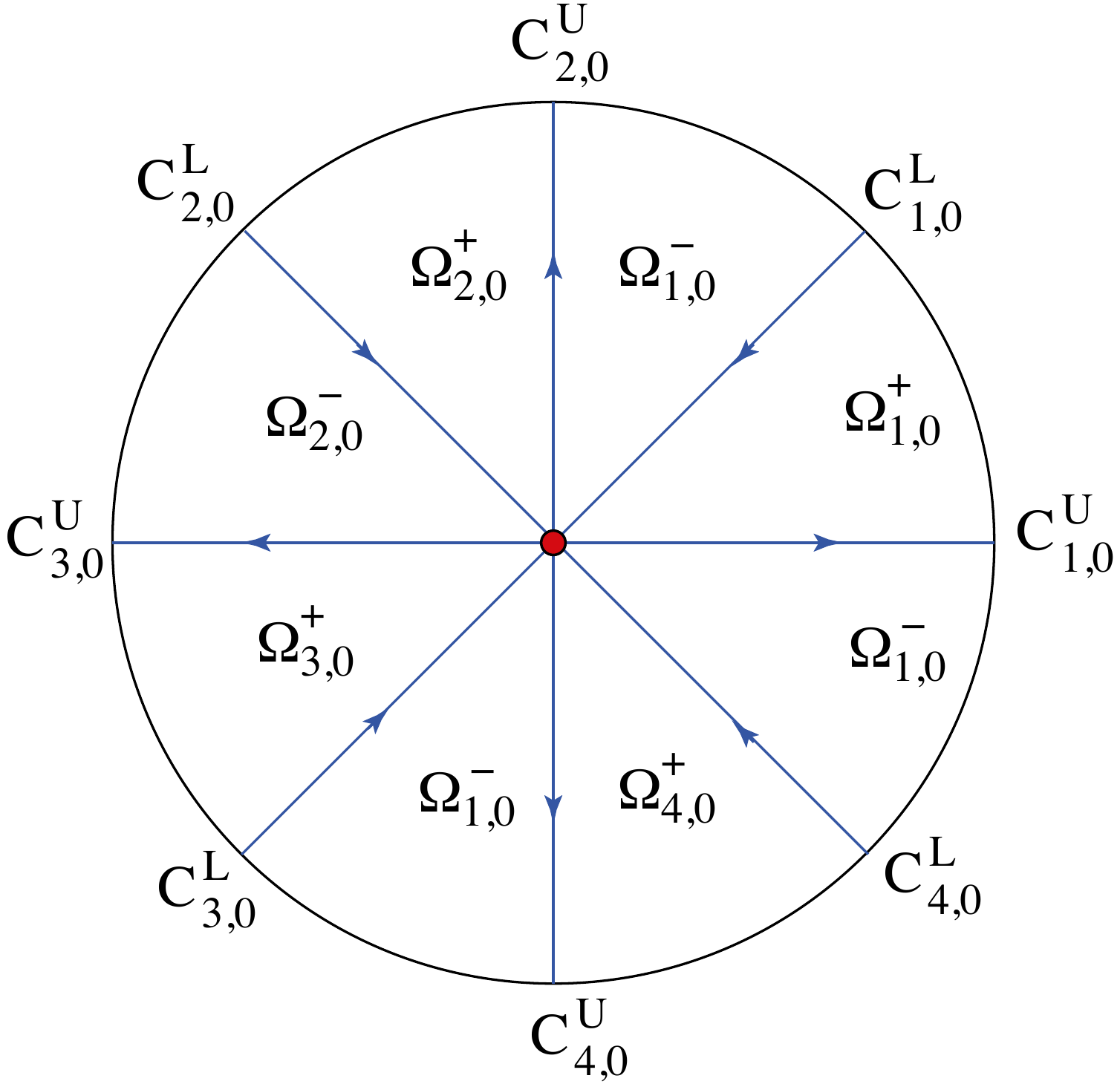}}\qquad\quad
\subfigure[$\eps\neq0$]{\includegraphics[width=5cm]{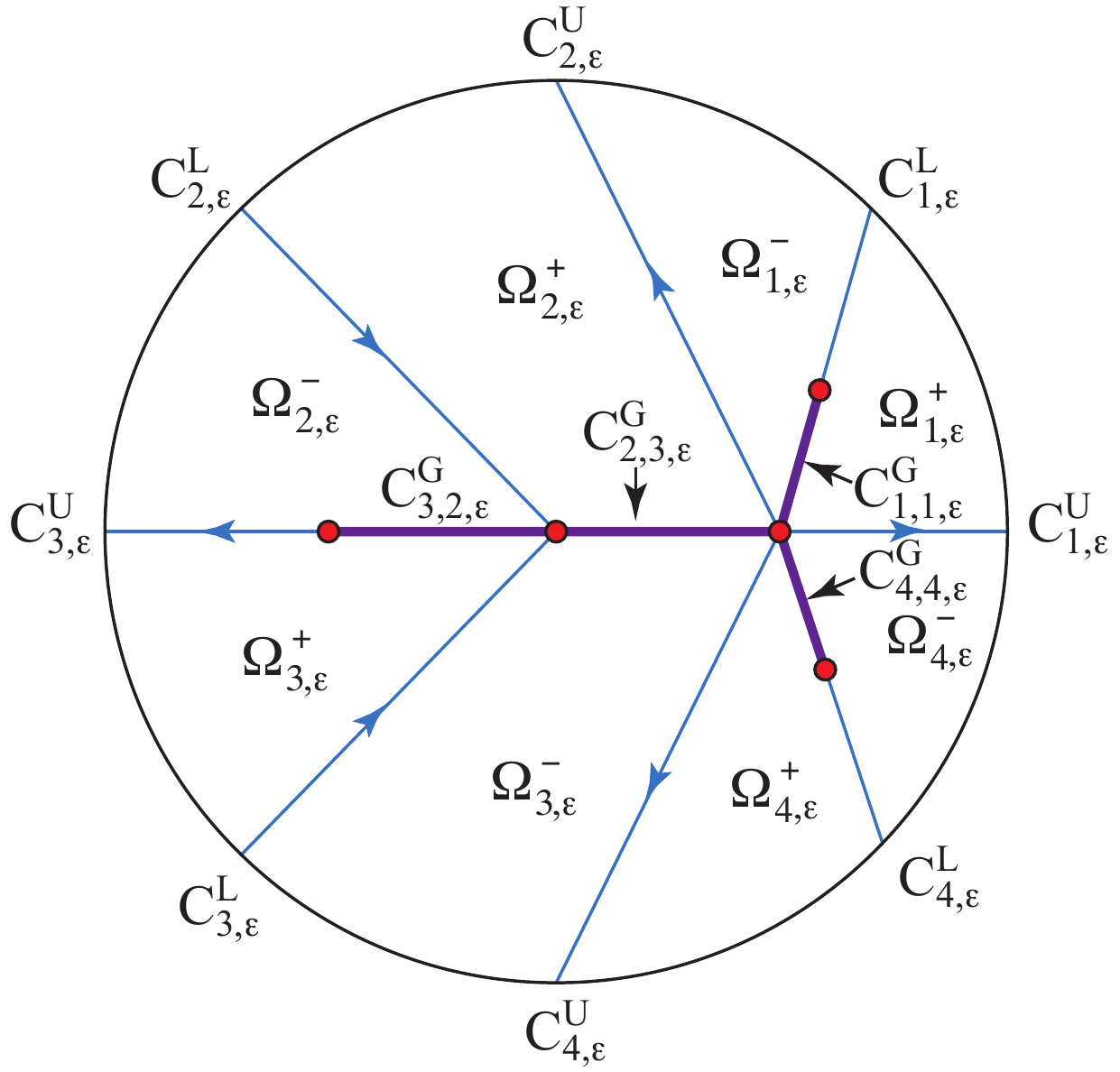}}
\caption{ The domains $\Omega_{k,\eps}^-$ and associated Stokes matrices.  In fact, for $\eps\neq0$ the domains are generically limited by spiraling curves in the neighborhood of the singularities: we have represented the boundaries by segments  for simplicity.} \label{Omega_Stokes} \end{center} \end{figure}

Turning now to the vector equation, one has first (see \cite{HLR})  a straightforward extension of the formal normal form  to  the deformed equation.
Indeed, we recall that a linear system $x^{k+1}y'=A(x)\cdot y$, $y\in\C^n$, with  irregular singular point of Poincar\'e rank $k$, and leading order term with distinct eigenvalues, has a diagonal normal form 
\begin{equation} y'=  \frac{1} {x^{k+1}}(\Lambda_0 + \Lambda_1 x+ \dots + \Lambda_kx^k)\cdot y, \qquad y\in \C^n,\label{Normal_form_0}\end{equation}
where the $\Lambda_j$ are diagonal matrices containing the $(k+1)n$ formal invariants of the system, and $\Lambda_0$ has distinct eigenvalues. This extends   to the deformed equation (\cite{HLR}):
\begin{equation} y'=  \frac{1} {p_\eps(x)} (\Lambda_0(\eps) + \Lambda_1(\eps) x+ \dots + \Lambda_k(\eps) x^k) \cdot y, \qquad y\in \C^n.\label{normal_form_eps}\end{equation}
For systems with a fixed formal normal form, one then wants an analytic classification: for this, one first fixes a DS domain $S_s$. For $\eps\in S_s$, on each sector $\Omega_{j,\eps}^\pm$ in the $x$-plane, one has   geometrically defined bases of solutions, defined up to an action of the diagonal matrices, which are such that in going from one sector to the next, the change of basis matrices generalize the Stokes matrices, and indeed exhibit the same alternation between upper and lower triangular. When one goes from $S_s$ to $S_{s'}$, the decomposition bifurcates, and one obtains different matrices; we will see that these are related to each other by algebraic relations. 
 
In \cite{HLR}, it was shown that these matrices (together with the formal invariants of \eqref{normal_form_eps}) provide a complete invariant for the system. That is to say, if one fixes the formal normal form, the generalized Stokes matrices defined for each DS domain $S_s$  determine the system: two unfoldings with the same invariants are analytically equivalent. The current paper addresses the {\it realization problem}:  fixing 
  formal invariants given by $k+1$ diagonal matrices $\Lambda_0(\eps), \dots, \Lambda_k(\eps)$ depending analytically of $\eps$, and given $C_k$ sets  of $2k$ unfolded Stokes matrices, one for each DS domain $S_s$, depending analytically on $\eps\in S_s$ with the same limit for $\eps\to 0$ under which conditions 
 does one have a system (\ref{vector}) corresponding to them? The answer turns out to be that the formal invariants and Stokes matrices define the same monodromy representations, and  exhibit some natural regularity near the discriminantal locus.  This is achieved in particular by realizing the equivalence class as the germ of a system defined over all of $\CP^1$, and exploiting compactness. In the course of doing this, we also discuss the development of {\it normal forms}, i.e. canonical representatives of equivalence classes.

 To prove the realization we proceed in two steps. The first step is to realize the modulus over each DS domain $S_s$ in parameter space. We find that any formal data and Stokes matrices depending analytically on the parameters can be realized. When the singular points are simple, we obtain a Fuchsian system  on $\CP^1$ with singular points at $x_1, \dots, x_{k+1}, R, \infty$,  where  $x_1, \dots, x_{k+1}$ are the zeroes of $p_\eps$, and $R, \infty$ are two fixed auxiliary points independent of $\eps$:
\begin{equation}y'=\left(\sum_{\ell=1}^{k+1} \frac{A_\ell(\eps)}{x-x_\ell} + \frac{\widetilde A(\eps)}{x-R}\right) \cdot y.\label{family_Fuchsian}\end{equation}
This realized family is unique up to gauge transformations which are constant in $x$. As a connection, it is indecomposable;  by, for example, diagonalizing    $A(0,0)$ and normalising certain terms, we can make it unique. The residue matrix at infinity of \eqref{family_Fuchsian} is simply $A_\infty(\eps)= -\sum_{\ell=1}^{k+1}A_\ell(\eps) - \widetilde A(\eps)$. 

\bigskip The next step is to realize the modulus over a full neighbourhood of the origin in parameter space, which we can take as a polydisk $\D_\rho$. For this, an additional condition is needed. Indeed, since the modulus characterizes families of systems of linear differential equations up to analytical equivalence, it is obviously a necessary condition that the realized families over the different DS domains be analytically equivalent over the intersection of DS domains. A necessary condition ensuring the equivalence is that the monodromy representations associated to the realizations over the different DS domains be the same (up to conjugacy as in the Riemann-Hilbert problem). The matrices conjugating the representations from one DS domain to the next must in addition form a trivial cocycle.   This condition turns out to be  sufficient, and can be expressed in terms of the modulus, i.e. the formal invariants and the Stokes matrices over each DS domain. 
Now, over each DS domain, we have realized unique normalized families of the form \eqref{family_Fuchsian}. Because of the normalization, they coincide on the intersection of the DS domains. 
Hence, we have realized the modulus over a polydisk $\D_\rho$ in parameter space minus  the discriminantal locus $\Delta=0$. To extend the realization to the whole polydisk $\D_\rho$, we first extend it to the regular points of $\Delta=0$ (where only two singular points coalesce).
We then fill in  for the remaining codimension two set of values of $\eps$ using  Hartogs' Theorem.

The particular case where the system of Stokes matrices at $\eps=0$ is irreducible is worth noticing: indeed, we can realize the data in a system
\begin{equation}\left(\sum_{\ell=1}^{k+1} \frac{A_\ell(\eps)}{x-x_\ell} \right) \cdot y,\label{family_irreducible}\end{equation}
which is Fuchsian when the singular points are distinct.

 Our realisation gives us a (local)  normalisation. We close the paper with a discussion of normal forms.

\section{Preliminaries}\label{sec:Preliminaries}
\subsection{The scalar equation and DS domains} We recall in this section the results of the unpublished work of Douady and Sentenac  \cite{DS05}. As discussed in  \cite{HLR} the construction of the modulus in  \cite{HLR} was governed by the dynamics $x(t)$ of the polynomial vector field \begin{equation}\dot x = \frac{dx}{dt}=p_\eps(x)\frac{\partial}{\partial x}, \qquad p_\eps(x) =x^{k+1} + \eps_{k-1} x^{k-1} + \dots+\eps_1x+\eps_0,\label{vf}\end{equation} on $\CP^1$. It will suffice for the moment to limit ourselves to the set of generic values \begin{equation}\Sigma_0=\{\eps : \Delta(\eps)\neq0\},\label{Sigma_0}\end{equation} where $\Delta(\eps)$ is the discriminant of $p_\eps(x)$. 

 The analysis centres on understanding the {\it real} flow lines in the $x$-plane, i.e. those given as the images of ${\rm Im}(t)$ = constant. Thus, we are restricting complex flow lines  to a foliation of real lines; as such, the dynamics near the singular points $p_\eps(x) = 0$ has certain special properties, not shared with generic real vector fields on the plane. On $\Sigma_0$, each singular point $x_\ell$ has an associated eigenvalue  $\lambda_\ell=p_\eps'(x_\ell)$. Then 
\begin{itemize} 
\item
The point $x_\ell$ is a radial node if $\lambda_\ell\in \R$. It is attracting (resp. repelling) if $\lambda_\ell<0$ (resp. $\lambda_\ell>0$).
\item The point $x_\ell$ is a center if $\lambda_\ell\in i\R$.
\item The point $x_\ell$ is a focus if $\lambda_\ell\notin \R\cup i\R$. It is attracting (resp. repelling) if ${\rm Re}(\lambda_\ell)<0$ (resp. ${\rm Re}(\lambda_\ell)>0$).
\end{itemize}
Moreover such a vector field never has a limit cycle. 

\begin{remark}\label{attrac_repell} We emphasize that whether a singular point is of $\alpha$-type (repelling) or $\omega$-type (attracting) depends importantly on the family of real flow lines one is considering in the complex plane, in particular their asymptotic direction. On the complex line, there is no such concept. A linearized example suffices to illustrate. Indeed, if one is considering $\dot x = ax$, one has the solution $x = \exp (at)$ over the complex line. Substituting $t= r e^{i\theta_j}$, for $\theta_j $ real and  varying $r$ through the positive reals, one has one solution $\exp (ae^{i\theta_1} r)$ which, when $r\to +\infty$, spirals into the origin when $\theta_1$ is chosen so that  ${\rm Re}(ae^{i\theta_1})$ is negative, and one solution $\exp (ae^{i\theta_2} r)$ which spirals outward when ${\rm Re}(ae^{i\theta_2})$ is positive and $r\to +\infty$.  \end{remark}

To understand the global structure of the real flow lines, the point $x=\infty$ serves as an organizing centre; indeed, the   vector field $v_\eps(x)$ has a pole of order $k-1$ there, and the system is structurally stable in the neighborhood of $\infty$ as $\eps$ varies. 

\begin{figure}\begin{center}
\includegraphics[width=4.5cm]{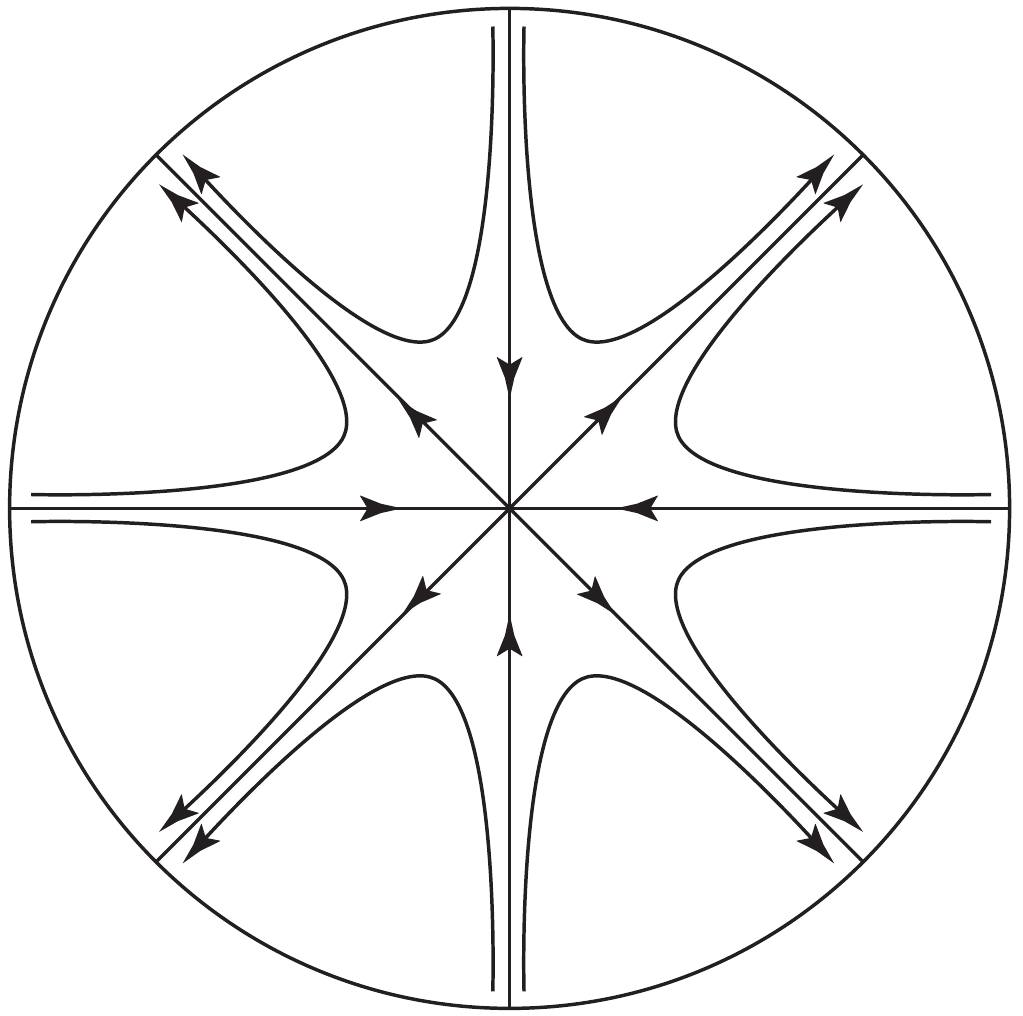}\caption{The separatrices of the pole at $\infty$.}\label{infinity}\end{center}\end{figure}

Among the solutions  ${\rm Im}(t) = \mathrm{constant}$, we  have   $2k$ separatrices at $x=\infty$, alternately attracting and repelling (see Figure~\ref{infinity}). On $\Sigma_0$, the dynamics is completely determined by the separatrices.  Following the separatrices in from infinity, either  backwards or forwards, one has:

\begin{itemize} 
\item For generic values of $\eps$ one lands at repelling (${\rm Re}(t)\rightarrow -\infty$) or attracting (${\rm Re}(t)\rightarrow  \infty$) singular points $x_\ell$ of focus or radial node type. For such an $\eps$, there exists trajectories joining the singular points two by two and choosing one trajectory for each pair yields  a tree graph (see Figure~\ref{Omega_Stokes} (b), which we call the \emph{Douady-Sentenac tree}. We denote by $\Sigma_1$ the set of generic values of $\eps$ in $\Sigma_0$ for which there are no homoclinic connections between separatrices of $\infty$.
\item The sets of generic $\eps$ are separated by the closures of bifurcation sets of real codimension $1$, where a homoclinic connection occurs between an attracting separatrix and a repelling separatrix of infinity: there is then a real integral curve flowing out from infinity in the $x$-plane and flowing back to infnity in finite time. On these bifurcation sets, the  singular points can be split into two sets $I_1$ and $I_2$ and 
\begin{equation}\sum_{\ell\in I_1} \frac1{p_\eps'(x_\ell)}\in i\R.\label{bifur}\end{equation} One sees this by integrating the form $dt$ along a homoclinic orbit, and evaluating residues.
When $I_1$ is a singleton, the corresponding singular point is a center. 
\end{itemize}

The closures of these codimension one sets partition their complement in $\eps$-space into a certain number of connected components, which we call {\it Douady-Sentenac domains}, or {\it DS domains}. As one moves through the closure of the real codimension $1$ homoclinic locus, the limit points of attachment of the separatrices will change. These attachments are constrained, and the results of Douady and Sentenac tell us that there are 
$$C_k= \frac{\binom{2k}{k}}{k+1}$$
ways of doing it (see Figure~\ref{Omega_Stokes}(b)), thus dividing the complement of the closure of the homoclinic set into $C_k$ open sets $\widetilde{S}_s\subset \Sigma_0$ in parameter space.  We will see how this happens below.

\subsection{The sectors in $x$-space over $\C$}

The sectors will be defined in several steps: we first define sectors $\widetilde{\Omega}_{j,\eps,S_s}^\pm$, which cover $\C$ minus a set of measure $0$, and then their enlargements $\Omega_{j,\eps,S_s}^\pm$, which cover $\C$ minus the zeros of $p_\eps(x)$. We modify them later so that they are adequate on a disk. In this section we limit ourselves to sectors on $\C$. 

Depending on the meaning one gives to the word ``explicit'', the flow lines of the scalar equation are explicitly solvable. Indeed, one has a $t(x)$, which globally is multi-valued  with multi-valued inverse,  defined by
\begin{equation} t(x) =\int\frac{dx}{p_\eps(x)} .\label{def:time}\end{equation}
 
To fix ideas, when  $k=1$, one can solve further:
\begin{equation}
t(x)= \begin{cases} -\frac1{x}, &\eps_0=0,\\
\frac1{\sqrt{-2\eps_0}}\log\frac{x-\sqrt{-\eps_0}}{x+\sqrt{-\eps_0}},&\eps_0\neq0.\end{cases}\label{example:time}\end{equation}
For $\eps_0\neq0$, the image of a disk $\D$ is the complement of a line of periodic holes located $\frac{2\pi i}{\sqrt{-2\eps_0}}$ apart. At the limit when $\eps_0=0$, all holes but one disappear to infinity.

More generally, since $\infty$ is a pole of order $k-1$ for $v_\eps(x)$ (hence a regular point if $k=1$), it can be reached in finite time. 
Hence, the image of $\infty$ in $x$-space consists of finite point(s) in $t$-space. 
Also, the time $t$ is ramified at $\infty$ for $k>1$, since $t\sim -\frac{k}{x^k}$ near $\infty$. 
For $\eps=0$, the image in $t$-space of a disk $\D$ in $x$-space is the outside of a disk on a $k$-sheeted Riemann surface. For $\eps\neq0$, the map $t$ is multivalued with the different images periodically spaced. The image of a disk is the complement of a countable number of periodically spaced holes placed on a branched $k$-sheeted Riemann surface. The periods between the holes tend to $\infty$ when $\eps\to 0$. (More details in \cite{HLR} and below.)

The fact that the integral curves are given in terms of a simple integral gives quite a lot of control over the behaviour of solutions.
For $\eps$ in a DS domain $\widetilde S_s$ in parameter space, we will   get a decomposition of the complement in the $x$-plane of the separatrices emerging from infinity in the complex $x$-plane  into $2k$ (generalized) sectors $\widetilde\Omega_{j,\eps,S_s}^\pm$,  ordered cyclically at infinity in a {\it clockwise} direction (around infinity!) $\widetilde{\Omega}_{1,\eps,S_s}^+$,  $\widetilde{\Omega}_{1,\eps,S_s}^-$,  $\widetilde{\Omega}_{2,\eps,S_s}^+$,  $\widetilde{\Omega}_{2,\eps,S_s}^-$, \dots,  $\widetilde{\Omega}_{k,\eps,S_s}^-$, as in Figure~\ref{Omega_Stokes}(b). The clockwise direction at infinity will become the standard anti-clockwise direction when viewed from the finite regions of the plane. In a neighbourhood of $x=\infty$ the sectors are well defined: they are simply the complement of the separatrices in the bifurcation diagram at infinity. The sign $\pm$ reflects whether the sector has its outgoing separatrix on its right ($+$) or left ($-$) hand side, when looking outward from infinity. 

The question is then what happens as one extends inwards. Following the separatrices inwards, each  sector will be adherent to two vertices $x_\alpha$, $x_\omega$, with  $x_\omega$ attracting and $x_\alpha$ repelling, and both points being  singular points of $v_\eps(x)$. The  boundary  of each sector will consist of three sides:
\begin{itemize} 
\item an 
attracting separatrix of $\infty$ (i.e. going to infinity) emerging from a singular point $x_\alpha$ of $\alpha$-type (a source) of   $v_\eps$: in $t$-coordinate it is represented by a horizontal half-line (real values of time) starting from infinity to the left (corresponding to the singular point $x_\alpha$) to a finite image in $t$ of $x=\infty$;
\item a repelling separatrix of infinity going to a singular point $x_\omega$ of $\omega$-type (a sink) of some $v_\eps$: in $t$-coordinate, it is represented by a horizontal half-line starting at a finite image of $x=\infty$ and directed to the right; 
\item and a curve (not uniquely defined) from $x_\alpha$ to $x_\omega$, corresponding to a real  trajectory  of $v_\eps$ from $x_\alpha$ to $x_\omega$. 
 \end{itemize}
 
Thus the  sector is topologically a triangle, though of course    most of the time   the boundaries of the  sectors in $x$-space will be  spiraling curves.  While giving the formulae for the boundary is impossible, because of  the differential equation $\frac{dx}{dt}=p_\eps(x)$,  the sector in $t$-space is fairly simple, and is given by  a horizontal strip (see Figure~\ref{strip_horizontal}).

\begin{figure} [ht!]
\begin{center}
\subfigure[Sector on $\C$]{\includegraphics[width=5.5cm]{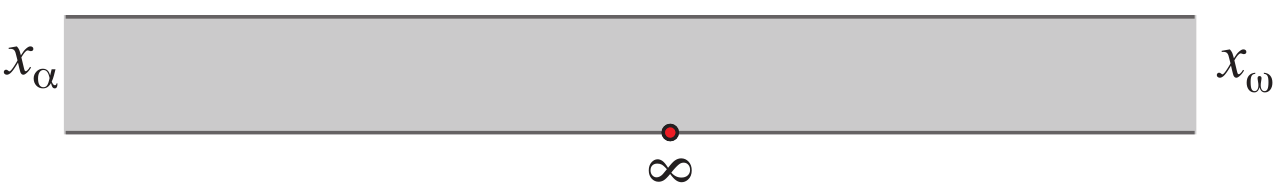}}\qquad\subfigure[Sector on a disk]{\includegraphics[width=5.5cm]
{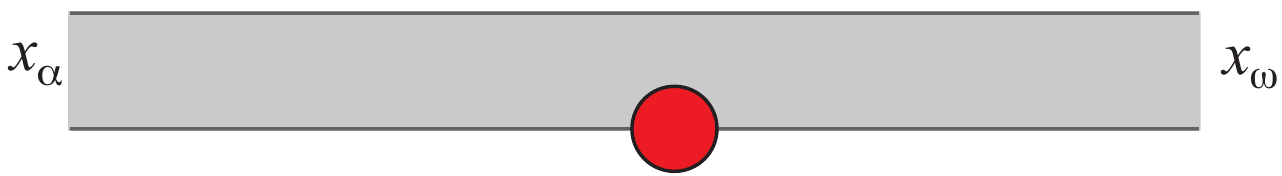}}
\caption{ A strip in $t$-space whose image is a sector $\Omega_{j,\eps, S_s}^-$ in $x$-space. } \label{strip_horizontal} \end{center} \end{figure}
\medskip

\noindent{\bf Construction of a pair of  sectors $\widetilde{\Omega}_{j_1,\eps}^+, \widetilde{\Omega}_{j_2,\eps}^-$ in $x$-space.}   We examine the construction in a bit more detail, giving now a pair of sectors as the image of a (wider) horizontal strip in $t$-space. Let us choose a $\widetilde{\Omega}_{j_1,\eps}^+$ for some $j_1$, and consider the two separatrices on its boundary. These can be thought of as the image of one line ${\rm Im}(t)$ = constant in $t$-space,  the union of two half-lines, with the first emerging at ${\rm Re}(t)= -\infty$ from some  singular point $x_\alpha$, and  arriving at some time $t_0$ at infinity in the $x$ plane; the second half-line starts at  $t_0$ and then flows off to an $x_\omega$ at ${\rm Re}(t)= +\infty$. Moving inwards in the $x$-plane near infinity in the sector $\widetilde{\Omega}_{j_1,\eps}^+$ (and so downwards in the $t$-plane), the lines ${\rm Im}(t)$ = constant now become flow lines from $x_\alpha$ to  $x_\omega$. This patterns continues downward in the $t$-plane, until one hits a line ${\rm Im}(t)$ = constant
 which again goes through $x=\infty$ at a time $t_1$, which this time will be the boundary of a domain $\widetilde{\Omega}_{j_2,\eps}^-$, now viewed as a part of the   plane above the $t_1$-line. One has then created a horizontal  strip in the $t$ plane, bounded above by the line through $t_0$, and below by the line through $t_1$, which   corresponds to the union of two sectors $\widetilde{\Omega}_{j_1,\eps}^+, \widetilde{\Omega}_{j_2,\eps}^-$; we choose the demarcation line between the two sectors, somewhat arbitrarily, to be  the image of the  real line in $t$-space through $(t_1+t_0)/2$ (see Figure~\ref{strip_horizontal_double}).
 \begin{figure} 
\begin{center}
\subfigure[Sectors on $\C$]{\includegraphics[width=5.5cm]{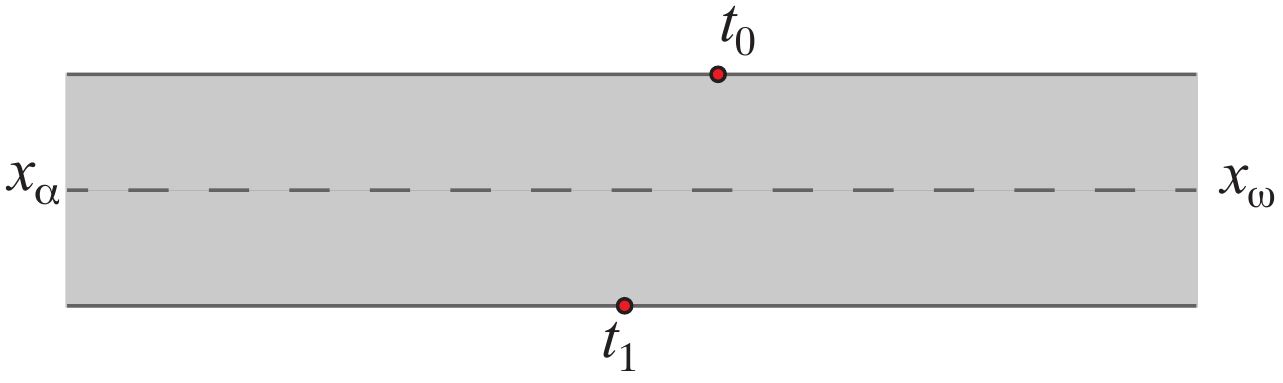}}\qquad\subfigure[Sectors on a disk]{\includegraphics[width=5.5cm]{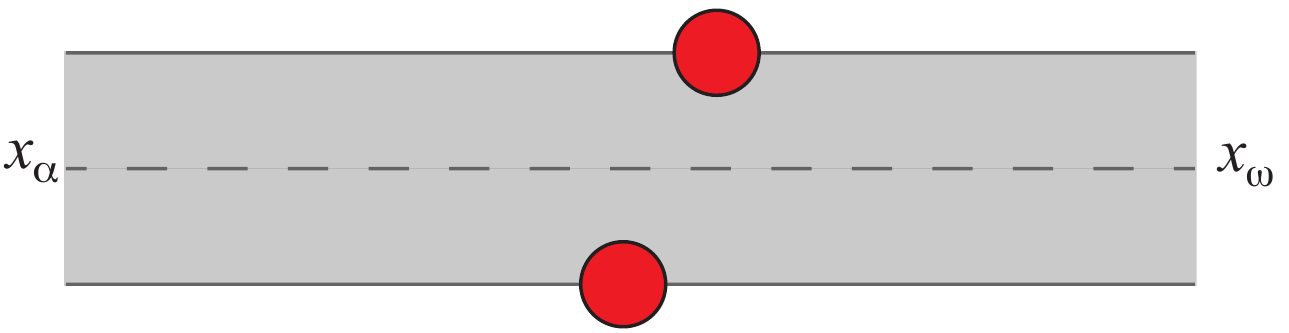}}
\caption{ A strip in $t$-space whose image is the union of sectors $\widetilde{\Omega}_{j,\eps,S_s}^+$ and $\widetilde{\Omega}_{\sigma(j),\eps,S_s}^-$ in $x$-space. } \label{strip_horizontal_double} \end{center} \end{figure}

 The above construction defines the sectors; it also does a bit more. Indeed, first of all,  it pairs the sectors, with each sector $\widetilde{\Omega}_{j,\eps,S_s}^+$ (and its enlargement $\widetilde{\Omega}_{j,\eps,S_s}^+$) with a positive sign paired with one sector $\widetilde{\Omega}_{\sigma(j),\eps,S_s}^-$ (and its enlargement $\Omega_{\sigma(j),\eps,S_s}^-$) with a negative sign. It also gives us, for each pair, a path joining the positive and negative sectors, simply as the image $\gamma_j$ in $x$-space of the segment joining $t_0$ and $t_1$; the curve $\gamma_j$ is then a path out from infinity, going back to infinity (Figure~\ref{skeleton}). As Douady and Sentenac remark, the various $\gamma_i$ do not intersect; their complement in a  sufficiently large finite disk   is a union of disjoint open sets, each containing a single zero of $p_\eps$. Homotopically, our flow lines in $x$-space give  a diagram consisting of a disk with points representing the sectors $\widetilde{\Omega}_{1,\eps,S_s}^+$,  $\widetilde{\Omega}_{1,\eps,S_s}^-$,  $\widetilde{\Omega}_{2,\eps,S_s}^+$,  $\widetilde{\Omega}_{2,\eps,S_s}^-$, \dots,  $\widetilde{\Omega}_{k,\eps,S_s}^-$, in that order, on its boundary, and non-intersecting chords (the $\gamma_j$)  joining the sectors which are paired. Dually, one has the Douady-Sentenac tree joining the various zeroes of $p_\eps$. Douady and Sentenac show that this (``Douady-Sentenac'') invariant characterizes the DS domains, and it is the count of the possible diagrams that gives us a Catalan number.
\begin{figure}
\begin{center}
\subfigure[The separating graph]{\includegraphics[width=5.5cm]{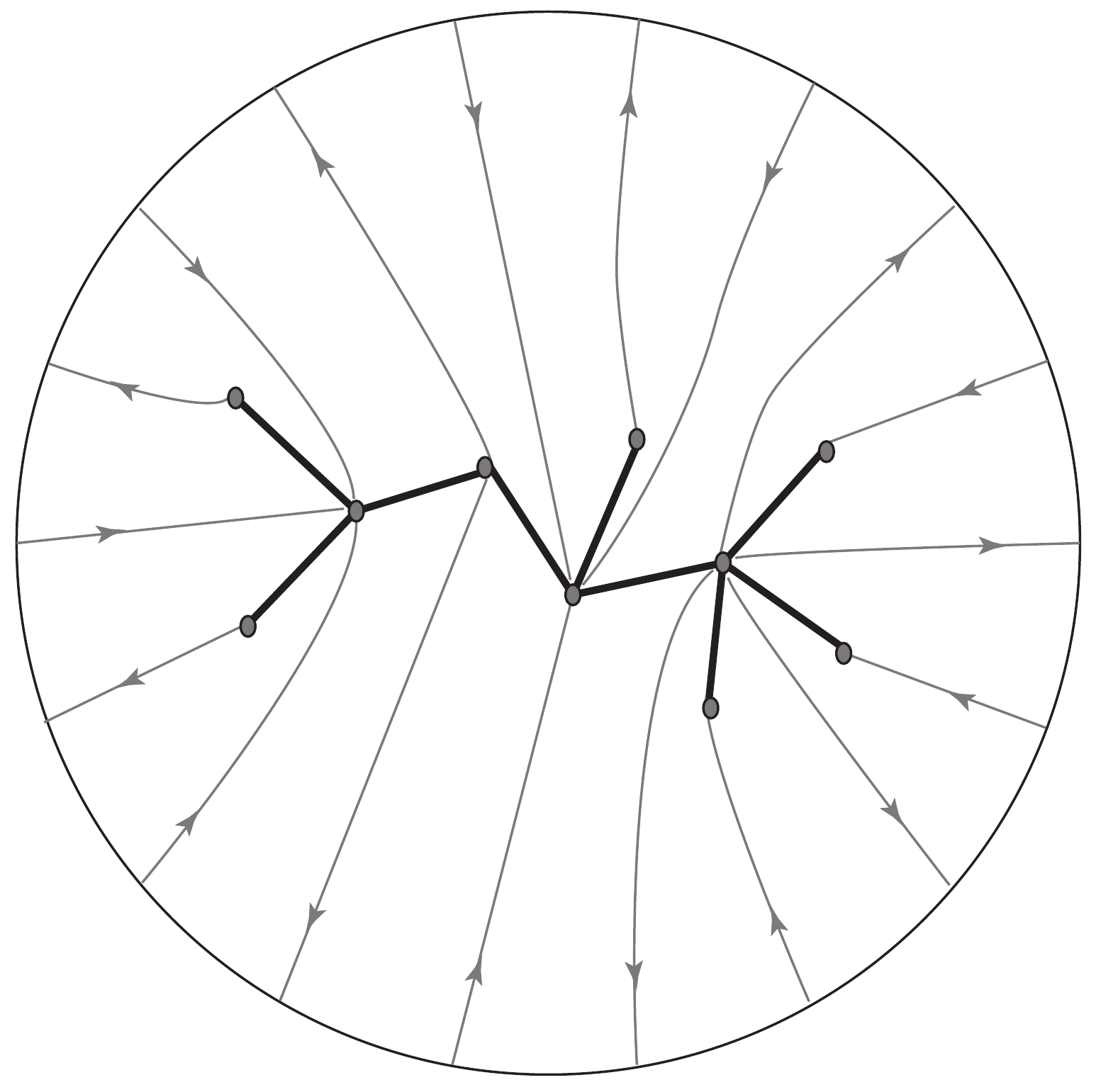}}\qquad \subfigure[The curves $\gamma_i$]{\includegraphics[width=5.5cm]{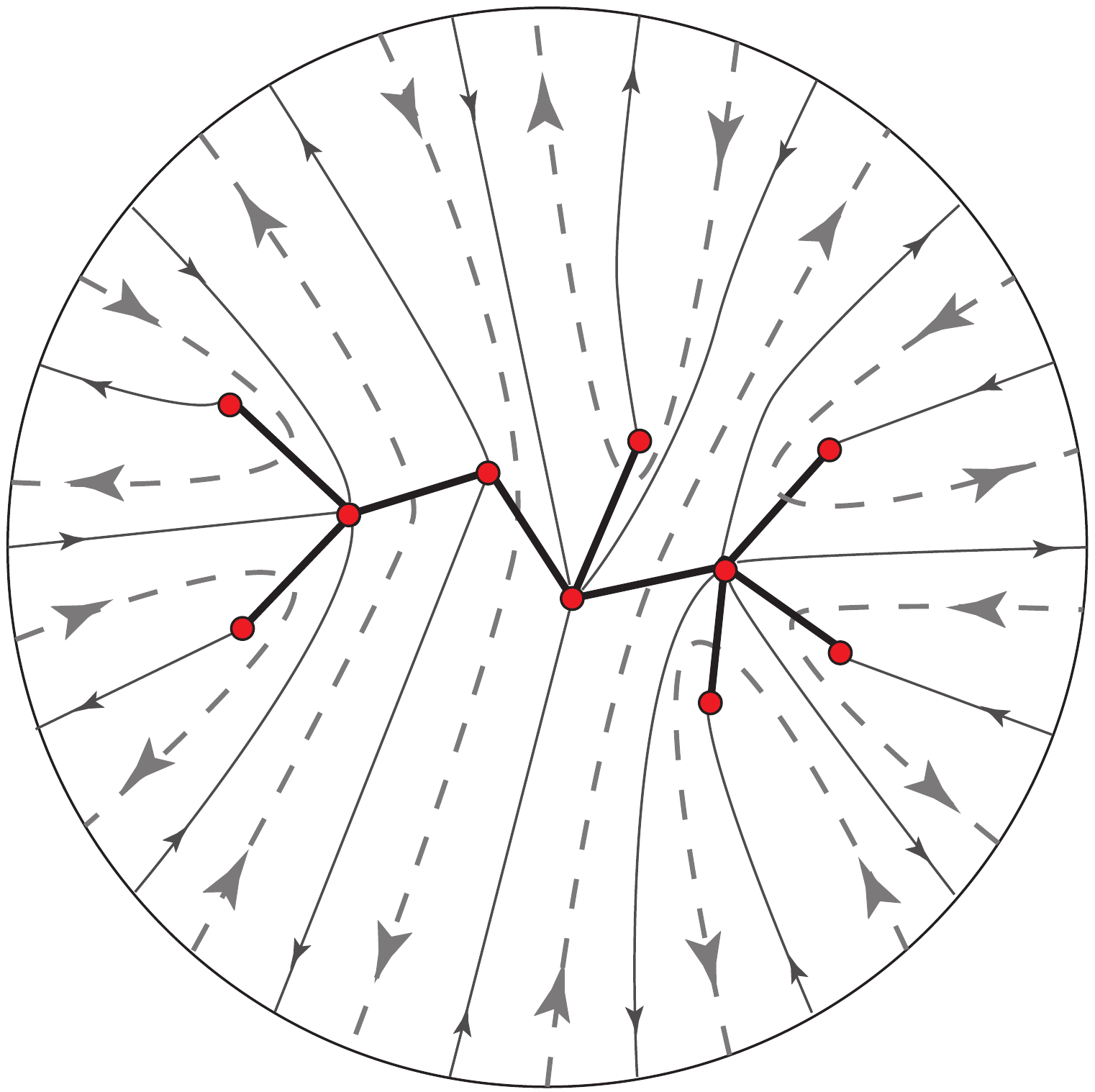}}\caption{The separating graph formed by the separatrices landing at the singular points and the curves $\gamma_i$ (in dotted lines) used to calculate the $\tau_i$.}\label{skeleton}\end{center}\end{figure}
 
 One also has more; the difference 
 $$\tau = t_0-t_1 = \int_\gamma dt =\int_\gamma \frac{dx}{p_\eps(x)}$$ gives us a complex parameter in the upper half plane, and Douady and Sentenac show that the set of such parameters as one ranges through the paired sectors parametrizes the DS domain:
  
 \begin{theorem}\label{thm:DS}\cite{DS05} There exists a biholomorphism $\Phi_s:\widetilde{S}_s\rightarrow \H^k$, where $\H$ is the upper half-plane. The map $\Phi_s$ is defined as follows: $\Phi_s(\eps)= (\tau_1, \dots, \tau_k)$, where $$ \tau_j = \int_{\gamma_j}  \frac{dx}{p_\eps(x)}.$$
 In particular, $\widetilde{S}_s$ is contractible, and so, simply connected. The set $\Phi_s^{-1}((i\R^+)^n)$, which we call the \emph{spine} of $\widetilde{S}_s$, corresponds to polynomial vector fields with only real eigenvalues at each singular point. \end{theorem}
 
Note that each Douady-Sentenac parameter $\tau_j$ is computable as  a sum of residues $\tau_j=2\pi i \sum_{\ell\in J} (p_ \eps'(x_\ell))^{-1}$ at the zeroes $x_\ell$ surrounded by $\gamma_j$. The fact that the $\gamma_j$ partition the $x$-plane into regions each containing only one zero of $p_\eps(x)$ then tells us that each residue in turn is computable from the $\tau_j$: there is a bijective linear map between the $\tau_j$ and the  $p_ \eps'(x_j))^{-1}$ (satisfying $\sum_{j=1}^{k+1}p_ \eps'(x_j))^{-1}=0$). The map between the space $\Sigma_0$ of polynomials $p_\eps$ with distinct zeroes and the quantities $p_ \eps'(x_j))^{-1}$ is many to one, a branched covering, but of course bijective to its image when restricted to a $\widetilde{S}_s$.  

The radial rescaling on $\C$ and its effect on polynomials has some interesting consequences. The bifurcation diagram of the ODE \eqref{vf}  rescales naturally when  changing $(x,t) \mapsto (r x, r^{-k}t)$, with $r\in \R^+$, which induces the parameter change \begin{equation}(\eps_{k-1}, \dots, \eps_1,\eps_0)\mapsto (r^2\eps_{k-1}, \dots, r^k\eps_1, r^{k+1}\eps_0) \label{parameter_conic}\end{equation}
This suggests a natural ``norm'' in $\eps$-space given by 
\begin{equation} \pl\eps\pl:=\max\left(\left|\eps_{0}\right|^{1/(k+1)},
\ldots,\left|\eps_{k-1}\right|^{1/2}\right)\label{norm_eps},\end{equation}
which is essentially the maximum   norm of the roots of $p_\eps$, and the following equivalence relation on the set of $\eps$:
\begin{equation}
\eps=\left(\eps_{0},\dots,\eps_{k-1}\right)\simeq\eps'=\left(\eps_{0}',\dots,\eps_{k-1}'\right)
\Longleftrightarrow\exists r\in\mathbb{R}^+\,:\:\eps_{j}'=r^{k+1-j}\eps_{j}.\label{equiv_eps}\end{equation}
This yields a conic structure in the parameter space, which is foliated by curves (equivalence classes) $\{(r^2\eps_{k-1}, \dots, r^k\eps_1, r^{k+1}\eps_0)\: :\: r\in\R^+\}$, with fixed $\eps$. Both the sets $\widetilde{S}_s$ and the strata of the bifurcation diagram are unions of equivalence classes, which can be described by their intersection with the sphere $\pl\eps\pl=\rho$ for some positive $\rho$. 

\begin{lemma}\label{lemma_estimate_tau}
 Let $S_s$ be a DS domain. Then there exists $C>0$ such that, for each $\eps\in S_s$ with $\pl \eps\pl<1$, if $\Phi_s(\eps)= (\tau_1(\eps), \dots, \tau_k(\eps))$, then 
 \begin{equation}|\tau_j(\eps)|\geq \frac{C}{\pl \eps\pl^k}.\label{tau:estimate}\end{equation}
 \end{lemma} 
 \begin{proof}    We consider the intersection of the subspace $\pl \eps\pl=1$ in parameter space with $S_s$. The quantity $C=\min_{\pl \eps\pl=1} \{|\tau_1(\eps)|,\dots, |\tau_k(\eps)|\}$ is bounded below on this intersection by a positive constant $C$. Indeed, we have seen that we have a bijective linear map relating the 
 $\tau_j$ to the allowed values of the quantities  $\nu_\ell=p'(x_\ell)^{-1} = \prod_{\ell\neq k} (x_\ell-x_k)^{-1}$ (note that $\sum_{\ell=1}^{k+1} \nu_\ell=0$), so that a bound for one is a bound for the other. On the other hand, the product $\prod_{\ell\neq k} (x_\ell-x_k)^{-1}$ is bounded below, as the roots are bounded above on the set we are considering. In particular, the product goes  to infinity as one approaches $\Delta_0$.   
 
 Then, the rescalings of the roots of $p_\eps$ and \eqref{parameter_conic} yield $|\tau_j(\eps)|= \frac{\left|\tau_j(\eps/\pl \eps\pl)\right|}{\pl \eps\pl^k}$, from which the conclusion follows.
\end{proof}

\begin{remark}The perceptive reader will have noticed a problem with our construction. When dealing with sectors in $\C$, the construction has no limit when approaching the boundary of a DS domain $\widetilde{S}_s$ where $t_0-t_1\in \R$. Also, the image of a disk $D_R$ in $x-space$ is, for small $\eps$, approximately the exterior of disks of radius $\frac1{kR^k}$. Even if $t_0-t_1\in \H$, where $\H$ is the upper half-plane, it may not be possible to pass a horizontal strip in the case of sectors on a disk (Figure~\ref{strip_horizontal_double}(b)). Of course, since $t_0-t_1=O(\pl \eps\pl)^{-k})$ (see Lemma~\ref{lemma_estimate_tau} below), for any nonzero argument of $t_0-t_1$, it suffices to reduce $\eps$ to be able to pass a horizontal strip. But we want a uniform description on a polydisk $\{\pl\eps\pl=\rho\}$. The solution will be to slant the strips. We come back to this in Section~\ref{sec:enlarging}.\end{remark}

\subsection{Sectors at $\eps=0$}

We pause for a second to remark that when $\eps$ = 0, the separatrices partition the Riemann sphere into $2k$ isomorphic sectors, rather like the portions of an orange, with vertices at $x=0$. The point is that the sectors, with our normalisation of the equations, can serve as Stokes sectors. The real flow lines within each sector flow in tangentially to the separatrices at the origin  (taking a blow-up $z=re^{i\theta}$). Meanwhile, in $t$-space, the strips  become half-spaces, as $\eps$ moves to zero.

This is, in some sense, the purpose of the construction: we are getting, for general $\eps$, sets which generalise the Stokes sectors. Unfortunately, we cannot do it uniformly in $\eps$; each DS domain has its own extension.

\subsection{Enlarging the sectors and the DS domains}\label{sec:enlarging}

We have thus, following Douady-Sentenac, described the generic locus $\Sigma_1$  of systems with no homoclinic loops through $\infty$  as the  union of a certain number of DS domains $\widetilde S_s$, each isomorphic to $\H^k$, and a boundary consisting of (the closure of) the set for which there is a homoclinic orbit joining two separatrices. Our approach for understanding the deformation of an equation will be to first study the deformations over the domains $\widetilde S_s$, and then to make identifications along their mutual boundary. As things are now, one would then need to have information continuous up to the boundary, and then know how to interpret it. It will be simpler for us to extend the DS domains $\widetilde S_s$ somewhat to  wider domains $ S_s$, on the complement  $\Sigma_0$ of the discriminant locus, so that they overlap, and we can simply glue over open sets. As usual, we are only concerned with $\eps$ small. 

Likewise, for a fixed $\eps$ we will want to expand our sectors $\widetilde{\Omega}_{i,\eps,S_s}^\pm$, so that together they constitute an open covering of the complement of infinity in the $x$-plane. This is indeed quite simple: one simply widens the strips by a fixed amount $\rho$ in $t$-space, everywhere except for a $\eta$-neighbourhood of the points $t_0, t_1$ on the boundaries of the strips, where one instead widens the angle of the intersection of the $\eta$-neighbourhood with the strip, from $\pi$ to $\pi+2\theta$, increasing the angle by $0<\theta<\pi/2$ on each side; one can choose, for neatness' sake, that $\rho = \sin(\theta)\eta$. Let us denote our widened sectors by ${\Omega}_{i,\eps,S_s}^\pm$, removing the tilde.

A final issue is that we have partitioned the complex plane; we would also like to ensure that this partition restricts well to a disk, which was our original situation.

Let us now consider the DS domains. 

Before extending the domains, let us highlight what we need. 
\begin{itemize}
\item The first thing that we need is a partition (now a covering) of the $x$-space by $2k$ sectors, with on the  intersection of neighbouring sectors, separatrices extending out in some reasonable way from infinity and lying in the intersection of our extended sectors. 
\item The characteristic feature of the DS domains  is the DS invariant, which joins a sector $ {\Omega}_{i,\eps,S_s}^+$ to the sector ${\Omega}_{\sigma(i),\eps,S_s}^-$, joining our sectors pairwise; this must of course be preserved in our extension; likewise the dual invariant, the DS tree joining the singular points. Note that both these invariants imply a partition of the singular points into $\alpha$- and $\omega$-type. These invariants should extend to the $S_s$.
\item The third thing that is required is some control of the aspect of the approach of our ``real'' lines as they go into the singular points in $x$-space, or alternately the asymptotic direction one chooses to  go out to infinity 
 in $t$-space. One of two key reasons for this has already been given in  Remark~\ref{attrac_repell}:  A simple singular point $x_\ell$ of $\dot x= p_\eps(x)$ such that $p_\eps'(x_\ell)\notin \R$ can be of $\alpha$-type or of $\omega$-type depending on how we approach it. On some of the boundary of $\widetilde{S}_s$, $x_\ell$ is a center since ${\rm Re}(p_\eps'(x_\ell))=0$. But we could force it to keep the same $\alpha$-type or $\omega$-type as in $S_s$ by approaching it along appropriate spirals inside a sector $\widetilde{\Omega}_{i,\eps,S_s}^\pm$. Again, one wants our singular  points to retain their type over the extended domains. Another reason, is that we will also want to define a flag of solutions, with decay rates, of our vector equation $\dot y = A(x,\eps) \cdot y$. This will again depend on the asymptotic direction in $t$-space. What we will want here is that the family of directions chosen not deviate too much from the real one. 
Hence the $\alpha$-type or $\omega$-type is not an absolute property of the point, but a relative property depending on a sector  $\widetilde{\Omega}_{i,\eps,S_s}^\pm$ adherent to $x_\ell$. 
The same can be done with the other boundary parts of $\widetilde{S}_s$ corresponding to  a homoclinic loop surrounding several singular points. 
 \item One final requirement, of a more technical nature, appears in the deformation of our strips as one approaches the boundary of the $\widetilde{S}_s$.  Since we are in the end dealing with a local problem, on a disk $\D$ in $x$-space, we would like our sectors to partition the disk $\D$ in a clean way. As we are considering deformations around $\eps = 0$, we note that a neighbourhood of infinity in $x$-space translates into (half-) disks $D_0$, $D_1$ around the points $t_0$, $t_1$ in our horizontal strips in $t$-space. We would like, to separate our two sectors defined by the  strip, that at least one of the real lines from $x_\alpha$ to $x_\omega$ not intersect the half discs. On the other hand one need not stick to real lines, and in fact what we really want is that the lines defining the DS invariant and the dual DS-tree stay inside our disk $\D$.
 \item On the other hand, staying in a finite region of the $t$-plane, i.e away from the singular points in $x$-space, one has a holomorphic differential equation; this is a flat connection, and parallel transport along any two curves that are homotopic on the complement of the singularities gives the same result. In other words, one can deform quite freely as long as one stays away from the singularities, and does not wander out towards infinity.
\end{itemize}

Now that we know what we want, let us see what we get when we  go to the boundary of $\widetilde S_s$.
The parametrization of the $\widetilde S_s$ places, quite conveniently, the discriminantal locus out at infinity in $\H^k$, with the width of at least one strip in $t$-space tending to $\infty$.  In  particular, for $\eps\to 0$, each sector for $\eps=0$ corresponds to the image of a horizontal half-plane in $t$-space. 
We are, as we said,  interested in deformations of an irregular singular point, i.e. in $\eps$ small. From the scaling action considered above in \eqref{parameter_conic}, or simply by considering the relation defining $x$, this means  that we can restrict our attention to DS parameters $\tau_1, \dots,\tau_k$ with $|\tau_i|>C\pl\eps\pl^{-k}$. This fact is of some use in helping  our constructions  go through.

On the other hand, in real codimension one, as one is moving to and through a homoclinic orbit, it is easy to see that   the width ${\rm Im}(\tau)$, where $\tau = t_0-t_1$, of a strip is tending to zero,  so that in the limit one has a real line in the $t$-plane passing both through $t_0$ and $t_1$; the homoclinic orbit is then the image in $x$-space of the real segment joining $t_0$ and $t_1$. Moving through changes the order of the strips   and bifurcates some of the limit points at $t=\pm\infty$ to other singular points in the $x$-plane.  When the homoclinic orbit surrounds a unique singular point, then  the singular point bifurcates from attracting to  repelling, or vice versa.

The domains are defined by where the separatrices, the flow lines ending or starting at $x=\infty$ end up in the $x$-plane. The idea for extending  the DS domain  is simple; basically, away from the singular points, we are dealing with one large complex (and multi-valued) orbit.  There is nothing sacrosanct about the lines ${\rm Im}(t )= $ constant; one could equally well, for example, consider lines ${\rm Im}(e^{i\beta}t )= $ constant, or even piecewise linear curves of that type; these are, after all, restrictions of the complex flow lines of the same complex equation. We will see that the angle $\beta$ is important, but that there is some flexibility, as long as the angle $\beta$ is kept small. Thus, for example  fixing $t_0=0$, if one is moving the DS parameter $t_1$ up to the real positive axis, one can also move the family of lines, for example to ${\rm Im}(e^{i\beta(t_1)}t )$ = constant, for a suitable $\beta(t_1) $ so that the strip moves with $t_1$ and preserves the limit points of the separatrices, as $ {\rm Im}(t_0-t_1 )$ becomes zero, and even beyond (see Figure~\ref{strip_hor_slant_double}).  
\begin{figure} 
\begin{center}
\subfigure[Sectors on $\C$]{\includegraphics[width=5.5cm]{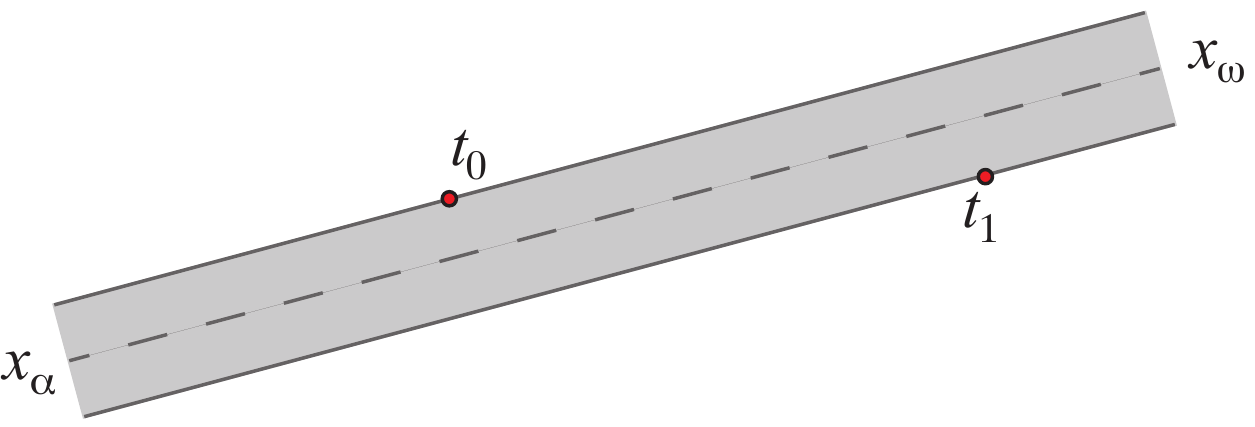}}\qquad\subfigure[Sectors on a disk]{\includegraphics[width=5.5cm]{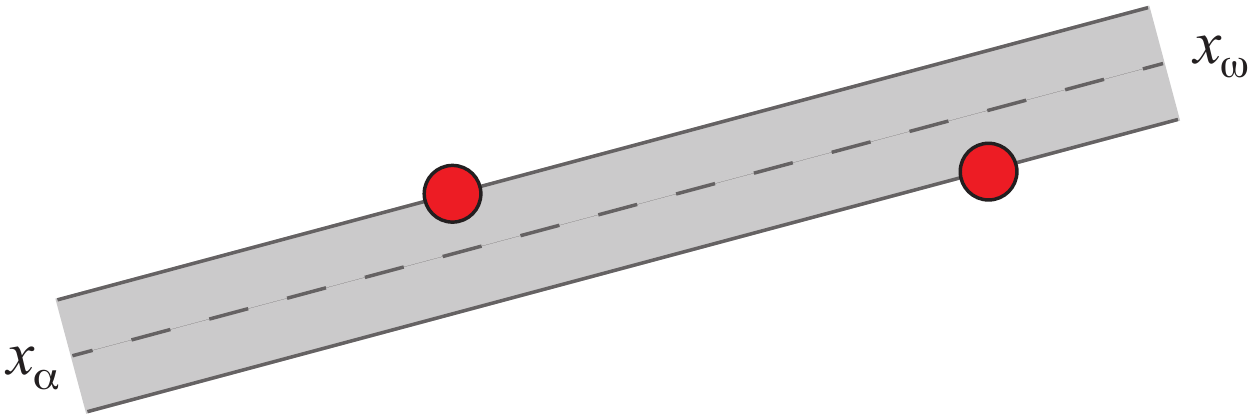}}
\caption{ A slanted strip in $t$-space whose image is the union of sectors $\Omega_{j,\eps,S_s}^+$ and $\Omega_{\sigma(j),\eps,S_s}^-$ in $x$-space. } \label{strip_hor_slant_double} \end{center} \end{figure}

But Figure~\ref{strip_hor_slant_double} is a special case. In general, there is no reason why $x_\alpha$ and $x_\omega$ would turn of the same amount and in the same direction. Hence, once we have passed $t_0$ and $t_1$, we may have to bend the strip again so that it makes an angle $\beta_\alpha$ on the left (resp. $\beta_\omega$ on the right) with the horizontal direction. One would also have to deform in between as several strips come together;  
a precise description of the strips in $t$-space seems a bit complicated, but we will see that this will not be needed.
What is important is that all half-strips having the same $x_\alpha$ or $x_\omega$ limit point have the same slope, thus yielding in $x$-space spirals with same rate of spiraling. Also they must cover a full neighbourhood of the limit point. Let us give the constraint when the limit point is $x_\alpha$. Let $B= \sum_{j\in J} \tau_j$, where $J$ is the set of indices of $\tau_j$ corresponding to strips with limit point $x_\alpha$. Then $B= \pm \frac{2\pi i}{p_\eps'(x_\alpha)}$. 
Hence the sum of the width (in the direction of $B$) of the half strips having limit $x_\alpha$ must be more than $B$  so as to cover a full neighbourhood of $x_\alpha$.\begin{figure} 
\begin{center}
\includegraphics[width=9cm]{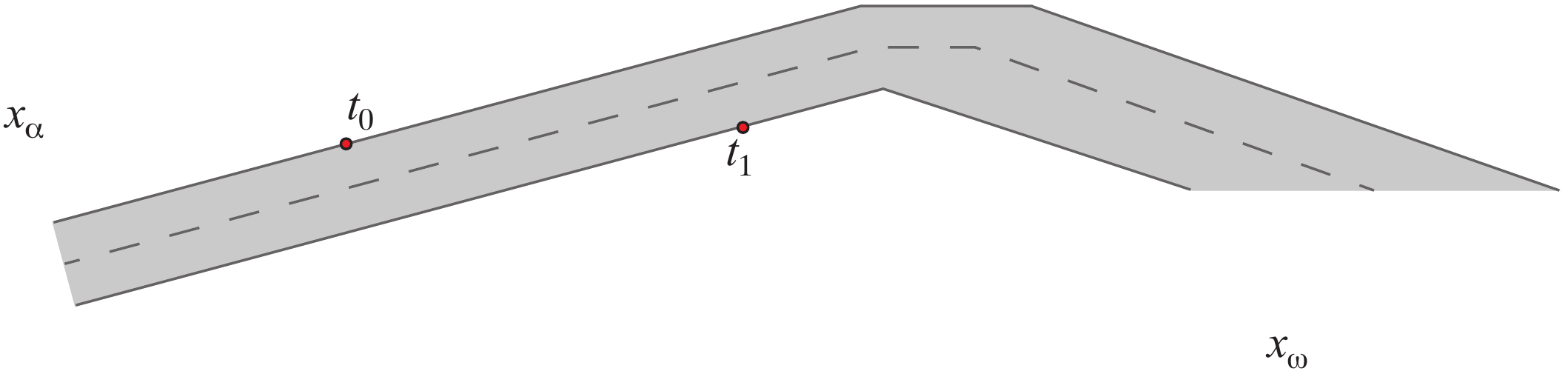}\caption{ A slanted strip in $t$-space whose image is the union of sectors $\Omega_{j,\eps,S_s}^+$ and $\Omega_{\sigma(j),\eps,S_s}^-$ in $x$-space. The slope on the right is chosen so as to approach $x_\omega$. } \label{slant} \end{center} \end{figure}

Let us consider how our space decomposes both in the $t$ and $x$-parametrisations, before deformation of $\widetilde{S}_s$ into $S_s$ (see Figure~\ref{x_t_space}):
\begin{figure}
\begin{center}
\subfigure[$x$-space]{\includegraphics[width=5.5cm]{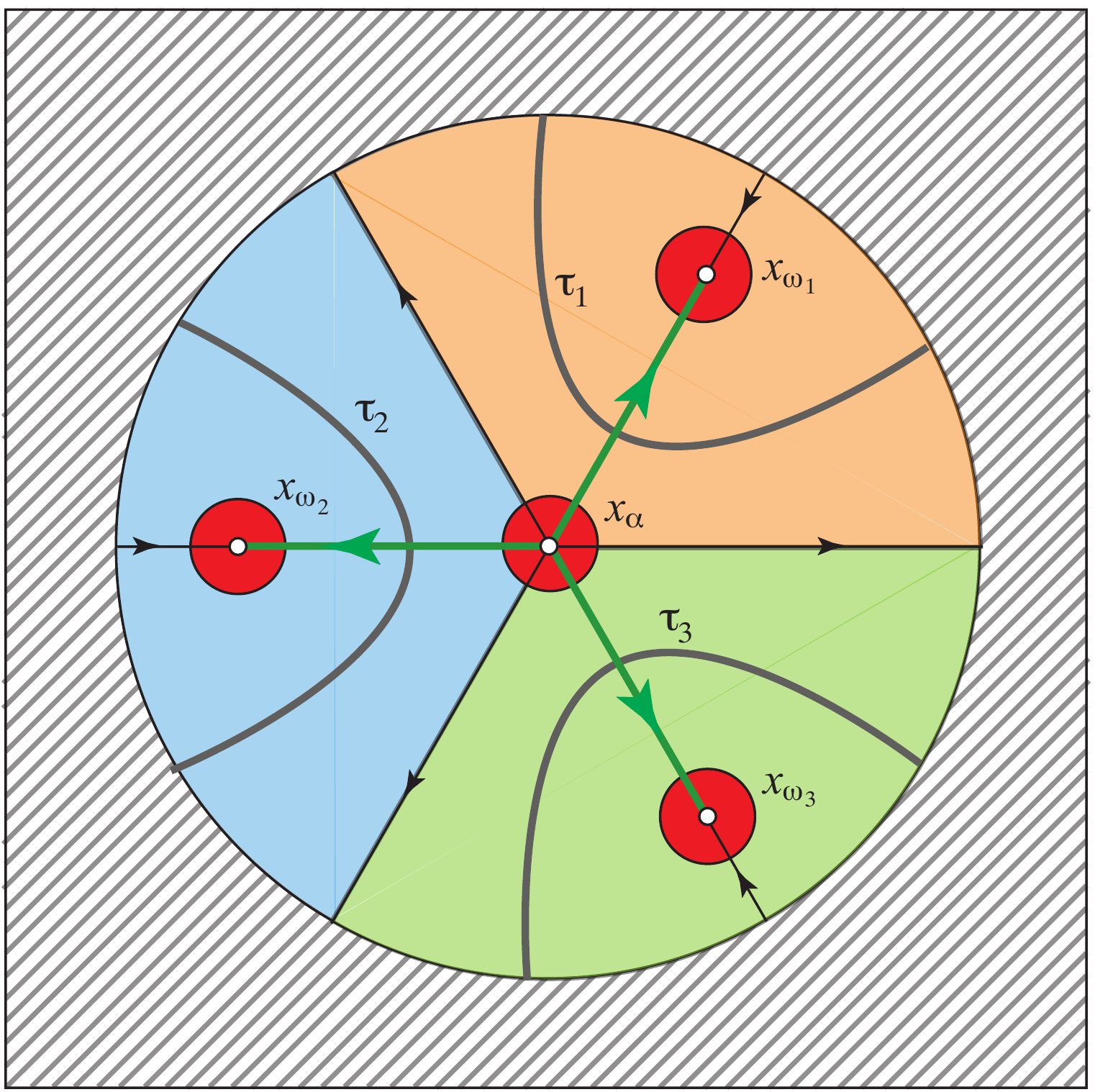}}\qquad\subfigure[$t$-space]{\includegraphics[width=6.2cm]{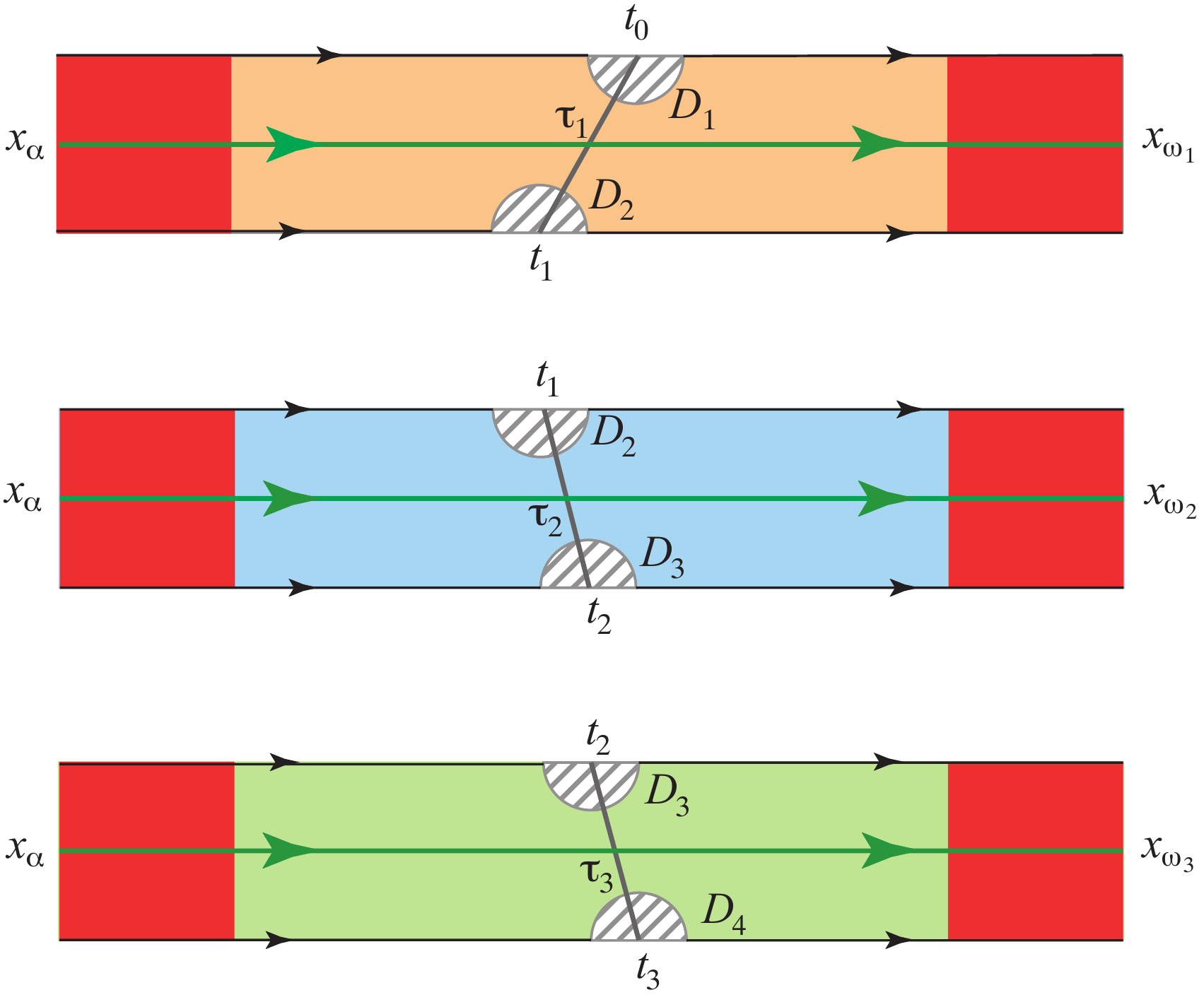}}\caption{The different regions in $x$-space and $t$-space. The DS invariant is in grey and the DS tree in green.}\label{x_t_space}\end{center}\end{figure}

\begin{itemize} 
\item One has points $t_i$ in $t$-space, mapping to $\infty$ in $x$-space. These are surrounded by (diffeormorphs of) disks $D_i$, with $D_i$ minus a slit  mapping to  the intersection of a pair of consecutive sectors ${\Omega}_{i,\eps,S_s}^\pm$ with the set $|x|>R$. One can cover the region $|x|>R$ by a finite number of the $D_i, i=1,\dots, m$, with the lines joining $t_i$ and $t_{\sigma(i)}$ generating the DS invariant in $x$-space. Let all the $D_i$ lie inside a vertical strip $|Re(t)|<C$. 
\item  Taking our singular points $x_\alpha$, $x_\omega$ in $x$-space lying inside the disk $|x|<R/2$, we have  consecutive horizontal half-strips in the region $|Re(t)|>2C $ covering disjoint disks $B_\alpha$, $B_\omega$ in $x$-space around the $x_\alpha$, $x_\omega$, with $Re(t)<-2C$ for an $\alpha$-point and $Re(t)>2C$ for an $\omega$-point. The spiralling in or out of the strips on these disks is important, as indeed is the particular aspect (slope in $t$-space) at which  it is done: indeed, the in or out tells us whether the singular point is attractive or repelling, an important feature of the DS invariant, and the aspect of the lines influences the various growth rates of our solutions to the differential equation. \item Between the two, we have families of horizontal lines in $t$-space joining in $x$-space the $B_\alpha$ to $B_\omega$, with many of them leaving the disk $|x|>R$. The whole family is not important, and indeed all that is necessary is the isotopy class of one of the lines, say $\ell_{\alpha\omega}$ in $x$-space for each horizontal strip, as this is what defines the DS-tree. One has in addition 
 the segments defining the DS-invariant  joining the circle $|x|=R$ to itself, and avoiding the $B_j$. Again it is only the isotopy that is important.
 \end{itemize}
 
 This is the picture on a $DS$ domain   $\widetilde S_s$. We want to extend this picture to a larger set, obtained (as we are on the generic locus of distinct zeroes of $p_\eps$) by moving the individual zeroes of $p_\eps$ around in small balls. As well, we would like to ensure that even on $\widetilde S_s$, we can ensure that the segments defining the DS invariant and the dual DS-tree stay inside. It is clear that this is possible: 
 \begin{itemize} 
\item One leaves the region $|x|>R$ as is.
\item With an isotopy one pulls in the segments of the DS-invariant and the segments of the $DS$ tree so that they lie in the ball $|x|<R/2$. Moving the zeroes of $p_\eps$ around, one can isotope the DS-invariant and the DS tree to follow.
\item One moves the balls $B_\ell$ around with their zeroes, and modifies the slopes of the strips in $t$-space that map into them so as to satisfy our requirement of the points staying attractive or repulsive.
\end{itemize}

This procedure highlights the need to define the new DS domains $S_s$ as  covering open sets in the \'etale sense; that is, as spaces above the complement $\Sigma_0$ of $\Delta = 0$ ($\Sigma_0$ has real codimension $2$), which in fact contain the extra information of the attachment of the singular points to infinity and which retract to   $\widetilde S_s$. This is important when one branches around the locus $\Delta =0$: indeed, while the locus of the zeroes of $p_\eps$ is the same, the pattern of attachment of the separatrices to the singularities can change; likewise, the segments of the DS-invariant or the dual DS-tree will follow the zeroes as they are isotoped. This was brought to the fore in the treatment of the $k=1$ case in \cite{cLR2}; one  sees there  that the projection of $S_s$ to the neighbourhood of $\Delta =0$  self-intersects, but with different attachment to infinity (see Figures~\ref{tubular} and \ref{auto_intersection} later in the paper).  This persists here: for generic points of $\Delta =0$  where exactly two points of $p_\eps = 0$ coincide, a single $S_s$ can be extended in order to cover the neighbourhood of that point in a ramified way. In addition, of course, neighbourhoods of points of $\Delta=0$ inside $\Sigma_0$ can be covered by several $S_s$.

\subsection{Some conventions}

We will be considering our linear o.d.e. problem from a slightly more abstract point of view, as lying on bundles equipped with singular meromorphic connections. As this involves changes of basis and so on, we pause to fix some conventions, some of which do not seem to be entirely uniform in the literature.

{\it A fundamental matrix solution to the o.d.e., or of covariant constant sections.} This will be a matrix $X$, everywhere non degenerate, with 
$$( d-A)X = 0.$$
Then the coordinates of a solution vector, with respect to a basis of solutions is a constant column vector $v$, with 
$$X\cdot v, $$  representing the solution vector in our trivialisation.
Likewise, two fundamental matrix solutions are related by a constant invertible matrix $C$
$$X_1 =X_2\cdot C$$

{\it Changing gauge.} The fundamental matrix solution gets multiplied by $g$ on the left, so that $\hat X = gX$; simultaneously, the trivialisation of the bundle gets acted on by $g^{-1}$ from the right. The matrix is still thought of as  giving the same set of solutions, but given in a different gauge: we have 
$$d\hat X - (dg\cdot g^{-1} + gAg^{-1}) \hat X= 0$$
In particular, taking $g$ = $X^{-1}$ gauges the connection to zero.   

{\it The Stokes matrices.} Recall that the real flow lines at $\eps = 0$ partitioned the Riemann sphere into $2k$ sectors, all meeting at $0$ and infinity. Around the origin or the boundary of a disk the unfolded sectors are ordered counterclockwise, for arbitrary DS domains and $\eps$, as   $ {\Omega}_{1,\eps,S_s}^+$,  $ {\Omega}_{1,\eps,S_s}^-$,  $ {\Omega}_{2,\eps,S_s}^+$,  $ {\Omega}_{2,\eps,S_s}^-$, \dots,  $ {\Omega}_{k,\eps,S_s}^-$. Note that this is the clockwise order if one turns around infinity.

 The theory of irregular o.d.e. gives us on these sectors (limited to a disk), fundamental matrix solutions $X_i^\pm$ on $\Omega_i^\pm$ asymptotic to a diagonal formal solution; on intersections, they are related by Stokes matrices:
$$X_j^+ =X_{j-1}^- \cdot C_{j}^U\quad X_j^- =X_{j}^+ \cdot C_{j}^L,$$
with a cyclicity convention in the indices that $0 = k$. Dually, they can be seen as changes of coordinates in flat trivializations:
$$ v_{j-1}^- = C_{j}^U v_j^+,\quad v_{j }^+ = C_{j}^L v_j^-$$
 (the v's are all the same vector, but in different coordinates).
 
 We will, on each DS domain, be extending these Stokes matrices to {\it generalized Stokes matrices} $C_{j,\eps,S_s}^U$, $C_{j,\eps,S_s}^L$, along with some extra diagonal matrices $C_{j,\sigma(j), \eps}^G$. 
 
{\it Normalization of the Stokes matrices.} Geometrically, the Stokes matrices  are defined up to the action of a certain number of diagonal matrices. Traditionally in the literature, the Stokes matrices $C_{j,\eps,S_s}^U$, $C_{j,\eps,S_s}^L$ are normalized so that their diagonal coefficients are $1$. We rather choose a normalization so that the product of the  inverses of the  Stokes matrices in the appropriate order yields the monodromy matrix along the loop going around all the $k+1$-singular points. One way to achieve this is to take all Stokes matrices $C_{j,\eps,S_s}^U$, $C_{j,\eps,S_s}^L$  with diagonal coefficients equal to $1$, except $C_{1,\eps,S_s}^U$, which incorporates the formal monodromy $\exp\left(-2\pi i \Lambda_k(\eps)\right)$, with $-\Lambda_k(\eps)$ the residue matrix at infinity in \eqref{Normal_form_0}. 

{\it Monodromy.}  We follow the conventions of \cite{yIsY}. The monodromy is defined by taking a fundamental matrix solution $X$, continuing it around a loop $\gamma$, and relating the result    $X_e $  at the end to the beginning $X =X_b$ by $X_e = X_b\cdot M $  , with $M$ a constant matrix. On coordinates $v_b$ (resp. $v_e$) of a solution in a basis given by the columns of $X_b$ (resp. $X_e$), we have $v_b= M \cdot v_e$. This  convention gives a representation of the fundamental group in the sense that 
 $$ M_{\gamma_1*\gamma_2}= M_{\gamma_2}\cdot M_{\gamma_1},$$
where the composition $\gamma_1*\gamma_2$ of two loops is obtained by going through $\gamma_1$ first, then $\gamma_2$. 
Considering the monodromy around the origin, starting from $X_k^-$ on $\Omega_k^-$, the extension of $X_k^-$ on $\Omega_1^+$ is $X_1^+(C_1^U)^{-1}$, the extension of $ X_1^+$ on $\Omega_1^-$ is $X_1^-(C_1^L)^{-1}$\dots,
giving the monodromy around the irregular singularity
 $$M = (C_{k}^L)^{-1}\cdot (C_{k}^U)^{-1}\cdot (C_{k-1}^L)^{-1}\cdot (C_{k-1}^U)^{-1}\cdot...  \cdot (C_{1}^L)^{-1}\cdot (C_{1}^U)^{-1}.$$
This coincides with the conventions of \cite{yIaK}. More generally, if we have, moving {\it from} sector $A$ {\it to} sector $B$, matrices $C_{AB}$ relating flat bases $X_B = X_A C_{AB}$, then if along a loop, if we go through sectors $A_1, A_2, ..., A_k, A_1$ in turn, the monodromy will be  
$$M= C_{A_1,A_k} \cdot C_{A_{k}A_{k-1}}\cdot ...\cdot C_{A_2A_1}.$$

\subsection{Flags of solutions and generalized Stokes matrices}
 
Let us now turn our attention to the vector equation (\ref{vector}):$$\dot y  = A(\eps, x(t))\cdot y.$$
We will in particular be interested in the behaviour of solutions as $x(t)$ goes to the limit points $x_\alpha, x_\omega$, i.e. as $t$ goes to $\pm \infty$ along our real lines. The dominant term of the equation is then $A(\eps, x_\alpha)$ or $A(\eps, x_\omega)$, which we denote $A(\eps,x_*)$. We then have the approximate equation: 
$$\dot y  = A(\eps, x_*)\cdot y. $$
As we are deforming from a generic irregular singularity, we can suppose
$$A(\eps, x_*) = {\rm diag} (\lambda_1(\eps), \dots, \lambda_n(\eps)),$$ and that \begin{equation}{\rm Re} \, \lambda_1(\eps)> {\rm Re} \, \lambda_2(\eps)>\dots> {\rm Re} \, \lambda_n(\eps).\label{order_lambda}\end{equation}
 
 The approximate equation has solutions along each real line $y^i(t) = (0,..,0, \exp(\lambda_i t),0,...0)$. The solution $y^{i+1}$ grows  exponentially more slowly than $y^i$ as $t\rightarrow +\infty$, and grows exponentially faster as $t\rightarrow -\infty$. We thus have two flags of subspaces  in the solution space, defined by growth rates: the first, given by behaviour as $t\rightarrow +\infty$, defined by $W_i(x_\omega) = <y^{n-i+1},..,y^n>$, and the second, given by behaviour as  $t\rightarrow  -\infty$, defined by $W_i(x_\alpha) = <y^{1 },..,y^i>$, with 
 $$ W_1(x_\omega)\subset W_2(x_\omega)\subset...\subset W_n(x_\omega),$$
 $$W_n(x_\alpha)\supset W_{n-1}(x_\alpha)\supset...\supset W_1(x_\alpha).$$
  These two flags are transversal for sufficiently small $\eps$ (see Lemma~\ref{transverse} below):
  $$W_i(x_\omega)\cap W_{n-i}(x_\alpha) = 0.$$
  
  At a generic singular point, for generic $\epsilon$, the actual behaviour of solutions is given precisely by the formulae above: the system admits a gauge transition to a normal form. Such a gauge transformation is not, however necessary, and indeed these flags are remarkably robust as one varies in $\eps$. One can see that  this is reasonable in that the flag manifold is compact, and so there  is at least a convergent subsequence as one moves around in $\eps$. In any case, the existence of the flag, and its continuity as one varies $\eps$ (even moving out of $\Sigma_0$) is given by the following theorem of Levinson \cite{Lev} (see Coddington and Levinson \cite{CL}, Theorem 8.1, p. 92; see also the discussion after Theorem 5.3 of \cite{HLR}).

\begin{theorem}\label{thm:CD} Let a system of linear differential equations of the form 
 \begin{equation}
\dot{y}=\left(\widetilde{\Lambda}_0(\eps) + \widetilde{\Lambda}(\eps,t)  + P(\eps,t) \right)\cdot y,\\
\end{equation}
be given on the real line, for which $\widetilde{\Lambda}_0$ is diagonal, with distinct real parts of the eigenvalues, $\widetilde{\Lambda}(\eps,t) $ is also diagonal, with limit zero at $t= \infty$ 
 and 
 \begin{equation}\label{intbounds} \int_0^\infty \left| \frac{d}{dt}(\widetilde{\Lambda}(\eps,t))\right| dt <\infty, \quad \int_0^\infty \left| P(\eps,t)\right| dt <\infty.\end{equation}
Then, setting $\lambda_\ell(\eps,t)$, $\ell=1, \dots,k$, to be the successive eigenvalues of  $\widetilde{\Lambda}_0(\eps) + \widetilde{\Lambda}(\eps,t) $, there exist $t_0\in (0,\infty)$ and  solutions $\phi_{\ell,\eps}(t)$ of the system for $t\in (t_0,\infty)$ with 
$$\lim_{t\rightarrow \infty} \phi_{\ell,\eps}(t)\cdot \exp \left(- \int_{t_0}^t \lambda_\ell(\tau) d\tau\right) = v_\ell(\eps),$$
for $v_\ell(\eps)$ a non-zero eigenvector of $\widetilde{\Lambda}_0(\eps)$ corresponding to $\lambda_\ell(\infty)$. If $\widetilde{\Lambda}_0(\eps)$, $\widetilde{\Lambda}(\eps,t)$ and $P(\eps,t)$ depend continuously (resp. analytically) on $\eps$ over compact sets in $t$-space,  with the integrals in \eqref{intbounds} uniformly bounded, then the solutions can be chosen depending continuously (resp. analytically) on $\eps$.
\end{theorem}

In the cases of concern to us, the flags are varying continuously, and indeed analytically within the DS domains. In particular, we have

\begin{lemma} \label{transverse}
For $\eps$ small within a DS domain, the flags $W_i(x_\omega)$, $W_{n-i}(x_\alpha)$ are transverse in each strip. 
\end{lemma}
\begin{proof} The flags are transverse at $\eps=0$. Indeed, the standard results of the theory of irregular singularities (see for instance \cite{yIsY})
give us on each sector a basis of solutions asymptotic to the standard basis of solution of the formal normal form; in these bases, the flags  $W_i(x_\omega)$, $W_{n-i}(x_\alpha)$ are simply the standard flags, and indeed are transverse.\end{proof}

Next, within a sector (i.e., a strip in $t$-space), the flags are constant, i.e. independent of the real flow line chosen. The reason is that one is solving an o.d.e. on the complex plane; one has a flat connection. Thus, to get the asymptotics of the flow lines on $Im(t) = c_1$, one can solve along $Im(t) = c_0$, and then just integrate in the imaginary direction "out near infinity", on  a segment of length $|c_1-c_0|$, where the coefficients of the equation are essentially constant, and the equation preserves the flag.

One thing to which the flags are sensitive, however, is the asymptotic direction of our flow lines in $t$-space. Indeed, if one just looks at the formal normal form, and takes the solutions $y^i(t) = (0,..,0, \exp(\lambda_i t),0,...0)$ for $t=e^{i\theta} s, s\in \R^+$, it is the ordering of  the real parts of $\lambda_i e^{i\theta}$ (assuming they are distinct ) which determines the flag. Unless $\theta$ is suitably small, this is not the same as the ordering of the real parts of $\lambda_i$. In particular, we must impose the constraint that the deformation angles be kept small in our construction of the extensions of the DS domains in the previous section. This is possible as long as we keep the parameters of our construction within a suitable range, which we can. 

\bigskip

\noindent{\it Flags and bases; generalized Stokes matrices}. On each sector, one has two transverse flags  $W_i(x_\omega)$, $W_{j}(x_\alpha)$, which gives one dimensional intersections
$$W_i(x_\omega)\cap W_{n-i+1}(x_\alpha) = <e_i>.$$
and so a basis $_Atr=<e_1,e_2,...,e_n>$ defined up scale, i.e.  to the action of the diagonal matrices. Now let us consider two consecutive sectors, sharing a separatrix; as such, they also share a flag. For instance, 
$\Omega_{j,\eps}^+ $, $\Omega_{j-1,\eps}^-$   share a separatrix   adherent to a unique singular point  $x_\alpha$  of $\alpha$-type; along these we have a shared flag $W_i(x_\alpha)$, and so their bases   $_Atr_{j-1}^-$ on $\Omega_{j-1,\eps}^-$ and  $_Atr_{j}^+$  on  $\Omega_{j,\eps}^+$ are related by an upper triangular matrix $C_{j,\eps,S_s}^U$: 
$$_Atr_{j}^+= {_Atr_{j-1}^-}\cdot C_{j,\eps,S_s}^U.$$
Similarly,   $\Omega_{j,\eps}^+$ and $ \Omega_{j,\eps}^-$ share a separatrix  adherent to a unique singular point $x_\omega$ of $\omega$-type; along these we have shared flag $W_{j}(x_\omega)$ and so their bases are related  by a lower triangular matrix $C_{j,\eps,S_s}^L$:
$$_Atr_j^- = {_Atr_{j}^+} \cdot C_{j,\eps,S_s}^L$$
Finally, consider $\Omega_{j,\eps}^+$ and $\Omega_{\sigma(j),\eps}^-$, where $\sigma$ is the (order two) permutation representing the Douady-Sentenac invariant;  these share  two singular points, one of $\alpha$-type, one of $\omega$-type. The sectors $\Omega_{j,\eps}^+$ and $\Omega_{\sigma(j),\eps}^-$ then  share both flags, and so the matrix $C_{j,\sigma(j),\eps}^G$,  relating the adapted bases, called \emph{gate matrix}, will be  diagonal:
$$_Atr_{\sigma(j)}^- = {_Atr_{j}^+} \cdot C_{j,\sigma(j),\eps}^G$$
These matrices will be our {\it generalised Stokes matrices and gate matrices.} 

\subsection{Normalisations} We remark that there is freedom on each sector of an action by the diagonal matrices, as the intersections of the flags do not specify a scale. The Stokes matrices $C_{AB}$ associated to passing from sector A to sector B then get acted on by $D_AC_{AB}D_B^{-1}$. This is somewhat different from the classical Stokes matrices, which are usually defined as being strictly upper triangular, i.e with ones on the diagonal; this arises from the asymptotics of the equation.  Here we allow an arbitrary diagonal, but then quotient by the action of the diagonal matrices. We note then that our Stokes matrices contain the information of the formal monodromy around the singularity. 

Various normalisations are possible.  One would like one such 
that  the Stokes matrices $C_{j,\eps,S_s}^L$ and $C_{j,\eps,S_s}^U$  have limits when the confluence of singular points occur. We choose a normalisation that accomplishes this by taking all the diagonal parts of $C_{j,\eps,S_s}^L$ and $C_{j,\eps,S_s}^U$ equal to the identity, except for the diagonal part $C_{j,\eps,S_s}^U$ which is taken as $\exp(-2\pi i \Lambda_k(\eps))$; this maintains the compatibility of our Stokes matrices with the formal monodromy.  On the DS domains, one also needs that our Stokes matrices  carry the monodromy of the formal normal form at each regular singular point; the product  of the inverses of the diagonal parts of a sequence of the $C_{j,\eps,S_s}^L$  or $C_{j,\eps,S_s}^U$ and of the $C_{j,\sigma(j),\eps}^G$ corresponding to the sectors touching the singular point should be the formal monodromy. For example,  if  all the diagonal parts of the $C_{j,\eps,S_s}^L$, $C_{j,\eps,S_s}^U$ corresponding to the singular point are the identity, and there is only one gate matrix, then that gate matrix or its inverse must be the formal monodromy. In general, as the singular points are arranged in a tree, one sees that one may inductively move down the tree to compute the gate matrices in terms of the formal normal form.  These normalised gate matrices $C_{j,\sigma(j),\eps}^G$  often have no limit when $\eps\to 0$, though the singularity disappears when we pass from this Stokes matrix picture A  to picture B, as explained in Section~\ref{three-pictures} below.

{\it In what follows, we will be referring to normalised (generalised) Stokes matrices  $C_{j,\eps,S_s}^L$ and $C_{j,\eps,S_s}^U$; we will suppose that the concomitant gate matrices have been computed from these and from the formal normal form.}

\subsection{ Building from the formal normal form; bundles on the disk or on the plane and Stokes matrices} \label{three-pictures}
We can consider our equations in a slightly more abstract form, as bundles equipped with singular (meromorphic) connections over the complex line. This allows us to choose different open sets covering the line, and to trivialize the bundle differently on these open sets. In our construction, we will consider as basic open sets, for $\eps\in S_s$, the sets $ \Omega_{j,\eps,S_s}^\pm$. There are three basic pictures we want to consider: 

{\it A) Trivial connections $d$ on each $ \Omega_{j,\eps,S_s}^\pm$, and constant transition matrices, defining a bundle on a punctured domain (either $\C\setminus\{x_1, \dots x_{k+1 }\}$ or $\D\setminus\{x_1, \dots x_{k+1 }\}$).} 

Here we have on each sector $ \Omega_{j,\eps,S_s}^\pm$ a bundle with  the trivial connection, and transition functions which will be automorphisms of the trivial connection (i.e. constant matrices; these will be the Stokes matrices). The bundle on each sector is equipped with two  constant  flags  $ W_i(x_\alpha)$  associated to the $\alpha$- singularity  adherent to $ \Omega_{j,\eps,S_s}^\pm$, and  $ W_i(x_\omega)$   associated to the $\omega$- singularity  adherent to $ \Omega_{j,\eps,S_s}^+,  \Omega_{j-1,\eps,S_s}^-$.  One has (covariant constant)  trivializations compatible with the flags, in the sense that 
the $i$-th element of the basis lies in the intersection of $W_i(x_\alpha)$ and $W_{n-i+1}(x_\omega)$.

As one moves from sector to sector, the bundles are glued by constant transition functions, which are   our generalized Stokes matrices. The intersections of two sectors, which we call \emph{intersection sectors}, are of three types:
\begin{itemize}
\item Sectors $\Omega_{j,\eps}^U$, the intersection of $\Omega_{j,\eps}^+\cap \Omega_{j-1,\eps}^-$,  containing a separatrix emerging from $x_\alpha$, a point  of $\alpha$-type; along these we have   a   trivialisation $_Atr_{j-1}^-$ on $\Omega_{j-1,\eps}^-$ and trivialisation $_Atr_{j}^+$  on  $\Omega_{j,\eps}^+$, related by the  upper triangular matrix $C_{j,\eps,S_s}^U$: 
$$_Atr_{j}^+= {_Atr_{j-1}^-}\cdot C_{j,\eps,S_s}^U.$$

\item Sectors $\Omega_{j,\eps}^L$, the intersection of $\Omega_{j,\eps}^+\cap \Omega_{j,\eps}^-$, containing a separatrix    converging to $x_\omega$, a point of $\omega$-type;  along these we have trivializations related by  a lower triangular matrix $C_{j,\eps,S_s}^L$:
$$_Atr_j^- = {_Atr_{j}^+} \cdot C_{j,\eps,S_s}^L$$
\item Gate sectors $\Omega_{j,\sigma(j),\eps}^G$, the intersection $\Omega_{j,\eps}^+\cap \Omega_{\sigma(j),\eps}^-$, where $\sigma$ is the permutation representing the Douady-Sentenac invariant;  the gate sector is adherent to two singular points, one of $\alpha$-type, one of $\omega$-type. The sectors $\Omega_{j,\eps}^+$ and $\Omega_{\sigma(j),\eps}^-$ intersecting in the gate sector share both flags, and the trivializations are related by a diagonal $C_{j,\sigma(j),\eps}^G$ :
$$_Atr_{\sigma(j)}^- = {_Atr_{j}^+} \cdot C_{j,\sigma(j),\eps}^G$$
 
\end{itemize}

 {\it B) Diagonal connections $\nabla^D = d - \frac{\Lambda(\eps,x)}{p_\eps(x)}$ on each $ \Omega_{j,\eps,S_s}^\pm$, and transition matrices tending to the identity at the singular points, defining a bundle on the   plane $\C $.} This picture is obtained from picture $A$  by applying on each sector  the gauge transformation given by the diagonal fundamental matrix of solutions $F_{j,\eps}^\pm$ to the {\it formal} diagonal normal form, (\ref{normal_form_eps}). (The fundamental solution given above is multivalued, and so one can choose different determinations on each sector.) The $F_{j,\eps}^\pm$ preserve on each sector the flags $W_i(x_\alpha)$ and $W_i(x_\omega)$, which now however have geometric meaning, as they correspond to the decay rates of solutions along lines $e^{i\theta}t, t\in \R^+$. The transition matrices gauge transform to automorphisms of $d - \frac{\Lambda(\eps,x)}{p_\eps(x)}$. Along our intersection sectors:
 \begin{itemize}
\item In $\Omega_{j,\eps}^U$, we have matrices $\widetilde C_{j,\eps,S_s}^U =F_{j-1,\eps}^-C_{j,\eps}^U (F_{j,\eps}^+)^{-1}$ representing a change of trivialisation from $\Omega_{j-1}^-$ to $\Omega_j^+$:
\begin{equation}_Btr_{j}^+= {_Btr_{j-1}^-}\cdot \widetilde C_{j,\eps,S_s}^U.\label{Btriv1}\end{equation}
These preserve the flag $W_i(x_\alpha)$,   are upper triangular, with constant diagonal terms and decaying off diagonal terms as one goes to the $\alpha$-type singular point. 
\item In $\Omega_{j,\eps}^L$, we have matrices $\widetilde C_{j,\eps,S_s}^L=F_{j,\eps}^+C_{j,\eps}^L (F_{j,\eps}^-)^{-1}$: 
\begin{equation}_Btr_j^- = {_Btr_{j}^+} \cdot \widetilde C_{j,\eps,S_s}^L.\label{Btriv2}\end{equation}
These preserve the flag $W_i(x_\omega)$,   are lower triangular, with constant diagonal terms and decaying off diagonal terms as one goes to the $\omega$-type singular point. 
\item In $\Omega_{j,\sigma(j),\eps}^G$, we set  $\widetilde{C}_{j,\sigma(j),\eps}^G= F_j^+ C_{j,\sigma(j),\eps}^G (F_{\sigma(j)}^-)^{-1}$:
\begin{equation}_Btr_{\sigma(j)}^- = {_Btr_{j}^+} \cdot \widetilde C_{j,\sigma(j),\eps}^G\label{Btriv3}\end{equation}
 is diagonal. \end{itemize}
With our normalisation, it is natural to take the $F_{j,\eps}^\pm$ as extensions of $F_{1,\eps}^+$ when starting in $\Omega_{1,\eps,S_s}^+$ and crossing all $\Omega_{j,\eps,S_s}^\pm$ counterclockwise on a circle around the singular points. Then the gate matrices $C_{j,\sigma(j),\eps}^G$ exactly compensate for the ramification of $F_{1,\eps}^+$ and $\widetilde C_{j,\sigma(j),\eps}^G=\rm{id}$. All our transition matrices have finite limits at the punctures, and so the bundle extends to the punctures.

{\it C) A globally defined singular connection $d- \frac {A(\eps,x)}{p_\eps(x)} $, with only one trivialization over a disk or $\C$; the formal normal form of $\frac {A(\eps,x)}{p_\eps(x)} $ is $\frac{\Lambda(\eps,x)}{p_\eps(x)}$.} Again, associated to each limit point $x_\alpha$ or $x_\omega$ there  will be (covariant constant) flags $W_i(x_\alpha)$,  $W_i(x_\omega)$ on each sector defined by growth rates as one goes to the singular points along the paths ${\rm Im}(e^{i\theta}t)= $ constant. In the $x$-plane, these paths are generically logarithmic spirals. On each sector, there is then  a pair of (covariant constant) flags $W_i(x_\alpha)$,  $W_i(x_\omega)$. The flags  are transversal for sufficiently small $\eps$, by Lemma  \ref{transverse}. This gives a unique (up to the action of diagonal matrices) basis of flat sections  on each sector, with the $i-$th element of the basis living in $W_i(x_\alpha) \cap W_{n-i+1}(x_\omega)$. We denote the fundamental matrix solution on each sector by $X_{j,\eps,S_s}^\pm$.

Points of view A) and B) are easily seen to be equivalent; the point of this paper is to show that they are equivalent to C). As noted above, the gauge transformation from A) to B) is given by the $F_{j,\eps}^\pm$. Relating  A) and  C), one has the gauge transformation $X_{j,\eps,S_s}^\pm$. The gauge transformation from C) to B) is then  $H_{j,\eps,S_s}^\pm = F_{j,\eps}^\pm (X_{j,\eps,S_s}^\pm)^{-1}$.

The monodromy of the connection on a path $\gamma$ around several singular points is given in different ways: in version C), one integrates the connection, as usual, and applies the convention above. In version A) it is  given by the product  of matrices  $ C_{j,\eps,S_s}^U ,  C_{j,\eps,S_s}^L $, and $ C_{j,\sigma(j),\eps}^G $ or their inverses, taken in the reverse order one meets the corresponding intersection sectors as one moves along $\gamma$. In version B), it is a hybrid, a product in the right order of the matrices $\widetilde C_{j,\eps,S_s}^U, \widetilde C_{j,\eps,S_s}^L, \widetilde C_{j,\sigma(j),\eps}^G $  or their inverses   and of the parallel transports by the diagonal connection on the intersection of $\gamma$ with each sector.

\section{The realization over a DS domain} 

In this section we often drop the index $S_s$, as we will be working with a fixed DS domain. Also, we are  concerned with $\eps$ small, so in fact we only need to push through our constructions on the intersection of $S_s$ with a polydisk $\D_\rho$ around the origin in $\eps$-space.

\subsection{A bundle with connection at $\eps=0$}
For fixed $\eps\in S_s$, the passage from versions A) or  B) above to version C) is fairly well known, and the construction depends analytically on $\eps\in S_s$. It amounts to   finding the necessary gauge transformations on each sector to make the  cocycles $C_{j,\eps }^U, C_{j,\eps }^L, C_{j,\sigma(j),\eps}^G$ trivial.  Various techniques do this; we refer to \cite{yIsY}.  This realizes the systems locally over a DS domain. However, we will want to glue these realizations over DS domains to obtain a global family of systems for $\eps$ in a neighbourhood $\D_\rho$ of $0$. This glueing involves using the action of the gauge group to glue the systems. Gauge transformations over $\C$ form an infinite dimensional group; we would like to reduce the degrees of freedom somewhat, and we do this by compactifying.

One would hope to realize the systems over $S_s$ as singular Fuchsian connections on a trivial bundle over $\CP^1$, by adding in an extra singularity at infinity carrying the required monodromy. This is generically possible, but not always, as was shown by Bolibruch \cite{Bo} and Kostov \cite{K2}. On the other hand, if we allow two singularities, then we can do it, at least for $\eps$ in a small neighbourhood of the origin.

We first realize the system for $\eps= 0$ as a system  on a trivial bundle over $\CP^1$ with an irregular singularity at the origin, and two Fuchsian singularities, one at infinity, and one at some point at a large distance $R$ from the origin, along the positive axis. This will reduce the gauge transformations to constant matrices in $Gl(n,\C)$. We would like the system to be rigid, in a suitable sense:

\begin{definition} A system of linear differential equations $y'=A(x)\cdot y$ is \emph{indecomposable} if  it cannot be gauge transformed to a block diagonal form. \end{definition}

\begin{definition} A bundle with connection is \emph{reducible} if it admits a nontrivial subbundle invariant under the connection. Otherwise, it is \emph{irreducible}, i.e. the connection cannot be conjugated to a block triangular form. It is then also indecomposable.\end{definition}
 
\begin{remark}  We show in Lemma~\ref{normalization_indecomposable} below that any indecomposable connection can be normalized to a unique normal form. Then the same will follow for an irreducible bundle with connection. \end{remark} 

\begin{remark}\label{Basepoint} {\bf Choice of base point and trivialization.} 
 We will consider the monodromy of the connection at a base point $x_b=3R/4$ , with some paths $\ell_R, \ell_\infty$ around the singularities at $R$, $\infty$, whose product $\ell_R\ell_\infty$ is homotopic to an anticlockwise loop around the origin.     We choose as trivialization of the bundle the canonical coordinate on  $\Omega_{k,0}^-$ in which the connection (in picture $A$) is the normal form.  We will choose changes of trivialization (to pictures $B$ and $C$)  which are the identity at the base point, so that the monodromy computed from this base point stays the same.\end{remark}
 
  With respect to a global  trivialization (picture $C$), we will obtain a connection of the form
\begin{equation} y' = \left(\frac {A_0  + A_1 x +...+ A_k  x^k}{x^{k+1}} + \frac{\widehat{A}_R}{x-R}\right)\cdot y= \frac{B(x)}{x^{k+1}(x-R)}\cdot y.\label{nabla-global}\end{equation}
Let $\widehat A_\infty= -A_k - \widehat{A}_R$ be the residue matrix at infinity (which vanishes when $\infty$ is a regular point).

\begin{theorem}\label{thm:eps_0} Suppose given \begin{itemize} 
\item formal invariants given by diagonal matrices $\Lambda_0, \dots \Lambda_k$ such that $\Lambda_0$ has distinct eigenvalues satisfying \eqref{order_eigenvalues}, determining a formal normal form \eqref{normal_form_eps},
\item  normalized invertible (Stokes) upper (resp. lower) triangular matrices $C_{j}^U$, (resp. $C_{j}^L$), $j=1, \dots, k$ determining a monodromy \begin{equation}
 M=(C_{k,0}^L)^{-1}\cdot(C_{k,0}^U)^{-1}\cdot\dots\cdot (C_{1,0}^L)^{-1}\cdot(C_{1,0}^U)^{-1}; \label{monodromy_zero}\end{equation} 
 \item a generic matrix $M_\infty$ representing a conjugacy class  with distinct eigenvalues, with all entries nonzero, close to the identity,
 and such that $M_\infty\cdot M^{-1}$ has distinct eigenvalues. 
\end{itemize}
Then  there exists a globally trivialized  irreducible rational linear differential system \eqref{nabla-global} on $\CP^1$
 with \begin{itemize}
\item formal normal form \eqref{Normal_form_0} at the origin,   
\item  Stokes matrices  $C_{j}^U$ and $C_{j }^L$, $j=1, \dots, k$,
\item monodromies in the global trivialization $ M(\ell_R)$, $M(\ell_\infty)=M_\infty $  around $\ell_R, \ell_\infty$, both with distinct eigenvalues.

\end{itemize}

The automorphisms of the system are multiples of  the identity. Furthermore, acting  by an automorphism of the formal normal form (the constant diagonal matrices), which transforms both the Stokes data and our rational system, it is possible to normalize the coefficients of $n-1$ suitable  monomials of  entries of  $B(x)$ (or of $n-1$ non diagonal entries of either the monodromies $ M(\ell_R)$, $M(\ell_\infty)$) to $1$, thus leading to a unique normalization of \eqref{nabla-global} for each equivalence class under automorphisms of the Stokes data.
\end{theorem} 
\begin{proof} 
 Let us first build a bundle  in picture $B$ of Section~\ref {three-pictures} above, in the background of the diagonal connection. We take our bundle on a neighbourhood of the origin, defined on sectors $\Omega_{j,0}^\pm$, with transition matrices $\widetilde C_{j,0}^U, \widetilde C_{j,0}^L$.  By the classical result going back to Birkhoff \cite{B}, this gets realized in picture $C$ as a   singular connection of Poincar\'e rank $k$ on a disk $\D_{R}$ of radius $R$.  

From the base point $x_b= 3R/4$  with the trivialization of the bundle fixed above, the monodromy $M$ of the connection on the  circle of radius $3R/4$ around the origin starting in $\Omega_{k,0}^-$ is given by \eqref{monodromy_zero}. Now, build a holomorphic connection on the annulus $\A_r=\{r\in (R/2, 3R)\}$, with a singularity of Fuchsian type at $x= R$, whose monodromy along the circle $r= \frac{3R}{4}$ (resp. $r= 2R$) is $M $ (resp. the identity). (Note that it could happen that we start with $x=R$ regular when $M$ is the identity.) We glue this to our singular connection on the disk of radius $R$. The result is a bundle with a connection over the disk of radius $3R$, with an irregular singularity at the origin, and another Fuchsian singularity at $x=R$. It has trivial monodromy around the boundary. Now glue this to the bundle on the disk $\Omega_\infty=\{r>2R\}$ with the trivial connection. The result is a global bundle with connection on $\CP^1$ with two singularities, one at the origin, and one at $x=R$. Any bundle on $\CP^1$ decomposes as a sum of line bundles; let this one have holomorphic type ${\mathcal O}(a_1)\oplus {\mathcal O}(a_2)\oplus...\oplus {\mathcal O}(a_n) $. 

We now switch trivializations for a while, to the standard trivializations on $\oplus_{i=1}^n \mathcal O(a_i), a_1\leq a_2\leq ...\leq a_n$. Let  $\mathcal{A}$ be our connection  in this trivialization, and let $z=\frac1x$. Let $S={\rm Im}(x^{a_i})=diag(z^{-a_i})$ be the transition from  $U_\infty$, a neighborhood of $\infty$ (resp. $0$) in $x$-space (resp. $z$-space), to $U_0 = \{x \neq \infty\}$. On $U_\infty$, the connection is represented by a holomorphic matrix $A(z)=(A_{ij}(z))$. A Schlesinger transformation on $U_\infty$ given by $S = {\rm Im}(z^{-a_i})$ modifies the bundle to the trivial one, modifies the connection matrix by  $ SAS^{-1} +dS\ S^{-1}$, and changes the off-diagonal terms of the connection to $ SAS^{-1}$, i.e, $A_{ij}$ changes to $A_{ij}z^{a_j-a_i}$, thus introducing poles potentially of order greater than $1$ below the diagonal if $A_{i,j}$ is suitably generic.

We thus want to normalise the $A_{ij}$ by killing the terms of order less than $a_i-a_j$ in $z$ in the lower triangular terms $A_{i,j}$  {\it before} conjugating by $S$.

To do this we use the automorphisms $g$ of $\oplus_i \mathcal O(a_i)$. These are also given by invertible lower triangular matrices, which in the $U_\infty$ trivialization are polynomials, with $g_{ij}$ of degree  $a_i-a_j$, representing a section of ${\rm Hom}(\mathcal{O}(a_j), \mathcal{O}(a_i))$. Indeed, because of the constraint on the degrees of the $g_{ij}$, then $SgS^{-1}$ is again a lower triangular invertible polynomial matrix in $x$. 

Let us consider automorphisms $g(z) $ such that $g(0)=\mathrm{Id}$, i.e. $g(z)=\mathrm{Id} +N(z)$, with $N(z)=O(z)$ lower triangular nilpotent. The automorphism changes the connection by
 $$ A\mapsto B=gAg^{-1} + g'g^{-1}$$
One  wants to get rid of terms in $B_{ij}$ of degree less than $m_{ij}=a_i-a_j$ for $i>j$; in other words, solving for some of the Taylor series of the differential equation above. 

Since $N$ is nilpotent, then $g^{-1}(z) = \mathrm{Id} -N(z)+ N^2(z) - \dots +(-1)^{n-1}N^{n-1}(z)$. Let $B(z)=(B_{ij}(z))$, $N(z)= (g_{ij}(z))$ and $E(z)=g(z)A(z)+ g'(z)=(E_{ij}(z))$. And let $B_{ij}^k, g_{ij}^k, A_{ij}^k, E_{ij}^k$ be the terms of order $k$ in $z$ of $B_{ij}(z)$, $g_{ij}(z)$, $A_{ij}(z)$, $E_{ij}(z)$ respectively. Then 
$E_{ij}^k= A_{ij}^k + (k+1)g_{ij}^{k+1}+ Q_{i,j}^k(A,g),$ and so
$$B_{ij}^k= A_{ij}^k + (k+1)g_{ij}^{k+1}+ \ov{Q}_{i,j}^k(A,g).$$
where $Q_{ij}^k$ and $\ov{Q}_{i,j}^k$ are polynomials in the $A_{\ell,\ell'}^{j'}$ and $g_{\ell,\ell'}^{j'}$, with $j'\leq k$.   Hence we can solve successively degree by degree the equations corresponding to asking that the terms of $B_{ij}$ of degree $m$ be equal to zero, starting from $m=0$; solving for $ B_{ij}^k$ ``uses up'' $g_{ij}^{k+1}$; the form of the equations is such that all equations corresponding to putting some terms of degree $m$ in some $B_{ij}$ with $i>j$ are uncoupled, and hence can  be solved independently for a fixed degree. 

 Applying then the Schlesinger transformation $S$ to $B$ makes the bundle trivial and introduces a Fuchsian singularity (via the $S^{-1}dS$ term) at $\infty$.

\begin{remark}
 We could have achieved the same purpose by using a polynomial gauge transformation on $U_0$ and the  \emph{permutation lemma} of Bolibruch (Lemma 16.36 of \cite{yIsY}), so as to bring the singular point at infinity to be Fuchsian. \end{remark}

\begin{remark}
 Note that, generically, if we choose the right residues for our Fuchsian singularity at $R$, the bundle is already trivial and we do not need to perform the Schlesinger transformation. In that particular case, the residue matrix at infinity $A_{\infty,0}$ vanishes.  
\end{remark}

We have built a bundle on $\CP^1$ with two sets of trivializations: the first, version B, has as open sets the sectors and  the disk $\D_\infty=\{|x|>R/2\}$; it has the transition matrices $\widetilde C_{j,0}^U, \widetilde C_{j,0}^L$ and $\Gamma$ between $\Omega_{k,\eps}^-$ and the annulus $\A_R$. In this trivialization, the monodromy around $R$ is $M$, and  the monodromy around infinity  is the identity. One also has a global trivialization, picture $C$, where our connection at $\eps =0$ is of the form \eqref{nabla-global}; we can choose this trivialization so that the leading term of the connection is $\Lambda_0$. So far, though, the monodromy around infinity is still trivial even though infinity can be a singular point, and the bundle-connection pair might have non-trivial automorphisms.

We would like  
 \begin{itemize} 
 \item that our pair (bundle, connection) be irreducible, since then a further normalization will allow bringing it to a unique form (see Lemma~\ref{normalization_indecomposable} below);
 \item that the point at $\infty$ have diagonalizable monodromy with distinct eigenvalues;
 \item that the singular point at $x=R$ have diagonalizable  monodromy with distinct eigenvalues.\end{itemize} 
 To do this, we deform the connection in picture $B$, keeping the same Stokes matrices, but modifying  the monodromy around infinity from trivial to $M_\infty$, while modifying the monodromy around $R$   in the opposite direction, so that  the monodromy along the circle of radius $R/2$ stays constant. This uses Lemma  \ref{Lemma:decomposition} below. 
  
  One can arrange that the monodromy around  $R$ is also of the desired form, by Lemma~\ref{Lemma:decomposition}  below.
 
  Since $M_\infty$ has all its coefficients nonzero, the system is irreducible (proof in Lemma~\ref{normalization_irreducible}  below).  

 A key point is that a small deformation of a trivial bundle on $\CP^1$ remains trivial.  Hence, for $M_\infty$ close to the identity, our result is still a trivial bundle. Passing to our picture $C$, of a global trivialization, we get our desired connection. 
 
One can normalise the coefficients of the connection as in Lemma~\ref{normalization_indecomposable}, thus ending the proof of the theorem. \end{proof} 

 \begin{lemma}\label{Lemma:decomposition} 
Any invertible matrix $B$ can be written as a product $B=C_1C_2$ of two invertible diagonalizable matrices $C_1$ and $C_2$ with distinct eigenvalues. One of the $C_i$ may be chosen in a Zariski open set.
\end{lemma}
\begin{proof} The map $P: GL(n,\C)\rightarrow GL(n,\C)$ defined by $P(C_1)=C_1^{-1}B$ is holomorphic (even algebraic) and invertible. Let $\mathcal{D}\subset GL(n,\C)$ be the subset of diagonalizable matrices with distinct eigenvalues. It is the complement of an algebraic subset (where the discriminant of the characteristic polynomial vanishes). Then $P(\mathcal{D})\cap \mathcal{D}$ is the complement of an algebraic set in $GL(n,\C)$, hence nonvoid of full measure. \end{proof}

\begin{lemma}\label{normalization_irreducible} Suppose that the monodromy $M_\infty$ has all its coefficients nonzero in the trivialization of the bundle over $\Omega_{k,0}^-$ in which the connection  is given by the formal normal form \eqref{Normal_form_0} (a ``picture B'' trivialization). Then the system is irreducible. \end{lemma}
\begin{proof} The point is then that from the point of view of the singularity at the origin, any decomposition must respect the geometric structures at the origin, for example the flags of growth rates. In the trivialization in question, these flags are the standard ones; this forces an invariant subbundle to be generated by a subset of the basis vectors.  On the other hand, applying the  change of basis to bring these basis vectors to be the first $k$ vectors, which is  a permutation matrix $P$, the transformed $PM_\infty P^{-1}$ still has all its entries non-zero,  and so does not leave the subbundle invariant, as it would have to be block triangular to do so.  
\end{proof}

\begin{lemma}\label{normalization_indecomposable} We consider an indecomposable linear differential system \eqref{nabla-global}, where $A_0$ is diagonal with distinct eigenvalues. Then there exists $n-1$ distinct pairs $(i_\ell, j_\ell)$, with $i_\ell\neq j_\ell$, $\ell=1, \dots, n-1$ and exponents $m_{i_\ell, j_\ell}\in\{1, \dots, k+1\} $, such that the equation is conjugate by means of an invertible diagonal transformation to a unique system with the coefficient of the monomial $x^{m_{i_\ell, j_\ell}}$ of $(B)_{i_\ell, j_\ell}$ normalized to $1$. The only automorphisms of this normalised  unique system are scalars $cI_n$ for some $c\in \C^*$. 
\end{lemma}
\begin{proof} The normalization is done by action of invertible diagonal matrices. In view of the fact that the matrices  $cI_n$ are symmetries of the system, we limit ourselves to diagonal matrices $D= \mathrm{diag}(d_1, \dots, d_n)$  with $d_1=1$. In order to be able to normalize $n-1$ coefficients, we need to show that if $D\neq D'$ are two such matrices such that the conjugates of $x^{k+1}(x-R)y'=B(x)\cdot y$ by $D$ and $D'$ are the same, then the system is decomposable. Let $\{1,\dots, n\}= I\cup I'$, where $d_i=d_{i}'$ if and only if $i\in I$. Then both $I$ and $I'$ are nonvoid. Moreover for any $i\in I$ and $j\in I'$,  $b_{ij}(x)\equiv b_{ji}(x)\equiv0$. Indeed, $(DB(x)D^{-1})_{ij}= d_id_j^{-1} b_{ij}(x)$  and $(D'B(x)(D')^{-1})_{ij}= d_i'(d_j')^{-1} b_{ij}(x)$, yielding $b_{ij}(x)\equiv 0$ and, similarly, $b_{ji}(x)\equiv0$. Hence, the system is decomposable.   \end{proof}

 The proof of Theorem~\ref{thm:eps_0} shows that there is considerable choice for the deformed monodromy $M_\infty$, indeed an open set's worth. We note that it provides a local normal form in the sense that given the choice of Stokes matrices $C_{j}^U$,   $C_{j}^L$, one then has an open set of suitable $M_\infty$ to choose; once one has done this,  the resulting 
$C_{j}^U$,   $C_{j}^L$,  $M_\infty$ then determine the  normal form \eqref{nabla-global} uniquely. One can use the diagonal automorphisms to further normalise the  connection, by normalizing some off diagonal coefficients to one.

\subsection{The special case of an irreducible irregular singularity}

The irreducibility we have is for the system as a whole on $\CP^1$; on the other hand one has generically that the irregular singularity itself is irreducible. Bolibruch showed the following theorem

\begin{theorem} (\cite{Bo} or Theorem 20.4 of \cite{yIsY})\label{thm:Bolibruch}  A linear system of differential equations with a nonresonant locally irreducible singularity of Poincar\'e rank $k$ is locally holomorphically equivalent to a polynomial system $x^{k+1} y'=A(x)\cdot y$, where $A(x)= \sum_{i=0}^k A_i x^i$. \end{theorem}

\begin{corollary} A linear system of differential equations with a nonresonant locally irreducible singularity of Poincar\'e rank $k$ is locally holomorphically equivalent to a unique normalized polynomial system $x^{k+1} y'=A(x)\cdot y$, where $A(x)= \sum_{i=0}^k A_i x^i$, with a normalization as in Lemma~\ref{normalization_indecomposable}. \end{corollary}
\begin{proof} 
We can of course suppose that $A_0$ is diagonal. Moreover, an irreducible bundle with connection is in particular indecomposable. Hence, it is possible to apply a further normalization of $(n-1)$ off-diagonal  coefficients using Lemma~\ref{normalization_indecomposable}.\end{proof}

The condition of irreducibility in Bolibruch's theorem in some sense explains the need for introducing the extra singularity  in our normal form; we need some freedom so that the monodromy around $\infty$ be sufficiently transverse to the diagonal structure at the origin coming from $A(0)$. 
 
\subsection{Deforming in a fixed DS domain; continuity at $\eps = 0$ in a disk $\D_R$.}\label{sec:deforming} 

We would now like to vary $\eps$ in a fixed DS domain $S_s$. Since the DS domain is fixed, we drop the index $s$.  We first work on a disk $\D_R$ in $x$-space. Later, in Section~\ref{sec:Sectoral_domain}
 we will extend to the whole of $\CP^1$.
 
We begin by clarifying our notion of a continuous family of bundles plus connection, as $\eps$ varies. Indeed, if one has a family of bundles given in pictures $A$ or $B$ of subsection \ref{three-pictures} above, one can imagine that there could be difficulties in understanding continuity as one varies $\eps$, since the transition matrices are defined over sets varying with $\eps$. Within the DS domain $S_s$, this poses no problem, as the open sets are varying quite smoothly. But one should be a bit careful as one moves in to  $\eps =0$, or more generally, as one is going to the boundary of the DS domain, at a point where the singularities merge together, for example with two poles becoming a double pole.  

On the other hand, if one thinks of our pair of (bundle, flat connection) as a family of flat continuous connections in a global  (not necessarily holomorphic) trivialization (let us call this, {\bf picture D}), then it is fairly easy to see what a good notion of continuity should be. We first must fix the singularities of the connections; this amounts to choosing conjugacy classes of the singular part of the connection, up to holomorphic gauge transformation. This polar part must vary continuously in $\eps$. 

For our example, in the disk $\D_{R}$ around the origin, we first choose a formal normal form, which must itself be deforming continuously in $\eps$:  
 
\begin{equation}\mathcal{A}_{n,\eps}= \frac{\Lambda_0(\eps) +....\Lambda_k(\eps)x^k}{p_\eps(x)}dx \label{normal_Form}\end{equation}

 Our continuous family $\mathcal{A}_\eps$ will have poles that are holomorphic conjugates of our normal form, and will have in addition  a
 finite piece, living in an appropriate function space, with both $(1,0)$ and $(0,1)$ components:
 \begin{equation}\mathcal{A}_\eps-f(x,\eps)\mathcal{A}_{n,\eps} f(x,\eps)^{-1}= a^{1,0} + a^{0,1},\label{gauge_f}\end{equation}
with $f(x,\eps)\in\mathrm{Mat}(n\times n,\C)$ invertible, holomorphic in $x$ and varying continously in $\eps$, and with $a^{1,0} + a^{0,1}$ lying in  a Sobolev space $W^{1,q}$ and varying continuously in $\eps$, as elements of the Sobolev space; note that this does not imply that the  $a^{1,0}$ and $a^{0,1}$ are themselves continuous. Thus, while the $(1,0)$ (holomorphic) term of $\mathcal{A}_\eps$ has singularities, the $(0,1)$ term $a^{0,1}$ is bounded and continuous, again as a family of elements $a^{0,1}(\eps)$ of $W^{1,q}$, with $1<q<2$.

\begin{remark}Some comments are in order as to the choice of space $W^ {1,q}$, a space of functions in $L^q$ with one $L^q$ derivative. One wants to solve for a gauge transformation, $g_\eps^{-1}\overline{\partial}g_\eps =-a^{0,1} $ which will take us to a holomorphic gauge, and we would like to have the solution be continuous, so that, for example applying the gauge transformation does not change the topology of the bundle; this suggests setting up our problem as finding an inverse image under a map $W^{2,q}\rightarrow W^{1,q}$, since one has the Sobolev embedding theorem $W^{2,q}\subset C^0$, for $q>1$; of course, over the subset where $a^{0,1}$ is smooth, elliptic regularity will tell us that $g_\eps$ is also smooth.\end{remark}

\begin{remark} While the notion of continuity allows for the gauge freedom given by $f$
in \eqref{gauge_f}, in our case, however, this freedom can be  normalized away by preparing $\mathcal{A}_\eps$ so that it has the same polar part as $\mathcal{A}_{n,\eps}$, and we only consider the equation with $f\equiv1$:
 \begin{equation}\mathcal{A}_\eps-\mathcal{A}_{n,\eps} = a^{1,0} + a^{0,1}.\label{gauge}\end{equation}\end{remark}

Now consider a family of connections in picture $A$ or $B$, i.e., given by Stokes data ($A$), or Stokes data shifted to the normal diagonal form ($B$). In picture $B$, explicitly, one first builds a bundle by glueing trivialised bundles on the sectors $\Omega_{j,\eps}^\pm\cap \D_R$ by symmetries  of the normal form coming from the Stokes matrices on intersections;  the connections are constructed by putting the normal form $\mathcal{A}_{n,\eps}$ on the sectors $\Omega_{j,\eps}^\pm\cap \D_R$, and patching together with a partition of unity, obtaining a $C^\infty$ connection on $\D_R$; initially this is on the complement of the singularities, but the asymptotics of the glueing functions are such that the bundles extend to   the singularities. More explicitly:

\begin{definition}\label{def:pair} (Definition of the pair (bundle, connexion $\mathcal{A}_\eps$)) Let us describe the bundle, with the transition functions  coming from the Stokes matrices; this is essentially picture B above.  Let 
 $F_{j,\eps}^\pm$ be the standard fundamental matrix solutions for the connection given by \eqref{normal_Form} on the $\Omega_{j,\eps}^\pm\cap \D_R$, which are analytic continuations of each other, except on $\Omega_{1,\eps}^U$.  Then, as in (\ref {Btriv1}), (\ref {Btriv2}), (\ref {Btriv3}),
the transition functions are of the form
\begin{equation}\begin{cases} \widetilde C_{j,\eps}^L= F_{j,\eps}^+C_{j,\eps}^L(F_{j,\eps}^-)^{-1}= \mathrm{D_{j,\eps}^L} + N_{j,\eps}^L,&x\in \Omega_{j,\eps,s}^L,\\
\widetilde C_{j,\eps}^U= F_{j-1,\eps}^-C_{j,\eps}^U(F_{j,\eps}^+)^{-1} =\mathrm{D_{j,\eps}^U} + N_{j,\eps}^U,&x\in \Omega_{j,\eps}^U, \\
\widetilde C_{j,\sigma(j),\eps}^G= F_{j,\eps}^+C_{j,\sigma(j),\eps}^G(F_{\sigma(j),\eps}^-)^{-1}= \mathrm{Id}, &x\in \Omega_{j,\sigma(j),\eps}^G,\end{cases}\label{gluing_sectors}\end{equation}

Here, the transition functions of \eqref{gluing_sectors}  are all of the form $D_\eps + N_\eps(x)$, with $D_\eps$ diagonal, constant in $x$ and analytic in $\eps$, and $N_\eps$ either strictly upper or lower triangular  and tending to zero at the singularity. Since $D_\eps + N_\eps(x)$ is an automorphism of $\mathcal{A}_{n,\eps}$ near the origin, it satisfies 
$$ (D_\eps + N_\eps(x)) \mathcal{A}_{n,\eps} (D_\eps + N_\eps(x))^{-1} + d N_\eps(x) (D_\eps + N_\eps(x))^{-1} = \mathcal{A}_{n,\eps},$$
(where $d$ is the derivative in the $x$ direction only),  whence 
$$dN_\eps(x) -  [\mathcal{A}_{n,\eps},N_\eps(x)]=0,$$
since $[D_\eps,\mathcal{A}_{n,\eps}]=0$.  This equation involves the formal, Abelian connection, and decouples each entry of the matrix $N$;  in consequence, the entries $n_{i,i'}$, $i<i'$ of  $N_\eps(x)$ (to fix ideas, in the upper triangular case, with ${\rm Re}(t)\rightarrow -\infty$), an upper triangular matrix with zero diagonal terms, satisfy
\begin{equation}n_{i,i'}(x,\eps)  \simeq c_{i,i'}\exp\left((\lambda_i(x_\ell) - \lambda_{i'}(x_\ell))t\right) ,\label{n_i}\end{equation} where 
$$\Lambda_0(\eps)+ \dots + \Lambda_k(\eps)x^k= \mathrm{diag} (\lambda_1(x,\eps), \dots, \lambda_n(x,\eps)).$$
 The asymptotic decay would be   an application of the theorem of Levinson given above (Theorem~ \ref{thm:CD}), but in the simpler scalar case, since the equations for each $n_{i,i'}$ are decoupled.

The important ingredient is that $\mathrm{Re} \left(\lambda_j(x_\ell) - \lambda_{j'}(x_\ell)\right)>0$ for $j<j'$. Thus, $\lambda_j(x_\ell) - \lambda_{j'}(x_\ell)= r_{j,j'}e^{i\phi_{j,j'}}$ with $\phi_{j,j'}\in \left(-\frac{\pi}2, \frac{\pi}2\right)$. Let us suppose that we consider $t$ along a line of slope $\theta\in (-\frac{\pi}2,\frac{\pi}2)$, i.e $t= e^{i\theta}t'$ with $t'\in \R^-$ (since $x_\ell$ is of $\alpha$-type in the upper triangular case). Then 
$$\exp\left((\lambda_j(x_\ell) - \lambda_{j'}(x_\ell))t\right)= \exp\left(r_{j,j'}e^{i(\phi_{j,j'}+\theta)}t'\right),$$
and 
\begin{equation}|n_{j,j'}(x,\eps)|\leq C(\eps) \exp(-\beta |t|), \label{bound_n}\end{equation}
for some $\beta>0$ provided that $|\theta|<\delta$ with $\delta>0$ sufficiently small so that $\phi_{j,j'}+\theta\in (-\frac{\pi}2,\frac{\pi}2)$. It suffices to take $\beta = \min_{j<j'} \left(r_{j,j'} \cos(\phi_{j,j'}+\theta)\right)$.  (The number $\delta$ appearing here was the limit of the slope for the strips in $t$-space in \cite{HLR}.)  A similar estimate holds for the entries of $N_\eps(x)$ in the lower diagonal case when $\mathrm{Re}(t)>0$. This gives transition functions which tend to constant diagonal matrices in the limit, and so the bundle extends to the punctures.

We now can build the connection. On each sector, we have the  normal form connection $\mathcal{A}_{n,\eps}$.  Let $h_\eps$ be a smooth function, which is zero on $\Omega_\alpha\setminus\Omega_\beta$, $1$ on $\Omega_\beta\setminus\Omega_\alpha$, and takes values in the interval $[0,1]$ on the intersection. We then set the connection on $\Omega_\alpha\cup \Omega_\beta$ to be given by the gauge transformation on the overlaps:
\begin{equation}\mathcal{A}_\eps = \begin{cases} (D_\eps + h_\eps(x)N_\eps(x) )\mathcal{A}_{n,\eps} (D_\eps + h_\eps(x)N_\eps(x) )^{-1} &\\ \qquad + d (h_\eps(x)N_\eps(x)) (D_\eps + h_\eps(x)N_\eps(x))^{-1},& \text{on}\:\; \Omega_\alpha\\
 A_{n,\eps}, &\text{on}\:\; \Omega_\beta\setminus \Omega_\alpha.\end{cases}\label{def:connection_A}\end{equation}
Note that $\mathcal{A}_\eps= \mathcal{A}_{n,\eps}$ on $(\Omega_\alpha \setminus\Omega_\beta) \cup(\Omega_\beta \setminus\Omega_\alpha)$, and that the connection is everywhere flat.
Iterating the process for all intersection sectors allows defining $\mathcal{A}_\eps$ globally on $\D_R $, with the caveat that there might be some singularities in the connection at the singular points, a problem we will discuss below. 
\end{definition}

\begin{proposition}\label{prop:continuity} Let the Stokes data $C_{j,\eps}^U, C_{j,\eps}^L$ vary analytically in the intersection of the DS domain $S_s$ with the polydisk $\D_\rho$ in $\eps$-space, with a continuous limit at the boundary of the DS domain, near the points at which some of  the singularities coincide (i.e. $\Delta(\eps)=0$).  Then, choosing the functions $h_\eps$ appropriately, the pair of Definition~\ref{def:pair} defines a continuous family of pairs (bundle, connection) on $\D_R$ in the sense given above, that is as a map of $S_s\cap \D_\rho$ into $W^{1,q}$.\end{proposition}

\begin{proof} 
For the visualization of the sectors in the $t$-coordinate we refer to Figures~\ref{strip_horizontal}, \ref{strip_horizontal_double}, \ref{strip_hor_slant_double}, \ref{slant} and \ref{x_t_space}.

Our Stokes data is defined on a family of bundles over the  disk $\D_R$ in $\C$. One now wants to check that the result, when converted to \lq\lq picture $D$\rq\rq, gives a  continuous family. This is done by passing from trivializations defined on our various open sets to a common trivialization, using smooth  but non analytic functions on the overlaps of the open sets. As noted above, within the DS domain, everything (the Stokes matrices, and the open sets themselves) is varying smoothly, and obtaining continuity poses no problem, and indeed is quite classical.

The problems occur when one approaches the discriminantal locus $\Delta(\eps)= 0$, and indeed, only then in a neighbourhood of the singular points.  Note that $\Delta=0$ is invariant under \eqref{parameter_conic}. Near a regular point of $\{\Delta=0\}$, we can reparameterize $\eps\mapsto (\eta,\eta')$, with $\eta=\Delta$.

 We have  glued all our local definitions together in a single patch as described in Definition~\ref{def:pair}. One then considers $\mathcal{A}_\eps-\mathcal{A}_{n,\eps}$, referring to  \eqref{def:connection_A}. We first want to estimate its $L^q$-norm.  Since $(D_\eps + h_\eps(x) N_\eps(x))^{-1}$ is of the order of $1$, this is tantamount to estimating the norm of $(\mathcal{A}_\eps-\mathcal{A}_{n,\eps})(D_\eps + h_\eps(x) N_\eps(x))$, which is equal to 
\begin{align*} (D_\eps + h_\eps(x) N_\eps(x)) \mathcal{A}_{n,\eps}   &- \mathcal{A}_{n,\eps} (D_\eps + h_\eps(x) N_\eps(x))  - d (h_\eps(x)N_\eps(x))  \\
&= [h_\eps(x) N_\eps(x), \mathcal{A}_{n,\eps}]     -d (h_\eps(x)N_\eps(x)) ,\end{align*}
 and so, given that $d N_\eps(x) + [\mathcal{A}_{n,\eps},N_\eps(x)]=0$, the norm of the quantity 
$$N_\eps (x)d (h_\eps(x)).$$
Thus the behaviour of our chosen $h_\eps$ as we approach the boundary is quite crucial. This is where the $t$-uniformisation of the plane comes into play. Under this uniformisation, the sectors $\Omega_{j,\eps}^\pm$  have a horizontal part which, together with the width of the sector,  becomes large as $\eps$ tends to zero.  We can of course manage that the intersection of two adjacent sectors be a strip $V$ of uniform vertical width $1$.  When we approach a boundary point $\eps'$ of $S_s$ belonging to $\{\Delta=0\}$, the same occurs:  some sectors $\Omega_{j,\eps}^\pm$ (not all, only the ones attached to a multiple point) become large as well as the horizontal part when $\eps\to \eps'$.   Hence, we can use a function $h_\eps$ in the $t$ plane, which just depends on the imaginary part of $t$, and goes smoothly from zero to one from one side of the intersection strip $V$ to the other, so that $dh_\eps$ is supported on $V$. We then have that $dh_\eps$ is of  the order of one, and we want to estimate the norm $ N_\eps(x)$ on the strip $V$ of vertical width one, going to infinity. This is, however, already done; from $d N_\eps(x) + [\mathcal{A}_{n,\eps},N_\eps(x)]=0$, as we noted above, but now passing to the  $t$-parametrization, the entries of $N_\eps$ are of order $\exp(-\alpha|\mathrm{Re}(t)|)$, as noted in \eqref{bound_n}. One then finds, taking into account the change of variable, that the quantity one wants to bound for the $L^q$-norm is then
\begin{equation}C(\eps)\int_V \exp(-q\alpha |t|) |p_\eps(x)|^{2-q}dt\wedge d\overline{t},\end{equation}
which indeed remains bounded, independently of $\eps$, if $1<q<2$, as $|p_\eps(x)|$ is bounded near $t=\infty$.  Similarly, for the $L^q$-norm of the derivative, one has that $\frac{d}{dt}(dh_\eps)$ is also bounded, and so the derivative $\frac{d}{dx}$ is of the order of $p_\eps^{-1}$. The quantity one wants is then
\begin{equation}C(\eps)\int_V \exp(-q\alpha |t|)|p_\eps(x)|^{2-2q} dt\wedge d\overline{t},\label{der}\end{equation} which requires a bit more work to bound. To fix ideas, we can consider the linearized equation for $x(t) $ at a singular (say attractive) point, which is $\dot x = p'_\eps(x_\ell) x$, giving a solution $x(t) = \exp (p'_\eps(x_\ell) t)$, with ${\rm Re}( p'_\eps(x_\ell))<0$, and so in the linear approximation $|p_\eps(x)| \simeq |p'_\eps(x_\ell)| \exp (p'_\eps(x_\ell) t)$, and the integrand one is looking to bound is $ |p'_\eps(x_\ell)|^{2-2q}\cdot \exp((-q\alpha + (2q-2)p'_\eps(x_\ell) )|t|)$. Now, as $\eps$ gets small, so do the derivatives $p'_\eps(x_\ell)$ at the zeroes of $p$, and so for $\eps$ small we get an exponentially decreasing integrand. Likewise, at a multiple zero,  the corresponding approximation gives $x(t) = t^{\gamma}$, for some $\gamma$ and so the integral also converges. This argument, of course, is only heuristic; for example, there is no bound on the  constant $C(\eps)$ as $\eps$ approaches the discriminant locus. Nevertheless, it indicates that some estimate is possible, and indeed, it is provided by the following lemma, which gives a uniform bound in $\eps$ over open sets.

\begin{lemma} Let $x_\ell(\eps)$ be a zero of $p_\eps(x)$. Suppose that there exist $a,c>0$ such that 
$${\rm Re} \left(\frac{p_\eps(x)}{(x-x_\ell(\eps))}\right) >-a, \quad {\mathrm and }\quad |p_\eps(x)|> c|x-x_\ell(\eps)|^{k+1}$$
for $|x-x_\ell(\eps)|<\delta,$ and $\eps\in \D.$
 Let $x(t,\eps)$ satisfy 
$$\dot x(t,\eps) = p_\eps(x(t, \eps)),$$  and let it converge to $x_\ell(\eps)$ as ${\rm Re}(t)\rightarrow +\infty$ in a strip. Then there exist constants $K,K'>0$, with 
$|x(t,\eps)-x_\ell(\eps)|> K \exp(-at)$ and  $p_\eps(x(t, \eps))> K' \exp(-(k+1)at)$ for $t>t(\eps)$.

Similar estimates apply in the case of ${\rm Re}\rightarrow -\infty$.

\end{lemma}
 \begin{proof} The first constraint gives  
 $$\frac{d}{dt}({\rm Re}(\ln(x-x_\ell))) = {\rm Re} \left(\frac{p_\eps(x)}{(x-x_\ell(\eps))}\right)>-a,$$
and so, integrating, ${\rm Re}(\ln(x-x_\ell))>C-at$, and $|x-x_\ell|>K \exp (-at)$, whence $|p_\eps(x(t, \eps))|> K' \exp(-(k+1)at)$.\end{proof}

Returning to the theorem, for $\eps$ sufficiently small one can make the bound $a$ of the lemma sufficiently small for the estimate (\ref{der}) to converge. We then note that, on finite strips, all of our constructions are continuous in $\eps$; this, plus the exponential convergence, gives us a continuous family in $\eps$ in  the function space $W^ {1,q}$.
\end{proof}

From this  continuous family, in $W^ {1,q}$ (picture $D$)  on $S_s$,   one can pass to a holomorphic description, in the neighbourhood of the origin, of the connection in a single holomorphic trivialization. One can proceed as in the paper of Atiyah and Bott \cite{AB}, p.555. We note that starting from picture $D$, the aim is to find a gauge transformation $g_\eps$ solving 
\begin{equation}g_\eps^{-1}\overline{\partial}g_\eps =-a^{0,1} .\label{eq:g}\end{equation}   We know that there
 exists a gauge transformation $u $ to a holomorphic gauge for $\eps=0$. Gauge transforming the whole family with $u$, we can assume that we are deforming  from  a holomorphic trivialization at  $\eps=0$, so that $a^{0,1}|_{\eps=0}= 0$; in short, we are solving, as $\eps$ varies in a small set, for  $a^{0,1}$ in a neighbourhood of the origin in $W^ {1,q}$.   Viewing 
\begin{equation}g_\eps\mapsto g_\eps^{-1}\overline{\partial}g_\eps\label{map:g}\end{equation} 
as a map of Sobolev spaces $W^ {2,q}\rightarrow W^ {1,q}$ on a disk around the origin,  one wants to appeal to the implicit or inverse function theorem on Banach spaces. To do this, we compactify, so that the map on function spaces over $\CP^1$ is Fredholm.  We first extend the bundle trivially over $\CP^1$. We can of course suppose that $\eps$ is sufficiently small so that the singularities all lie within $\{|x|<\frac{R}2\}$. We use a $C^\infty$ function with bounded derivative
$$f(|x|)= \begin{cases} 1,&|x|<\frac{R}2,\\0,&|x|>R.\end{cases}$$ and we extend $\mathcal{A}_\eps$ and $\mathcal{A}_{n,\eps}$ to  $f(|x|)\mathcal{A}_\eps$ and $f(|x|)\mathcal{A}_{n,\eps}$ respectively, on a family of trivial bundles $\{E_\eps\}$ on $\CP^1$. 

The linearization at $g=\rm{Id}$ is the family in $\eps$ of elliptic Dolbeault complexes
$g_\eps\mapsto \overline{\partial}g_\eps,$
mapping the sections of a trivial bundle to the sections of the tensor product of the same trivial bundle with the line bundle of $(0,1)$ forms. Its  kernel is the family of global holomorphic sections $H^0(\CP^1, \mathrm{End}(E_\eps)) = gl(n,\C)$, and its     cokernel  is the first Dolbeault cohomology group $H^1(\CP^1, \mathrm{End}(E_\eps))=0$ (see, e.g. \cite{GH}, chap. 0), and so the Fredholm map \eqref{map:g} is locally surjective near the identity as a map from the Sobolev space $W^ {2,q}$ of sections of $\mathrm{Aut}(E_\eps)$ with two $L^q$ derivatives to the space $W^ {1,q}$ of sections of $(0,1)$-forms with values in $\mathrm{End}(E_\eps)$ and one $L^q$ derivative (here $q>1$). Hence, the map \eqref{map:g} is surjective, and its kernel is given by the constant sections.  Asking for example  that $g_\eps$ be orthogonal to the constants gives, by the inverse function theorem on Banach spaces, (e.g. Lang \cite{Lang}, p. 15) a unique solution $g_\eps$ for \eqref{eq:g} restricted   to a suitably small open set in $\eps$-space. The details follow Atiyah and Bott \cite{AB}.

Transforming $\mathcal{A}_\eps$ with $g_\eps$ we obtain a family of connections $\nabla_\eps=g_\eps\mathcal{A}_\eps g_\eps^{-1}+ d g_\eps \; g_\eps^{-1}$. The $(0,1)$ part  of $\nabla_\eps$, namely 
$(g_\eps a^{0,1}+ \overline\partial g_\eps)g_\eps^{-1} $, vanishes by the construction of $g_\eps$ as solution of \eqref{eq:g}. Since $\nabla_\eps$ has been obtained from a flat connection by gauge transformations, it is flat. Its flatness then ensures that it is holomorphic. 

Hence, we obtain a family of connections  on a disk $\D_{r}$ depending analytically on $\eps\in S_s$ with continuous limit at points of the closure of $S_s$ lying on $\{\Delta=0\}$ (this includes $\eps=0$).  

\subsection{Connections on $\CP^1$ over  a DS domain}\label{sec:Sectoral_domain} The previous section gave us a family of bundles in picture C over the product of a disk $\D_R$ in $\C$ and a closure of $S_s\cap \D_\rho$, with a singular connection in the disk direction. (Though we used an extension to $\CP^1$ in the previous lemma, this is not the extension we will use  here. ) 

Indeed, we proceed as for the bundle at $\eps= 0$ and glue in two Fuchsian singularities; the result will be a deformation of our normal form for $\eps = 0$.
 The  deformation will then also be irreducible, and the underlying bundle will be trivial.  The normalisation extends to these deformations also. Writing out the connection in a global trivialization, we will then define the analytic normal form for our singular rank $k$ systems; we suppose that we have already produced our extension  of the connection to $\CP^1$ at $\eps= 0$, as in  (\ref {thm:eps_0}), and so have a monodromy at infinity $M_\infty$ and  residue at infinity $\widehat A_\infty$:
 
 \bigskip
 
\begin{theorem}\label{thm:sectoral} Suppose  given \begin{itemize}
\item a DS domain $S_s$, 
\item formal invariants given by diagonal matrices $\Lambda_0(\eps), \dots \Lambda_k(\eps)$ depending analytically on $\eps$ in the intersection of a polydisk $\D_\rho$ with the closure of $S_s$,
such that $\Lambda_0(0)$ has distinct eigenvalues satisfying \eqref{order_eigenvalues}, determining a formal normal form \eqref{normal_form_eps},
\item    collections of normalized invertible (Stokes) upper (resp. lower) triangular matrices $C_{j,\eps}^U$, (resp. $C_{j,\eps}^L$), $j=1, \dots, k$, depending analytically on $\eps\in\D_\rho\cap S_s$, with continuous limit at the boundary points, and determining a monodromy
\begin{equation}M(\eps)= (C_{k,\eps}^L)^{-1}\cdot (C_{k,\eps}^U)^{-1}\cdot \dots \cdot(C_{1,\eps}^L)^{-1}\cdot (C_{1,\eps}^U)^{-1},\label{def:M_eps}\end{equation}
\item   a generic matrix  $M_\infty$ in a neighbourhood of the identity with distinct eigenvalues,  representing a conjugacy class  with distinct eigenvalues, with all entries nonzero, close to the identity,
 and such that $M_\infty\cdot M(0)^{-1}$ has distinct eigenvalues,
\item   conjugates \begin{equation} M_\infty(\eps)= \Gamma(\eps)M_\infty \Gamma(\eps)^{-1}\label{Choice_M_infty0}\end{equation} with non-zero entries  for $\eps$ in a small neighbourhood of $0$  for some suitable analytic function $\Gamma:S_s\rightarrow G
(n,\C)$ with continuous limit at the boundary points, and such that $\Gamma(0)=\mathrm{id}$.
\end{itemize}
Then, restricting $\rho$ if necessary,  there exists a family of irreducible rational linear differential systems 
\begin{equation} y' = \left(\frac {A_0(\eps) + A_1(\eps)x +...+ A_k(\eps) x^k}{p_\eps(x)} + \frac{\widehat A(\eps)}{x-R}\right)\cdot y = B(x,\eps)\cdot y, \label{extra-pole}\end{equation}
   depending analytically on $\eps\in S_s$  with continuous limit at $\Delta(\eps)=0$ (of the form \eqref{nabla-global}  for $\eps=0$) , and with
 \begin{itemize}
\item formal normal form at the origin \eqref{normal_form_eps},
defined for $\eps\in S_s\cap \D_\rho$,
\item generalized Stokes matrices   which are given by the matrices $C_{j,\eps}^U$ and $C_{j,\eps}^L$, $j=1, \dots, k$,
\item and monodromies $M_R(\eps)$, $M_\infty(\eps)$ around $\ell_R, \ell_\infty$  such that 
$M_R(\eps)M_\infty(\eps)= M(\eps)$, with $M(\eps)$ defined in \eqref{def:M_eps}. \end{itemize}

The automorphisms of the system are multiples of the identity.  Furthermore, acting by an automorphism of the formal normal form (a diagonal   matrix, constant in $x$), which changes   both the Stokes data and our rational system, it is possible to normalize the coefficients of $n-1$ non diagonal entries of   suitable off-diagonal entries of  $B(x,\eps)$ 
to $1$, thus leading to a unique normalization of \eqref{nabla-global} for each equivalence class of the Stokes data under automorphisms of the formal normal form.
 \end{theorem}
\begin{proof} 
We have already constructed a family of bundles with connection $(E_\eps, \nabla_\eps)$ over $\D_R$, which is a continuous deformation of  $(E_0, \nabla_0)$;  we glue to this, as for $\eps=0$, the bundle on the disk $\D_\infty=\{|x|> R/2\}$ with two Fuchsian singularities at $R, \infty$ with monodromies.  This is done by choosing a continuous family of  bundles over $\D_\infty\times (S_s\cap \D_\rho)$, trivialized at $x= 3R/4$, with a connection with Fuchsian singularities at $x= R,\infty$, monodromy $M(\eps)$  around the circle of radius $3R/4$ defined  in \eqref{def:M_eps},
  and monodromy around infinity given by   ${M}_{\infty}(\eps)$.
  The bundles with connection over $\D_R$ and $\D_\infty$ are then glued in the standard way, starting from the base point.
This gives the desired global bundle with connection, deforming the case $\eps= 0$. Since a small deformation of a trivial bundle is trivial, we can pass to a global trivialization. We then get a connection which has   the form (\ref{extra-pole})  above. As noted, for this connection, after diagonalizing the leading term, we can then normalize to $1$ the same   $(n-1)$ non diagonal terms as in the case $\eps=0$.   \end{proof}

Our Stokes matrices for a connection with Fuchsian singularities depend on the DS domain, as the matrices depend on the way that the singular points are tied to infinity via the separatrices; however, there is one invariant that does not depend on this data, and that is the monodromy representation. The importance of the monodromy representation is that it essentially determines a connection with poles of order one, up to some discrete choices of the polar parts of the connection: basically, the eigenvalues of the monodromy determine the eigenvalues of the residue matrices at the Fuchsian singularities up to integers. 

\begin{proposition} Suppose that two pairs (trivialized bundle, connection with Fuchsian singularities) on $\CP^1$ have, in our context
\begin{itemize}
\item the same singularities, given by the zeroes of $p_\eps(x)$,as well as $R$, $\infty$;
\item the same formal invariants $\Lambda_0(\eps), \dots \Lambda_k(\eps)$;
\item the same conjugacy class of residues at $R$ and  $\infty$, given by   the classes of $\widehat A_R(\eps), \widehat A_\infty(\eps)$, with distinct eigenvalues;
\item fixing the base point, the same monodromy representations, (the same, not simply conjugate), with distinct eigenvalues around $R$, $\infty$.  (Note here that the trivializations  on the fiber above the base point are given by the trivialized bundles, and so the monodromy map is uniquely defined.) \end{itemize}
Then they are isomorphic.
\end{proposition}

We emphasize that our bundles are trivialized, not simply trivial. The proof is given in the course of the proof of Theorem \ref{fund_thm}.

\section{The compatibility condition: glueing DS domains.}\label{sec:compatibility}  

Given our Stokes matrices $ C_{j,\eps,S_s}^U, C_{j,\eps,S_s}^L$, with same limit at $\eps=0$, independently of $s$, we want to realize a corresponding differential equation  (with a fixed formal normal form) over a polydisk in parameter space. We have already done this over our DS domains, and want to sew the results together. This involves three steps: \begin{enumerate} 
\item Glueing over DS domains. If the Stokes matrices are arbitrary on each DS domain, then there is no reason why the realized families over DS domains $S_s$ and $S_{s'}$ should be analytically equivalent one to the other  over $S_s\cap S_{s'}$. A suitable compatibility condition is necessary to ensure that this is the case and so to allow a glueing of the different DS realizations  in a uniform family depending on $\eps\in\Sigma_0$, where $\Sigma_0$ is the complement of $\{\Delta=0\}$;
\item  Showing that this extends to the generic locus of $\Delta=0$, where $\Delta$ is the discriminant of $p_\eps(x)$; this generic locus is the set of $\eps$ for which   $p_\eps(x)$ has one double zero and the remaining roots are simple; 
\item Extending to the rest of the polydisk by appealing to Hartogs' theorem. \end{enumerate} This section is concerned with the first step. 

\bigskip
Recall that  $\Sigma_0$ is covered by  the $C_k$ DS domains:  we define the compatibility condition on the intersection of the DS domains. So far,  given  our Stokes matrices $ C_{j,\eps,S_s}^U, C_{j,\eps,S_s}^L, C_{j,\sigma(j),\eps, S_s}^G$, we have realized our system abstractly as a pair  (bundle, connection) on the Riemann sphere, and showed that we could then realize it also as a singular connection on a globally trivialized bundle.  For this connection, we had some freedom on the choice of the monodromy $M_\infty(\eps)$ around infinity defined in \eqref{Choice_M_infty0} and the monodromy around $R$ is determined accordingly.   

No matter how it is presented, such a bundle over $\CP^1$, equipped with a singular connection on  the complement of the zeroes $x_1(\eps),..,x_{k+1}(\eps)$ of $p_\eps$, and of $x_{k+2} = R, x_{k+3}=\infty$, comes with a monodromy representation
$$\mathcal N_\eps: \pi_1(\CP^1\setminus\{x_1(\eps),..,x_{k+3}\})\rightarrow GL(n,\C),$$
defined up to global conjugation: one chooses a base point, then integrates. 
We note that since $M_\infty(\eps)$ is given in \eqref{Choice_M_infty0} and $M(\eps)$ is given in \eqref{def:M_eps}, one is in essence looking at the monodromy of the restriction of the connection to a disk $\D_{\frac{3R}{4}}$ of radius $\frac{3R}4$ :
$$\mathcal M_\eps: \pi_1(\D_{\frac{3R}{4}}\setminus\{x_1(\eps),..,x_{k+1}(\eps)\})\rightarrow GL(n,\C).$$

We do note, however, that the automorphisms of the bundle on $\D_{\frac{3R}{4}}$   a priori just form a subgroup of the diagonal matrices and not necessarily scalars; it is only when one goes to $\CP^1$  that one is forced to  an indecomposable representation.

Also note that as one changes DS domains, the open sets on which the transition matrices $ C_{j,\eps,S_s}^U, C_{j,\eps,S_s}^L ,C_{j,\sigma(j),\eps, S_s}^G$ were defined change; a bifurcation occurs.   If one computes the monodromy along a loop, one gets an ordered product of inverses of transition matrices  of the open sets that the loop intersects, taken in the inverse order of intersection (see for instance Example~\ref{example_compatibility} below). In a different DS domain, for the same loop, one gets a different product, even though the monodromy representations  must be the same.

\begin{compatibility}\label{Condition_comp} Let us assume given our data of a bundle plus connection, in picture $A$ (or $B$), over our family of  DS domains $S_s $. The Compatiblity Condition is as follows:
\begin{itemize}
\item
We fix a base point $x=\frac34 R$, and choose identifications of the fibers of our bundle over that base point, holomorphically and continuously up to the boundary in $\eps \in S_s\cap S_{s'}$. Then, for each $\eps$ in $S_s\cap S_{s'}$, we ask that the monodromy representations $\mathcal M_\eps, \mathcal M'_\eps$ defined by $ C_{j,\eps,S_s}^U, C_{j,\eps,S_s}^L ,C_{j,\sigma(j),\eps, S_s}^G$ and $ C_{j,\eps,S_{s'}}^U, C_{j,\eps,S_{s'}}^L ,C_{j,\sigma(j),\eps, S_{s'}}^G$ be equivalent, that is conjugate by an invertible  matrix $G_\eps= G_\eps(s,s')$ depending only on $\eps$,  holomorphically,  continuous up to the boundary points of each DS domain, and such that $G_0(s,s')=\mathrm{id}$.
\begin{equation}\mathcal M_{  \eps,S_s}  = G_\eps \mathcal M_{ \eps,S_{s'}} G_\eps^{-1},  \quad \eps\in S_s\cap S_{s'}.\label{cond_compatibility}\end{equation}
  \item The $G_\eps(s,s')$ form a cocycle ($G_\eps(s,s')G_\eps(s',s'') = G_\eps(s,s'')$) and  this cocycle is trivial: there exist invertible matrices $\Gamma_\eps(s)$ depending analytically on $\eps\in S_s$ with continuous limit  at the boundary points of the DS domain such that \begin{equation}G_\eps(s,s') = \Gamma_\eps(s)^{-1} \Gamma_\eps(s'),\label{def:Gamma}\end{equation}
  and $\Gamma_0(s)=\mathrm{id}$.
\item The $S_s$ forming a covering on the \'etale sense, condition also holds when $S_s$ is a ramified DS domain as in Figure~\ref{tubular} and we consider a connected component of its self-intersection (this happens for instance when $k=1$, and also when considering neighbourhoods of regular points of $\{\Delta=0\}$). In that case, one also asks that the matrices be a trivial cocycle in the sense that the matrix  $G_\eps(s,s)$ be equal to the product $\widetilde \Gamma_\eps(s) \widehat\Gamma_\eps(s)^{-1}$, where $\widetilde \Gamma,\widehat\Gamma$ represent  the two different branches of the function $\Gamma_\eps(s)$ as one turns around the divisor. \end{itemize}
\end{compatibility}

\begin{figure}\begin{center}
\includegraphics[width=8cm]{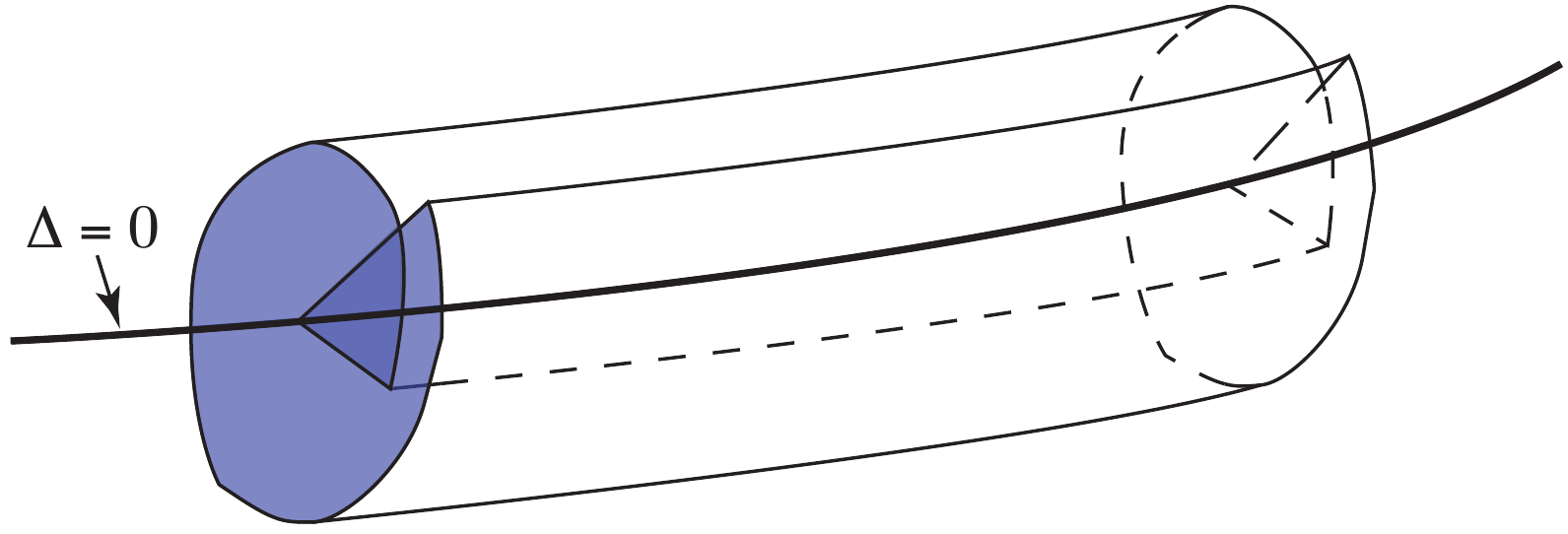}
\caption{A self-intersection of $S_s$ providing a tubular neighborhood of $\Delta=0$.}\label{tubular}
\end{center} \end{figure}

\begin{remark}\label{remark:necessity} The Compatibility Condition~\ref{Condition_comp} is necessary. Indeed, when considering an analytic family \eqref{eq_deployee}, it is obvious that the monodromy representations $\mathcal{M}_{\eps,S_s}$ and $\mathcal{M}_{\eps,S_{s'}}$ are conjugate since they are two representations of the monodromy of the same family \eqref{eq_deployee}. 
For the same reason, the cocycle of the $G_\eps(s,s')$ is trivial. What needs to be proved is the limit properties of $G_\eps(s,s')$ and $\Gamma_\eps(s)$. 
Let us suppose for instance that the given set of loops starts at the base point $x_b= \frac{3R}4$ in $\Omega_{k,\eps,s}^-\cap \Omega_{k,\eps,s'}^-$. It is shown in \cite{HLR} that well chosen  normalizing changes of coordinates $H_{j,\eps,s}^\pm$ (fibered in $x$!) over $\Omega_{j,\eps,s}^\pm$ transforming the normal form \eqref{normal_form_eps} into the family \eqref{eq_deployee} have a limit at points of $\{\Delta=0\}$. We can of course consider that the monodromy representation $\mathcal{M}_{\eps,S_s}$ is the expression  of the monodromy maps in the canonical basis over the base point $x_b$. This canonical basis is transformed into a basis given by the columns of $H_{k,\eps,s}^-(x_b)$ for  \eqref{eq_deployee} and $\mathcal{M}_{\eps,S_s}$ is the monodromy representation in that basis. We want to write the monodromy $\mathcal{M}_{\eps,S_s}$ in a fixed basis, independent of $\eps$ and $S_s$. We choose as a fixed basis the one given by the columns of 
$H_{k,0}^-(x_b)$. Then the monodromy is of the form $\Gamma_\eps(s)\mathcal{M}_{\eps,S_s}(\Gamma_\eps(s))^{-1}$, where the columns of the matrix  $(\Gamma_\eps(s))^{-1}$ are the column vectors of $H_{k,0}^-(x_b)$ written in basis formed by the column vectors of $H_{k,\eps,S_s}^-(x_b)$, namely $H_{k,0}^-(x_b)= H_{k,\eps,s}^-(x_b)(\Gamma_\eps(s))^{-1}$, which yields obviously $\Gamma_\eps(s) = (H_{k,0}^-(x_b))^{-1} H_{k,\eps,s}^-(x_b)$, from which  the result follows.\end{remark}

\begin{remark}\label{identity-at-eps=0}Since we are deforming from the same connection at $\eps=0$, one could have dropped the condition that the matrices $G_\eps(s,s'), \Gamma_\eps(s)$ take value $1$ at $\eps=0$. We would then have allowed the values at $\eps=0$ to lie in the automorphisms of the formal normal form. Then one does not necessarily have the same Stokes data, but equivalent Stokes data at $\eps =0$.
\end{remark}

\begin{lemma}\label{monodromy_equal}
The Compatibility Condition \ref{Condition_comp} allows the monodromy representations to be made equal (instead of just equal up to conjugacy), in a consistent way. Indeed, acting by  $\Gamma_\eps$ in the fibers over our base point, and so setting
$$\widetilde{\mathcal M}_{\eps,S_s}  = \Gamma_\eps(s) \mathcal M_{ \eps,S_{s}} \Gamma_\eps^{-1}(s),$$
one obtains that the transformed $\widetilde G_\eps(s,s')$ are the identity. \end{lemma}

Of course, once one has done this, the monodromy on a loop is no longer an ordered product of Stokes matrices and gate matrices and their inverses, but rather a conjugate by $\Gamma_\eps$ of this product. 

We can extend  the normalised representations $\widetilde{\mathcal M}_{\eps,S_s}$ to the full complement of $k+3$ points in  the Riemann sphere; one has fairly immediately:

 \begin{proposition} Let $S_s$ and $S_{s'}$ be two intersecting DS domains.  The Compatibility Condition~\ref{Condition_comp} implies that 
the full monodromy representations $\mathcal{N}_{\eps,S_s}$ and $\mathcal{N}_{\eps,S_{s'}}$ around the $k+3$ singular points are the same (or their conjugates $\widetilde{\mathcal N}_{\eps,S_s}$ and $\widetilde{\mathcal N}_{\eps,S_{s'}}$ by $\Gamma_\eps(s)$ and $\Gamma_\eps(s')$ respectively),  provided we choose  the monodromy at $\infty$ as in \eqref{Choice_M_infty0} (with $M_\infty$ having all nonzero entries) in  the realization over each DS domain given by Theorem~\ref{thm:sectoral}. 
\end{proposition} 
\begin{proof} In other words, if we conjugate by the $\Gamma_\eps(s), \Gamma_\eps(s')$ (change trivializations at the base point), so that the monodromy representations are the same, we then add in the same monodromy matrix  at infinity. More explicitly, we consider the monodromy $\mathcal{M}_{S_s}(\eps)$ and $\mathcal{M}_{S_{s'}}(\eps)$. By the Compatibility Condition~\ref{Condition_comp} we have that 
$$\widetilde{\mathcal M}_{\eps,S_s} = \Gamma_\eps(s)\mathcal{M}_{S_s}(\eps)(\Gamma_\eps(s))^{-1}= \Gamma_\eps(s')\mathcal{M}_{S_{s'}}(\eps)(\Gamma_\eps(s'))^{-1}=\widetilde{\mathcal M}_{\eps,S_{s'}} .$$
When constructing the realizations, the choice of monodromy we make at infinity in \eqref{Choice_M_infty0} with the function $\Gamma_\eps(s)$ defined in \eqref{def:Gamma} guarantees that the monodromies around infinity in the glued domain $\{|x|> \frac{R}2\}\cup\{\infty\}$ calculated from the same sector $\Omega_k^-$ are the same, indeed constant and equal to $M_\infty$. 
Then equality also  follows for the monodromies $\widetilde{M}_{R, S_s}$ and  $\widetilde{M}_{R, S_{s'}}$ around $x=R$.\end{proof}

 \begin{theorem}\label{fund_thm} Let  $\mathcal{E}_{S_s}$ and $\mathcal{E}_{S_{s'}}$ be two normalized realizations over $\CP^1$ as constructed in Theorem~\ref{thm:sectoral}. If the Compatibility Condition~\ref{Condition_comp} is satisfied, then
$\mathcal{E}_{S_s}$ and $\mathcal{E}_{S_{s'}}$ are equal on $S_s\cap S_{s'}$. 
\end{theorem} 
\begin{proof}
We have seen that we could consider that the monodromy representations of $\mathcal{E}_{S_s}$ and $\mathcal{E}_{S_{s'}}$ are equal. We consider a base point $x_b$ near infinity and loops $\gamma_\ell \in \Pi_1(\CP^1,x_b)$ surrounding $x_\ell$ alone in the positive direction, $\ell=1, \dots, k+3$. We can suppose that the points are numbered so that $\gamma_1^{-1}\dots \gamma_{k+2}^{-1}$ is homotopic to $\gamma_{k+3}$. Let $X_{s}$ (resp. $X_{s'}$) be a fundamental matrix solution whose columns are eigensolutions at $\infty$ of  $\mathcal{E}_{S_s}$ (resp. $\mathcal{E}_{S_{s'}}$).

Let us call $M_{x_\ell}$  the monodromy of $X_s$ and $X_{s'}$ along $\gamma_\ell$.  

One first has that the two realisations are isomorphic on the complement of the singular set. This basically is simply the fact that the induced monodromy on the bundle $\rm{Hom}( \mathcal{E}_{S_s}, \mathcal{E}_{S_{s'}})$ acts trivially on the analytic continuation of the identity map. More explicitly, we consider the map  $y\mapsto P(x)y=X_{s'}X_s^{-1}y$. We need to show that $P(x)$ is well defined. The analytic extension of $X_{s}$ (resp $X_{s'}$) along $\gamma_\ell$ is given by $ X_s M_{s,x_\ell}$ (resp.  $X_{s'} M_{s\,x_\ell}$).
Then $P(x)$ is well defined since
$$\left(X_{s'} M_{s',x_\ell}\right)\left(M_{s, x_\ell}^{-1} X_s^{-1}\right)=X_{s'}X_s^{-1}.$$

The next step is to show  that $P$  can be extended analytically to the singular points. Let us start by showing that $P$ can be extended at $x_{k+3}=\infty$. Our singularity  is non resonant, i.e. no two eigenvalues of the residue matrix differ by an integer, thus ensuring that the monodromy has no multiple eigenvalues and is hence diagonalizable; it is a generic situation for Fuchsian systems, in which the monodromy around the singularity is determined by the eigenvalues of the residue matrix and the system decouples as a sum of one-dimensional systems (\cite{De}).
Thus, if $\mu_1, \dots, \mu_n$ are the  residue  eigenvalues at $x_\infty$, then the columns of $X_s$ (resp. $X_{s'}$) which are eigensolutions are of the form 
$w_{s,j}(x)=x^{-\mu_j} g_{s,j}(x)$ (resp. $w_{s',j}(x)=x^{-\mu_j} g_{s',j}(x)$), where $g_{s,j}$ (resp. $g_{s',j}$) is analytic and nonzero in a neighborhood of $x_\infty$. Moreover, the matrix $L_s$ (resp.  $L_{s'}$) with columns $g_{s,1}(x), \dots, g_{s,n}(x)$ (resp. $g_{s',1}(x), \dots, g_{s',n}(x)$)
has nonzero determinant for $x$ close to $x_\infty$. 
We have that $P(x)w_{s,j}=w_{s',j}$.

Then 
$$\begin{cases}X_s = L_s\,\text{diag} (x^{-\mu_1}, \dots, x^{-\mu_n}),\\ X_{s'}=L_{s'}\,\text{diag} (x^{-\mu_1}, \dots, x^{-\mu_n}).\end{cases}$$ Hence, $P(x) = X_{s'}X_s^{-1} = L_{s'}L_s^{-1}$ is a nice analytic matrix with nonzero determinant in the neighborhood of $x=x_\infty$. 

\medskip

Let us now consider a singular point $x_\ell$ at which the monodromy  $M_{x_\ell}$ is diagonalizable (this is in particular the case for $x_\ell=R$ from our construction), and let $N\in GL(n,\C)$ be such that $N^{-1}M_{x_\ell} N= D$, where $D$ is diagonal; we place ourselves again in the generic situation at which there is no resonance (i.e. no two residue eigenvalues differ by an integer).
  The matrices $X_sN$ and $X_{s'}N$ are still fundamental matrix solutions of $\mathcal{E}_s$ and $\mathcal{E}_{s'}$, and their columns are eigensolutions for the monodromy around $x_\ell$. 
 If $\mu_1, \dots, \mu_n$ are the residue  eigenvalues at $x_\ell$, then the eigensolutions are of the form 
$w_{s,j}=(x-x_\ell)^{\mu_j} g_{s,j}(x)$ (resp. $w_{s',j}=(x-x_\ell)^{\mu_j}  g_{s',j}(x)$ ), where $g_{s,j}$ (resp. $g_{s',j}$) is analytic and nonzero in a neighborhood of $x_\ell$. Moreover, the matrix $Q_s$ (resp.  $Q_{s'}$) with columns $g_{s,1}(x), \dots, g_{s,n}(x)$ (resp. $g_{s',1}(x), \dots, g_{s',n}(x)$)
has nonzero determinant for $x$ close to $x_\ell$.  This comes from the fact that each system is diagonalizable near $x_\ell$ and the eigensolutions for the diagonal form are given by $C(x-x_\ell)^{\mu_j} e_j$, where $C\in\C^*$ and $e_j$ is the $j$-th vector of the standard basis.
Then 
$$\begin{cases}X_s N= Q_s\,\text{diag} ((x-x_\ell)^{\mu_1}, \dots, (x-x_\ell)^{\mu_n}),\\ X_{s'} N=Q_{s'}\,\text{diag} ((x-x_\ell)^{\mu_1}, \dots, (x-x_\ell)^{\mu_n}).\end{cases}$$ Hence, $P(x) = X_{s'}X_s^{-1} = Q_{s'}Q_s^{-1}$ is a nice analytic matrix with nonzero determinant in the neighborhood of $x=x_\ell$. 

\medskip We have built an equivalence depending analytically on $(x,\eps)$ (since it is the case for $X_s$ and $X_{s'}$) on the domain $\CP^1\times S\setminus \{(x,\eps): p_\eps(x)=0, \eps\ {\rm resonant}\}$. From our hypothesis, the set of resonant values of $\eps$ is of codimension $1$ in $\eps$-space. Hence, the set $ \{(x,\eps): p_\eps(x)=0, \eps\ {\rm resonant}\}$ is of codimension $2$, and we can extend the equivalence to it by Hartogs' theorem. 
 \end{proof}

 \subsection{Generic bifurcations and  the compatibility condition}

Now suppose that  $\eps$ lies in a connected component $S_s\cap S_{s'}$ of two DS domains: the bifurcation from $S_s$ to $S_{s'}$ is precisely obtained by changing some angle $\theta$ with which one goes out to infinity; this can switch some singular point(s) $x_\ell$ from $\alpha$-type to $\omega$-type (or the converse), or can change the points of attachment of the separatrices at infinity to the singular points.  We would like, in this subsection, to illustrate how this impacts on the generalised Stokes matrices.

One thus has a bifurcation; in codimension one (the generic bifurcation) there are  basically two types of bifurcations in $\eps$ that force a change in DS domain. Both involve going through a homoclinic connection between two of the separatrices at infinity.  In the process, one gate sector is untied from its endpoints and each end is then tied to a new endpoint. 
\begin{itemize} 
\item \emph{Outer connection}: a zero of $p_\eps(x)$ changes from being $\alpha$-type to $\omega$-type; this occurs at one end of a branch of the skeleton as in Figure~\ref{outer_homo};
\item \emph{Inner connection}: All the zeroes keep the same $\alpha$-type or $\omega$-type as before, but the attachment of the separatrices emerging from infinity to the zeroes of $p_\eps$ changes as in Figure~\ref{inner_homo}.
\end{itemize}
\begin{figure} [ht!]
\begin{center}
\includegraphics[width=9cm]{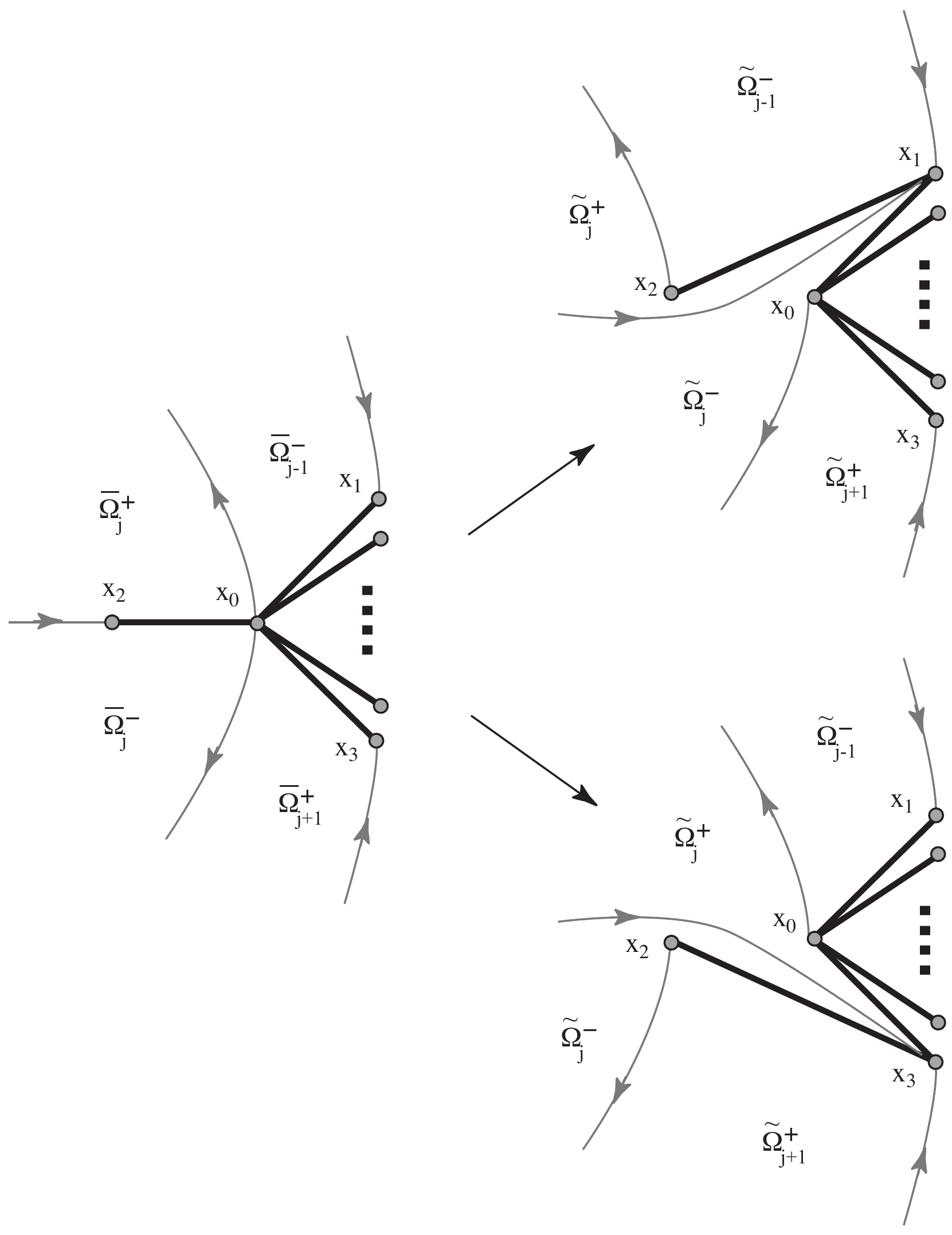}
\caption{ The two ways of passing through an outer homoclinic connection.} \label{outer_homo} \end{center} \end{figure}

\begin{figure} [ht!]
\begin{center}
\includegraphics[width=9cm]{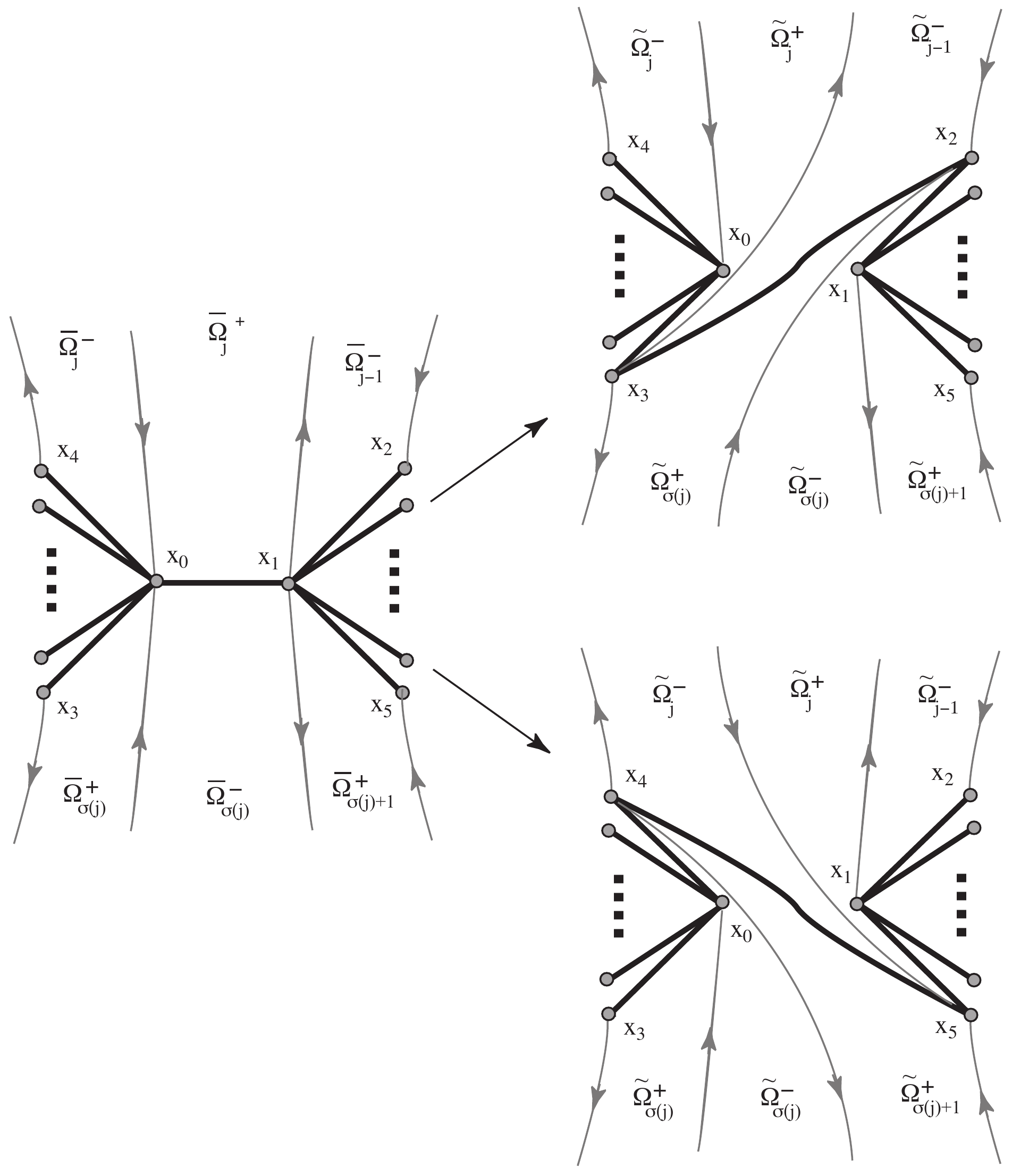}
\caption{ The two ways of passing through an inner homoclinic connection.} \label{inner_homo} \end{center} \end{figure}

\begin{example}\label{example_compatibility} {\bf An example  of the Compatibility Condition.} 
The Compatibility Condition means the following: given any set of loops $\gamma_{x_\ell}$ starting from a base point and going around $x_\ell$ in the positive direction as in Figure~\ref{fig:compatibility}, we can write the \emph{abstract} monodromy $M_{x_\ell, S_s}$  around that loop as a product of matrices $ C_{j,\eps,S_s}^U, C_{j,\eps,S_s}^L ,C_{j,\sigma(j),\eps, S_s}^G$ or their inverses corresponding to the intersection sectors that are crossed around that loop. The set of matrices $\{M_{x_\ell, S_s}\}$ is the monodromy representation $\mathcal M_{\eps,S_s}$ that we compare with $\mathcal M_{\eps,S_{s'}}$. This gives for the example of Figure~\ref{fig:compatibility}:
 \begin{align*}\begin{split}&\begin{cases}
M_{x_1, S_s}=(C_{1, S_s}^L)^{-1}C_{1,\sigma(1), S_s}^G,\\
M_{x_1, S_{s'}}=(C_{1, S_{s'}}^L)^{-1}C_{1,\sigma(1),S_{s'}}^G(C_{2, S_{s'}}^L)^{-1}C_{2,\sigma(2),S_{s'}}^G,\end{cases}\\
&\begin{cases}M_{x_2, S_s}=(C_{1,\sigma(1), S_s}^G)^{-1}(C_{1, S_s}^U)^{-1}(C_{2, S_s}^L)^{-1}C_{2,\sigma(2), S_s}^GC_{1,S_s}^UC_{1,\sigma(1), S_s}^G,\\
M_{x_2, S_{s'}}=(C_{2,\sigma(2),S_{s'}}^G)^{-1}C_{2, S_{s'}}^L(C_{1,\sigma(1),S_{s'}}^G)^{-1}(C_{1, S_{s'}}^U)^{-1}(C_{2, S_{s'}}^L)^{-1}C_{2,\sigma(2),S_{s'}}^G,\end{cases}\\
&\begin{cases}M_{x_3, S_s}=(C_{1,\sigma(1), S_s}^G)^{-1}(C_{1, S_s}^U)^{-1}(C_{2,\sigma(2), S_s}^G)^{-1}(C_{2,S_s}^U)^{-1},\\
M_{x_3, S_{s'}}=(C_{2,\sigma(2),S_{s'}}^G)^{-1}(C_{2,S_{s'}}^U)^{-1}.\\
\end{cases}\end{split}\end{align*} 
 \begin{figure} \begin{center} 
 \includegraphics[width=12cm]{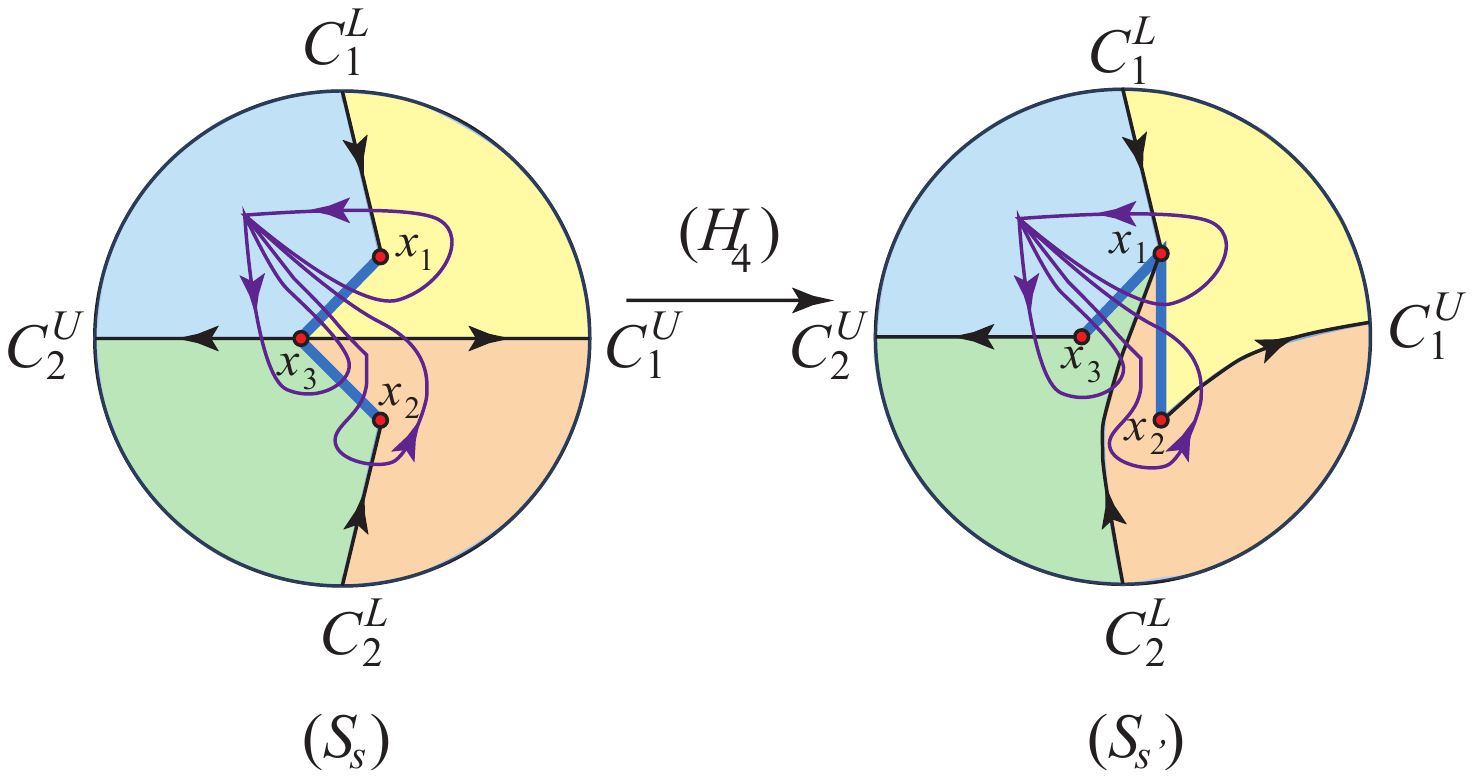}
 \caption{Two sets of sectors when $k=2$ with a transition given by a homoclinic loop through the fourth quadrant and the corresponding monodromy groups.   The permutation $\sigma$ is the identity (resp. the transposition on $1,2$) for $S_s$ (resp. $S_{s'}$). } \label{fig:compatibility} \end{center}\end{figure}\end{example}

\subsection{Extending the realization to the generic locus of $\Delta=0$}\label{sec: auto_intersec}

The generic locus of $\Delta = 0$ correspondes to two zeroes of $p_\eps$ coming together. For arbitrary $k$ this reproduces a parametrized version of the case $k=1$ studied in \cite{cLR2}. 

As one goes to this generic locus of $\Delta = 0$, we note that  the fundamental group of the complement of the zero locus of $p_\eps$ changes: it loses one generator. If one considers the representations defined by $ C_{j,\eps,S_s}^U, C_{j,\eps,S_s}^L ,C_{j,\sigma(j),\eps, S_s}^G$, there is indeed an issue, caused in particular by one gate element, say $C_{j,\sigma(j),\eps, S_s}^G$, as its gate is closing, or rather shrinking to zero. In general, also $C_{j,\sigma(j),\eps, S_s}^G$ has no limit as one goes to $\Delta= 0$. However, as  shown in Sections~\ref{sec:deforming} and \ref{sec:Sectoral_domain}, the passage to ``picture B'', with our choice of normalisation tames the $C_{j,\sigma(j),\eps, S_s}^G$, to $\widetilde C_{j,\sigma(j),\eps, S_s}^G = \rm{Id}$; if in addition 
  the Stokes matrices $ C_{j,\eps,S_s}^U, C_{j,\eps,S_s}^L $  (or their modifications $ \widetilde C_{j,\eps,S_s}^U, \widetilde C_{j,\eps,S_s}^L $) behave continuously when $\eps$ tends to the divisor $\{\Delta(\eps) =0\}$, then we have a continuous family of connections, and our argument about a uniform passage to ``picture C'' goes through, giving a continuous limit.

The same continuity is enforced when one has $S_s$ self-intersecting (one is dealing with an \'etale covering) as one moves around the regular part of the divisor, as in Figure~\ref{tubular} for sectors as in Figure~\ref{auto_intersection}.  We obtain  a uniform differential equation on the tubular neighbourhood, with a limit at the core, as long as the compatibility condition for the monodromy representation is satisfied. 

We thus  have a continuous limit at $\Delta = 0$, which then must be a holomorphic limit at $\Delta = 0$.

On the rest of $\Delta=0$, we are in codimension two, and an appeal to Hartogs' theorem suffices for the extension.
\begin{figure}\begin{center}
\includegraphics[width=8cm]{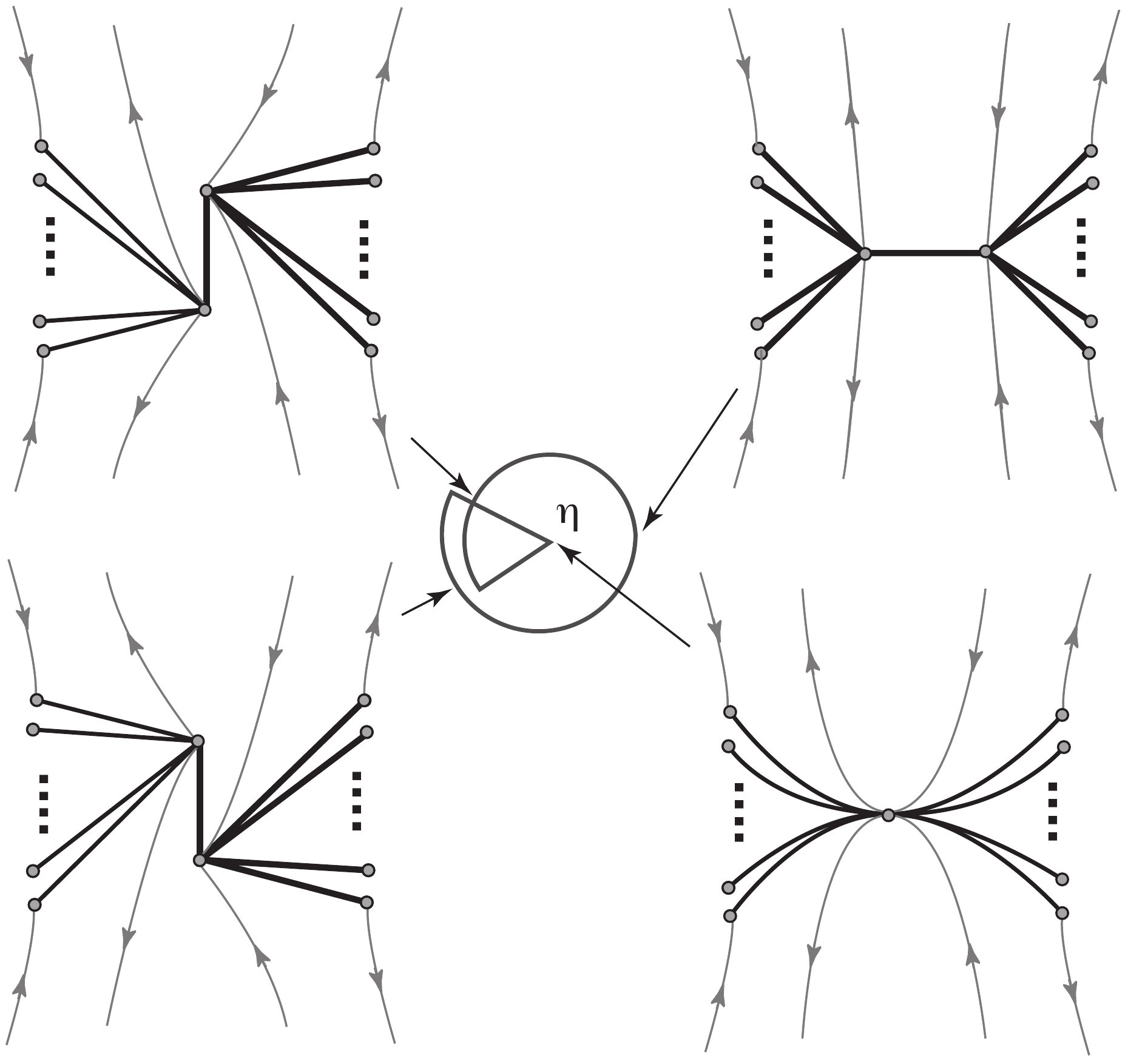}
\caption{The sectors $\Omega_{j,\eps}^\pm$ for a self-intersecting $S_s$ in a tubular neighborhood of $\Delta=0$.}\label{auto_intersection}
\end{center} \end{figure}

\section{The moduli space }

\subsection{The realization theorem}
\begin{theorem}\label{thm: real1} Suppose fixed \begin{itemize}
\item an integer $k\geq1$,
\item formal invariants given by diagonal matrices $\Lambda_0(\eps), \dots, \Lambda_k(\eps)$ depending analytically on $\eps$ in a polydisk $\D_\rho$ such that $\Lambda_0(0)$ has distinct eigenvalues satisfying \eqref{order_eigenvalues}. \end{itemize}

Now, suppose given  for each DS domain $S_s$ of radius $\rho$, collections of normalized invertible (Stokes) upper (resp. lower) triangular matrices $C_{j,\eps, S_s}^U$, (resp. $C_{j,\eps, S_s}^L$), $j=1, \dots, k$, depending analytically on $\eps\in S_s$ with continuous limit at $\eps=0$, independent of $s$, and continuous limit at generic points of $\{\Delta=0\}$. 
Moreover, suppose that the Compatibility Condition~\ref{Condition_comp} is satisfied.
Then, there exists an analytic family of rational linear differential systems \eqref{eq_deployee} for $\eps\in D_\rho$, with formal normal form \eqref{normal_form_eps}, and with Stokes matrices $C_{j,\eps, S_s}^U$ and $C_{j,\eps,S_s}^L$, $j=1, \dots, k$ over $S_s$. A particular analytic family of the form \eqref{extra-pole} exists, with two extra Fuchsian singular points at $R, \infty$, and the family can be uniquely normalised as in Lemma~\ref{normalization_indecomposable}. 

In the particular case where the system is irreducible for $\eps=0$, a second realization exists of the simpler form 
\begin{equation}p_\eps(x)y'= (\Lambda_0(\eps) +B_1(\eps)x+ \dots +B_k(\eps)x^k)\cdot y, \label{polynomial_normal_form}\end{equation}

\end{theorem}
\begin{proof} Let us call $\Sigma_0'\subset \D_\rho$ the locus of in $\eps$-space of the polynomials with either distinct zeroes or one double zero and the rest distinct; it is the union of $\Sigma_0$ (distinct zeroes; defined in \eqref{Sigma_0}) and  the generic points of $\Delta=0$ in parameter space (one double zero); its complement has codimension two. We have built an open covering $\{V_\alpha\}$ of $\Sigma_0'$, and on each open set $V_\alpha$, a realization by a linear system $\mathcal{E}_\alpha$ depending analytically on $\eps\in V_\alpha$. The Compatibility Condition~\ref{Condition_comp} ensures that for  $\eps\in V_\alpha\cap V_\beta$, the systems $\mathcal{E}_\alpha$ and $\mathcal{E}_\beta$ are analytically equivalent, and, once normalised, are the same. 
Hartogs' theorem allows to extend the family of rational  linear differential systems \eqref{eq_deployee} to a uniform family over $\D_\rho$; one uses implicitly the fact that a small deformation of a trivial bundle is trivial, and our normalisations  Lemma~\ref{normalization_indecomposable} for uniqueness. 

The same type of arguments can be used verbatim in the particular case of an irreducible system at $\eps=0$ and \eqref{polynomial_normal_form}, since everything relies on a uniquely normalized realization at $\eps=0$.  Here we use Theorem~\ref{thm:Bolibruch} instead of our normal form.    \end{proof} 

\subsection{The moduli space}

In the literature the  \lq\lq moduli space\rq\rq\ is a universal space, typically finite dimensional, with all families of a given type of object given by mapping into that space, and often (for a fine moduli space) obtaining the family by pulling back.

One could, of course describe the families directly, and in some sense that is what we are doing here. Our families of deformations will be, in effect, given by compatible families of Stokes matrices, and so by maps into the matrix groups. In that sense, we speak of the moduli space, even though what we are describing, is, a priori, an infinite dimensional family of maps $C_{1,\eps, S_s}^L(\eps),\dots,  C_{k,\eps, S_s}^{U}(\eps)$ and  $C_{j,\sigma(j),\eps,S_s}^G(\eps)$.
These maps come with equivalences  arising from actions of the group $\mathcal{D}_n$ of invertible diagonal $n\times n$ matrices, corresponding to varying the trivializations compatible with the flags on each sector, for  each DS domain. Once one has normalized,   this action reduces to one of the form \begin{equation}\begin{cases}
C_{j,\eps, S_s}^{L} &\mapsto\qquad K_s(\eps) {C}_{j,\eps, S_s}^{L}K_s(\eps)^{-1},\\
C_{j,\eps, S_s}^{U} &\mapsto \qquad K_s(\eps) {C}_{j,\eps, S_s}^{U}K_s(\eps)^{-1},\\
C_{j,\sigma(j),\eps,S_s}^G &\mapsto\qquad C_{j,\sigma(j),\eps,S_s}^G,\end{cases}\label{act2}
\end{equation}
for $K_s(\eps)$ a diagonal invertible matrix. 
We introduce the following equivalence relation on the collections of Stokes matrices:

\begin{definition}\label{def:Stokes_equiv} Two collections of (normalized)   Stokes matrices $\{C_{j,\eps, S_s}^{L,U}, C_{j,\sigma(j),\eps,S_s}^G\}_{\eps\in S_s}$ and $ \{\widehat{C}_{j,\eps, S_s}^{L,U},\widehat C_{j,\sigma(j),\eps,S_s}^G\}_{\eps\in S_s}$ on  a give DS domain  $S_s$  are \emph{equivalent} if there exists a family of invertible diagonal matrices $K_s(\eps)$, depending analytically on $\eps\in S_s$ with continuous invertible limit at $\eps=0$ and at generic points of $\{\Delta=0\}$ such that the action \eqref{act2} on the first collection of Stokes matrices gives the second. 
We denote the equivalence class by $\left[\left\{C_{1,\eps, S_s}^L,\dots,  C_{k,\eps, S_s}^{U}, \right\}_{\eps\in S_s}\right]$. \end{definition}

\begin{theorem}\label{thm:moduli_space} The moduli space under analytic equivalence for germs of generic unfoldings of nonresonant linear differential systems with an irregular singularity of finite nonzero Poincar\'e rank at the origin and diagonal matrix $\Lambda_0(0)$ with distinct eigenvalues satisfying \eqref{order_eigenvalues},  is given by the set of tuples
$$\left(k, \Lambda_0(\eps), \dots, \Lambda_k(\eps), \left[\left\{C_{1,\eps, S_s}^L,\dots,  C_{k,\eps, S_s}^{U}\right\}_{\eps\in S_s}\right]_{s=1}^{C_k}\right),$$
where
\begin{itemize}
\item $k\geq1$ is an integer;
\item $\Lambda_0(\eps), \dots, \Lambda_k(\eps)$ are formal invariants given by germs of analytic diagonal matrices; 
\item for each DS domain $S_s$, $\left[\left\{C_{1,\eps, S_s}^L,\dots,  C_{k,\eps, S_s}^{U}\right\}_{\eps\in S_s}\right]$ are collections of equivalence classes of germs of invertible normalized (Stokes) upper (resp. lower) triangular matrices $C_{j,\eps, S_s}^U$, (resp. $C_{j,\eps, S_s}^L$), $j=1, \dots, k$, depending analytically on $\eps\in S_s$ with continuous limit at $\eps=0$ independent of $\eps$, continuous limit  at generic points of $\{\Delta=0\}$, and satisfying the Compatibility Condition~\ref{Condition_comp}.\end{itemize}
\end{theorem}

\section{Ramifications} 

\begin{theorem} We consider a germ of family \eqref{eq_deployee}. If  there exists a permutation matrix $P$ such that the permuted Stokes matrices $PC_{j,\eps,S_s}^{\dag}P^{-1} $ have a common block diagonal structure with blocks of size $n_1, \dots, n_m$, $n_1+\dots + n_m=n$ for all $j=1, \dots, k$, for all $S_s$ and for all $\dag\in \{L,U\}$, then  the germ of family is analytically equivalent to a direct product of germs of $m$ families of linear differential equations on $\C^{n_i}$ for each $i$. \end{theorem}
\begin{proof} The modulus can be decomposed as a direct product of $m$ moduli. We realize $m$ families on $\C^{n_1}, \dots \C^{n_m}$, having as Stokes matrices the corresponding blocks of size $n_i$ and corresponding formal invariants. The direct product of these families has the same modulus as the original family. Hence, it is analytically equivalent to it. \end{proof}

An immediate consequence of the realization theorem is the following normal form theorem. 

\begin{theorem}\label{normal_form_rational} A germ of family of linear differential systems unfolding an  irregular non resonant singular point of Poincar\'e rank $k$  is analytically equivalent for an arbitrary choice of a non-zero  $R$ (independent of $\eps$) to a rational form 
\begin{equation}p_\eps(x)y'= \frac{\Lambda_0(\eps) +B_1(\eps)x+ \dots +B_k(\eps)x^k+ B_{k+1}(\eps)x^{k+1}}{1-x/R}\cdot y, \label{rational_normal_form}\end{equation}
where $\Lambda_0(\eps)$ is diagonal with distinct eigenvalues. A further normalization of $n-1$ non diagonal monomials in the numerator as in Lemma~\ref{normalization_indecomposable} can bring the system to a unique form. \end{theorem}
\begin{proof}  We realize the modulus of such a family in the form \eqref{rational_normal_form} by Theorem~\ref{thm: real1}. Then the initial family and \eqref{rational_normal_form} are analytically equivalent since they have the same modulus.\end{proof}

In the special case when the connection at $\eps=0$ is irreducible, as noted above,  we have given a new proof of the following theorem by Kostov:

\begin{theorem}\label{normal_form_polynomial} (\cite{K}) A germ of family of linear differential systems unfolding an irreducible irregular non resonant singular point of Poincar\'e rank $k$  is analytically equivalent to a polynomial form  as in \eqref{polynomial_normal_form}
\begin{equation}p_\eps(x)y'= (\Lambda_0(\eps) +B_1(\eps)x+ \dots +B_k(\eps)x^k)\cdot y, \label{polynomial_normal_form-2}\end{equation}
with $\Lambda_0(\eps)$ diagonal with distinct eigenvalues. A further normalization of $n-1$ non diagonal monomials as in Lemma~\ref{normalization_indecomposable} can bring the system to a unique polynomial form. \end{theorem}
\begin{proof}  We realize the modulus of such a family in the form \eqref{polynomial_normal_form-2} by Theorem~\ref{thm: real1}. Then the initial family and \eqref{polynomial_normal_form-2} are analytically equivalent since they have the same modulus.\end{proof}

\begin{remark} The number of parameters in Theorem~\ref{normal_form_polynomial} is optimal as remarked by Kostov in \cite{K2}. 
Indeed, for each $\eps$,  the modulus is described by 
\begin{itemize} \item 
$(k+1)n$ formal invariants (eigenvalues at each singular point) for each $\eps$;
\item $2k$ normalized upper or lower  triangular Stokes matrices each with $\frac{n(n-1)}2$ nontrivial upper or lower triangular entries. The equivalence defined above subtracts   $ n-1$ parameters (the conjugation action of the scalar matrices being trivial)  for a total of $(n-1)(kn-1)$ parameters.
\end{itemize} All together, this yields $kn^2+1$ parameters. 
Now, each system \eqref{polynomial_normal_form} is described for each $\eps$ by $kn^2+n$ coefficients. This form is unique up to the action of a diagonal matrix, which allows further scaling of  $n-1$ coefficients.

 The number of parameters in Theorem~\ref{thm: real1} can be explained in a similar way. The realized system $p_\eps(x)(x-R)= A_\eps(x)$ depends on $(k+1)n^2 + n$ parameters, which are the coefficients of $A_\eps(x)$ (remember that $A_\eps(0)$ is diagonal). A diagonal normalization reduces this number to $(k+1)n^2+1$. On the other hand,
 \begin{itemize} \item 
The  Stokes matrices  and the formal monodromy at the zeroes of $p_\eps$ together  give $kn^2+1$ coefficients, quotienting by our equivalences, as above.   
\item Moreover, the $n^2$ coefficients of the monodromy matrix around $\infty$ were chosen, basically arbitrarily in a small set. \end{itemize}This leads to the same total of $(k+1)n^2+1$. Hence, the number of parameters  is   explained by the full generality we introduced in the monodromy at infinity.
 \end{remark}


\begin{thebibliography}{9}

\bibitem{AB} M. Atiyah and R. Bott, \emph{The Yang-Mills equations over Riemann surfaces}, Phil. Trans. R. Soc. Lond. A, \textbf{308} (1982), 523--615.

\bibitem{B} G. Birkhoff, \emph{Singular points of ordinary differential equations}, Transactions of the AMS, \textbf{10} (1909), 436--470.

\bibitem{Bo} A. Bolibruch, \emph{On analytic transformation to Birkhoff standard form}, Trudy Mat. Inst. Steklov \textbf{203} (1994), 33--40, translated in Proc. Steklov Inst. Math. 1995, no. 3 (203), 29--35.

 \bibitem{CL} E. Coddington, N. Levinson \emph{Theory of Ordinary Differential Equations}, International Series in Pure and Applied Mathematics, McGraw-Hill, New York 1955.
 
 \bibitem{De} P. Deligne, \emph{\'Equations Diff\'erentielles \`a Points Singuliers R\'eguliers}, Springer LNM 163 (1970).

\bibitem{DS05}
A. Douady, S. Sentenac, \emph{Champs de vecteurs polynomiaux sur $\mathbb{C}$}, preprint, Paris 2005.

\bibitem{G} A. Glutsyuk, \emph{Stokes operators via limit monodromy of generic perturbation}, J. Dynam. Control Systems, \textbf{5} (1999), no.~1, 101--135.

\bibitem{GH} P. Griffiths and J. Harris, \emph{Principles of Algebraic Geometry} Wiley, New York, (1978)

\bibitem{HLR}
J. Hurtubise, C. Lambert, C. Rousseau, \emph{Complete system of analytic invariants for unfolded differential linear systems with an irregular singularity of Poincar\'e rank $k$}, Moscow Mathematical Journal, \textbf{14} (2014), no.~2, 309--338.

\bibitem{yIsY}
Yu. Ilyashenko, S. Yakovenko, \emph{Lectures on Analytic Differential Equations}, Graduate Studies in Mathematics, 86, American Mathematical Society, Providence, RI, 2008.

\bibitem{yIaK}
Yu. Ilyashenko, A. Khovanskii, \emph{Galois groups, Stokes operators and a theorem of Ramis}, Funct. Analysis and its Applications, 24:4, (1990), 286-296.

\bibitem{K} V.P. Kostov, \emph{Normal unfoldings of non-Fuchsian systems}, C.R. Acad. Sci. Paris, \textbf{318} S\'erie I (1994), 623--628.

\bibitem{K2} V.P. Kostov, \emph{Normal unfoldings of non-Fuchsian systems}, Aspects of complex analysis, differential geometry, mathematical physics and applications (St. Konstantin, 1998), 1--18, World Sci. Publ., River Edge, NJ, 1999. 

\bibitem{cLR2}
C. Lambert, C. Rousseau, \emph{Complete system of analytic invariants for unfolded differential linear systems with an irregular singularity of Poincar\'e rank 1}, Moscow Mathematical Journal, \textbf{12} (2012), Number 1, 77--138.

\bibitem{Lang} S. Lang, \emph{ Fundamentals of Differential Geometry.} Graduate Texts in Mathematics. New York: Springer.(1999)  ISBN 978-0-387-98593-0.

\bibitem{Lev} N. Levinson \emph{The asymptotic nature of solutions of linear systems of  differential equations}, Duke Math. J. 15 (1948), 111--126.

\bibitem{MR} J. Martinet, J.P. Ramis, \emph{Elementary acceleration and multisummability I}, Ann. Inst. H. Poincar\'e Phys. Th\'eor., \textbf{54} (1991), no. 4, 331--401.

\bibitem{Ra} J.P. Ramis, \emph{Confluence et r\'esurgence}, J. Fac. Sci. Univ. Tokyo Sect. IA, Math., \textbf{36} (1989), 703--716. 

\end{thebibliography}
\end{document}